\documentclass[a4paper,reqno,11pt]{amsart}
\usepackage{amsfonts,amsmath,amsthm,amssymb,stmaryrd,cite}
\usepackage{hyperref}
\usepackage{color}
\usepackage{mathrsfs}
\usepackage{esint}
\usepackage{graphics}
\usepackage{float}

\newtheorem{Theorem}{Theorem}[section]
\newtheorem{Lemma}[Theorem]{Lemma}
\newtheorem{Corollary}[Theorem]{Corollary}
\newtheorem{Proposition}{Proposition}[section]
\newtheorem{Remark}{Remark}[section]
\newtheorem{Definition}{Definition}[section]
\numberwithin{equation}{section} \allowdisplaybreaks
\allowdisplaybreaks 
\begin{document}
\author{Yi Peng}
\address{College of Mathematics and Statistics, Chongqing University,
                             Chongqing, 401331,  China.}
\email{20170602018t@cqu.edu.cn}
\author{Huaqiao Wang}
\address{College of Mathematics and Statistics, Chongqing University,
                             Chongqing, 401331,  China.}
\email{wanghuaqiao@cqu.edu.cn}

\title[Non-uniqueness of the Hall-MHD equations]
{Weak solutions to the Hall-MHD equations whose singular sets in time have Hausdorff dimension strictly less than 1}
\thanks{Corresponding author: wanghuaqiao@cqu.edu.cn}
\

\begin{abstract}
In this paper, we focus on the three-dimensional hyper viscous and resistive Hall-MHD equations on the torus, where the viscous and resistive exponent $\alpha\in [\rho, 5/4)$ with a fixed constant $\rho\in (1,5/4)$. We prove the non-uniqueness of a class of weak solutions to the Hall-MHD equations, which have bounded kinetic energy and are smooth in time outside a set whose Hausdorff dimension strictly less than 1. The proof is based on the construction of the non-Leray--Hopf weak solutions via a convex integration scheme.
\end{abstract}

\maketitle
\section{Introduction}
\subsection{Background and Main result}
The incompressible Hall-MHD model on the torus $\mathbb{T}^3=[-\pi,\pi]^3$ can be written as:
\begin{equation}\label{Hall-MHD6}
\begin{cases}
\partial _tu-\Delta u+{\rm div}(u\otimes u-B\otimes B)+\nabla p= 0,{\rm div}u = 0, \\
\partial _tB-\Delta B+{\rm div}(u\otimes B-B\otimes u)+\nabla\times\left((\nabla\times B)\times B\right)=0,
\end{cases}
\end{equation}
where $u,B$ and $p$ represent the velocity field, magnetic field and pressure of the fluid, respectively. In the field of physics, the Hall-MHD system \eqref{Hall-MHD6} has been studied for a long time. The Hall effect plays an important role on magnetic field reconnection in space plasmas \cite{f,hg}, star formation \cite{bt,w} and geo-dynamo \cite{mgm}, etc. Also, there are a great deal of mathematical researches on this model, we refer to \cite{adfl,cdl,CL,CS,CWW,CW1,CW,DM,DM1,DH1,dt,DH} and the references therein.

It is well known that the incompressible MHD system:
\begin{equation}\label{MHD}
\begin{cases}
\partial _tu-\Delta u+{\rm div}(u\otimes u-B\otimes B)+\nabla p= 0,{\rm div}u = 0, \\
\partial _tB-\Delta B+{\rm div}(u\otimes B-B\otimes u)=0,
\end{cases}
\end{equation}
is scaling invariance under the transformation
$$u_\lambda=\lambda u(\lambda x,\lambda^2t),\quad B_\lambda=\lambda B(\lambda x,\lambda^2t),\quad p_\lambda=\lambda^2 p(\lambda x,\lambda^2t),\quad$$
for $\lambda>0$. While for the Hall-MHD system \eqref{Hall-MHD6}, the Hall term destroys the scaling invariance. For this, defining the current function $J=\nabla\times B$, then in view of ${\rm div}B=0$, one has
$$B={\rm curl}^{-1}J:=(-\Delta)^{-1}\nabla\times J.$$
Therefore, the system \eqref{Hall-MHD6} can be written as the following extended Hall-MHD system:
\begin{equation}\label{Hall-MHD7}
\begin{cases}
\partial _tu-\Delta u+{\rm div}(u\otimes u-B\otimes B)+\nabla p= 0,{\rm div}u = 0, \\
\partial _tB-\Delta B-\nabla\times\left((u-J)\times B\right)=0,\\
\partial _tJ-\Delta J-\nabla\times\left(\nabla\times\left((u-J)\times {\rm curl}^{-1}J\right)\right)=0.
\end{cases}
\end{equation}
The extended system \eqref{Hall-MHD7} has a scaling invariance (the same for the incompressible MHD and Navier--Stokes equations.) compared to the classical system \eqref{Hall-MHD6}. We refer to \cite{dt} for more details. When the magnetic field is neglected, equation \eqref{Hall-MHD6} becomes the classical Navier-Stokes equations (NSE for short in the following paragraphs).

Previous literatures (for instance, see \cite{k1,k}) indicate scaling exponent is a very useful criterion for the well-posedness or the ill-posedness. For the NSE, the mixed Lebsgue space $L_t^\gamma L_x^p$ is critical as Ladyzenskaja-Prodi-Serrin (LPS for short) condition: $2/\gamma+d/p=1$ holds, where $d$ being the underlying spatial dimension. The Leray--Hopf solution for the NSE is unique in the critical and sub-critical spaces $L_t^\gamma L_x^p$ ($C_t L_x^p$ if $\gamma=\infty$), $2/\gamma+d/p\leq1$. While for $2/\gamma+d/p>1$, the space $L_t^\gamma L_x^p$ ($C_t L_x^p$ if $\gamma=\infty$) is called super-critical. For more details, we refer to \cite{cl} and the references therein.

There are many research results about non-uniqueness of weak solutions to the mathematical fluid mechanics in recent decades. For the 3D Euler equations, De Lellis-Sz\'ekelyhidi \cite{ds} introduced the convex integration method to construct wild solutions. Buckmaster-De Lellis-Sz\'ekelyhidi Jr-Vicol \cite{bdsv} constructed the dissipative Euler solutions below the Onsager regularity. Isett \cite{i} obtained the non-uniqueness of the 3D incompressible Euler equations in
the class $C_tC_x^\alpha$, $\alpha<1/3$. For more results of the Euler equations, see \cite{b,bdis,bds,ds1,ds2,ds3} and the references therein.
In the breakthrough work, Buckmaster-Vicol \cite{bv} proved that the weak solution of the 3D NSE is not unique in $C_tL^2_x$ by using intermittent convex integration. Later, Cheskidov-Luo \cite{cl} established the non-uniqueness of weak solutions to NSE in the super-critical spaces $L_t^\gamma L_x^\infty$, $1\leq\gamma<2$, which is sharp in view of the classical Ladyzhenskaya-Prodi-Serrin criteria. Moreover, their results do not rely on the parabolic regularization, thus they also obtained the analogous conclusions for the Euler equations. Recently, Cheskidov-Luo \cite{cl1} proved the sharp non-uniqueness for the 2D NSE in $C_tL^p_x$, $1\leq p<2$. Subsequently, Albritton-Br\'{u}e-Colombo \cite{ABC} showed non-uniqueness of Leray--Hpof weak solutions of the NSE with a special force by constructing a background solution. For the stationary NSE, Luo and Cheskidov \cite{CL, Luo} proved the non-uniqueness of weak solutions. At the same time, many researchers consider the hyperdissipative NSE. The hyperdissipative Navier--Stokes equations:
\begin{equation}\label{ns}
\begin{cases}
\partial _tu+(-\Delta)^\alpha u+{\rm div}(u\otimes u)+\nabla p=0,{\rm div}u = 0, \\
u|_{t=0}=u_0,
\end{cases}
\end{equation}
are invariant under the scaling transform $u(x,t)\mapsto\lambda^{2\alpha-1}u(\lambda x,\lambda^{2\alpha}t)$ and the energy norm $L^\infty_tL^2_x$ is invariant under this transformation as $\alpha=5/4$. For the (sub)critical regime $\alpha\geq 5/4$, Lions \cite{l} proved the uniqueness of Leray--Hopf weak solution to the system \eqref{ns} for any divergence free initial data in $L^2$. For $\alpha$ slight below $5/4$, Tao \cite{t} obtained the global well-posedness for \eqref{ns} with logarithmically supercritical dissipation. For the smooth initial data, and $1<\alpha<5/4$ in the system \eqref{ns}, Katz-Pavlovi\'c \cite{kp} proved that the Hausdorff dimension of the singular set at time of the first blow-up is at most $5-4\alpha$.

For the super-critical regime $\alpha<1/5$ and $\alpha<1/3$, the non-uniqueness of Leray solutions to \eqref{ns} is proved by \cite{cdd} and \cite{d}, respectively. For $1<\alpha<5/4$, Luo-Titi \cite{lt} obtained the non-uniqueness of the very weak solution for the system \eqref{ns}. Later, Luo-Qu \cite{lq} proved the non-uniqueness of weak solutions in $C_t^0L^2_x$ to the 2D hyperdissipative NSE \eqref{ns} for $0\leq\alpha<1$. For $1\leq\alpha<5/4$, Buckmaster-Colombo-Vicol \cite{bcv} constructed the non-Leray--Hopf weak solutions in the super-critical spaces $C^0_tL^2_x$, and they also proved such weak solutions are smooth out side a fractal set of singular times with Hausdorff dimension strictly less than 1. Very recently, Li-Qu-Zeng-Zhang \cite{lqzz} proved the sharp non-uniqueness at two endpoints of the Ladyzhenskaya-Prodi-Serrin criteria condition for the 3D hyper-viscous NSE \eqref{ns} when the viscosity $\alpha$ is beyond the Lions exponent $5/4$.

To the best of our knowledge, compared with the Navier--Stokes equations, the non-uniqueness results for the MHD or Hall-MHD equations are less plentiful. Dai \cite{DM1} proved the non-uniqueness of Leray--Hopf weak solutions to the 3D Hall-MHD equations. Recently, Beekie-Buckmaster-Vicol \cite{bbv} constructed weak solutions to the ideal MHD equations based on a Nash-type convex integration scheme with intermittent building blocks. Moreover, they also showed that the magnetic helicity of those kind of weak solutions is not a constant, which implies that Taylor's conjecture is false. While for the fractional power dissipative MHD equations, Li-Zeng-Zhang \cite{lzz} established the non-uniqueness of weak solutions by using convex integration. Later, for the 3D MHD equations, Li-Zeng-Zhang \cite{lzz1} obtained non-uniqueness of weak solutions in $L_t^\gamma L_x^\infty$, $1\leq\gamma<2$, which is sharp in view of the classical Ladyzhenskaya-Prodi-Serrin criteria, and constructed weak solutions to the hyper viscous and resistive MHD beyond the Lions exponent, in the space $L^\gamma_tW^{s,p}_x$, where the exponents $(s,\gamma,p)$ lie in two supercritical regimes. Furthermore, they also provided the observation that the Taylor's conjecture fail in the limit of a sequence of non-Leray--Hopf weak solutions for the hyperdissipative MHD equations, which contrasted to the positive result of weak ideal limit \cite{fl}, namely, the weak limit of Leray--Hopf solutions to the MHD system.
For more results about non-uniqueness of the MHD system, see \cite{df,my,ny} and the references therein.

In this paper, we consider the following hyperdissipative Hall-MHD equations on the torus $\mathbb{T}^3=[-\pi,\pi]^3$:
\begin{equation}\label{Hall-MHD1}
\begin{cases}
\partial _tu+(-\Delta)^\alpha u+{\rm div}(u\otimes u-B\otimes B)+\nabla p= 0,{\rm div}u = 0, \\
\partial _tB+(-\Delta)^\alpha B+{\rm div}(u\otimes B-B\otimes u)+\nabla\times{\rm div}\left( B\otimes B\right)=0,\\
(u,B)|_{t=0}=(u_0,B_0),
\end{cases}
\end{equation}
where $\alpha\in[\rho, \frac{5}{4})$, for some positive constant $\rho\in (1,\frac{5}{4})$. We focus on the following non-Leray--Hopf weak solution:
\begin{Definition}[Weak solution]\label{def1} Given any zero mean initial data $(u_0,B_0)\in L^2$, we say that $(u,B)\in C^0([0,T);L^2(\mathbb{T}^3))$ is a weak solution of the Cauchy problem for the Hall-MHD equations \eqref{Hall-MHD1} if $(u,B)(\cdot,t)$ is weakly divergence free for all $t\in[0,T)$, with zero mean, and
\begin{align*}
\int_{\mathbb{T}^3}u_0\varphi(\cdot,0)dx+\int_{0}^T\int_{\mathbb{T}^3}u\partial _t\varphi-u(-\Delta)^\alpha \varphi+(u\otimes u-B\otimes B):\nabla \varphi dxdt=0,
\end{align*}
\begin{align*}
\int_{\mathbb{T}^3}B_0\varphi(\cdot,0)dx+\int_{0}^T\int_{\mathbb{T}^3}&B\partial _t\varphi-B(-\Delta)^\alpha \varphi+(u\otimes B-B\otimes u):\nabla \varphi \\
&+(B\otimes B):\nabla \nabla\times\varphi dxdt=0,
\end{align*}
for any $\varphi\in C_0^\infty(\mathbb{T}^3\times[0,T))$ such that $\varphi(\cdot,t)$ is divergence free for all $t$.
\end{Definition}

Our main result is stated as follows.
\begin{Theorem}[Main result]\label{main result}
For $T>0$, let $(v_1,b_1), (v_2,b_2)\in C^0([0,T];H^4(\mathbb{T}^3))$ be two strong solutions of the Hall-MHD equations \eqref{Hall-MHD1} on $[0,T]$ with zero mean initial data $(v_1(0,x),b_1(0,x)), (v_2(0,x),b_2(0,x))$. There exists $\beta(\rho,\alpha)>0$ such that the Cauchy problem \eqref{Hall-MHD1} on $[0,T]$ with the initial data $(u,B)|_{t=0}=(v_1,b_1)|_{t=0}$ has a weak solution $(u,B)\in C^0([0,T];H^\beta(\mathbb{T}^3))$ satisfying
$$(u,B)=(v_1,b_1)\quad on \quad [0,\frac{T}{3}] \quad and\quad (u,B)=(v_2,b_2)\quad on \quad [\frac{2T}{3},T].$$
Moreover, for every such solution $(u,B)$, there exists a zero Lebesgue measure set $\Sigma_T\subset[0,T]$ with Hausdorff (in fact box-counting) dimension less than $1-\beta$ such that
$$(u,B)\in C^\infty(((0,T]\backslash\Sigma_T)\times\mathbb{T}^3).$$
\end{Theorem}

\begin{Remark}
Theorem \ref{main result} implies the non-uniqueness of weak solutions to the Hall-MHD system, in the sense of Definition \ref{def1}. In fact, $0$ is not the only weak solution to the Hall-MHD equations with zero initial data.
\end{Remark}
\subsection{New ingredients of our proof}
In this subsection, we will point out the new ingredients of our proof compared to the previous literatures. In the convex integration stage, we introduce the intermittent velocity flows $W_\xi$, $\xi\in\Lambda_0\cup\Lambda_1$, and the intermittent magnetic flows $\bar{W}_\xi$, $\xi\in\Lambda_2\cup\Lambda_3$. The design of $W_\xi$ and $\bar{W}_\xi$ is almost the same, except the temporal oscillation parameter $\mu$ in $W_\xi$, but $\bar{\mu}$ in $\bar{W}_\xi$, see \eqref{wxi} and \eqref{wbarxi} below. Then we define the principle part of velocity increment $w_{q+1}$ and magnetic increment $d_{q+1}$ by
$w_{q+1}^p=\sum_{i\geq0}\sum_{\xi\in\Lambda_{(i)}}a_\xi W_\xi$,
$d_{q+1}^p=\sum_{i\geq0}\sum_{\xi\in\Lambda_{(i)+2}}a_\xi\bar{W}_\xi$, respectively, where $(i)=i\; {\rm mod}\; 2$, $i\in \mathbb{N}$.
Noted that the index sets $\Lambda_{(i)}$ and $\Lambda_{(i)+2}$ are disjoint for any $i\in \mathbb{N}$.

The Reynolds stresses $\mathring{\bar{R}}_q^u$ will be handled in a similar way in \cite{lzz}, and the temporal oscillation parameter $\mu$ in the velocity perturbations will be chose same to that of the NSE in \cite{bcv}. Because of the Hall term, the magnetic stresses will be treated different from the pattern in \cite{lzz}. More precisely, for the case of the MHD equations in \cite{lzz}, the magnetic stresses are dealt by the skew-symmetric matrix $d_{q+1}^p\otimes w_{q+1}^p-w_{q+1}^p\otimes d_{q+1}^p$, which is stemmed from the nonlinear terms in the Maxwell--Faraday equation for the magnetic field. However, in the present paper, the magnetic stresses $\mathring{\bar{R}}_q^B$ will be reduced by the zero frequency part of $d_{q+1}^p\otimes d_{q+1}^p$, which is deduced from the Hall term. Moreover, we make the contribution $d_{q+1}^p\otimes w_{q+1}^p-w_{q+1}^p\otimes d_{q+1}^p$ to $\mathring{R}_{q+1}^B$ disappear by choosing four groups of intermittent jets: $W_\xi$, $\xi\in\Lambda_0\cup\Lambda_1$, $\bar{W}_\xi$, $\xi\in\Lambda_2\cup\Lambda_3$, with disjoint spatial supports, and the fact that $\Lambda_{(i)}\cap\Lambda_{(i)+2}=\emptyset$ in the definitions $w^p_{q+1}$ and $d^p_{q+1}$. As a consequence, the construction of the temporal corrector for the magnetic field $d_{q+1}^t$ (see \eqref{dq+1t}) is totally different from that of the NSE or MHD equations. More precisely, $d_{q+1}^t$ has higher spatial derivative compared to the temporal corrector of the NSE in \cite{bcv} or MHD equations in \cite{lzz}, which has a similar form to $w_{q+1}^t$ defined in \eqref{wq+1t}. Thus, we choose the temporal oscillation parameter $\bar{\mu}$ in the magnetic perturbations much larger than $\mu$ for the purpose of making $d_{q+1}^t$ owns a comparable bounding to $w_{q+1}^t$ (see estimates \eqref{est35} and \eqref{est36}), and see \eqref{para} for the definitions of $\mu$ and $\bar{\mu}$. However, we must impose an upper bound on $\bar{\mu}$ in order to ensure the contribution $\mathcal{R}{\rm curl}^{-1}\partial_t(d_{q+1}^p+d_{q+1}^c)$ to $\mathring{R}_{q+1}^B$ is small, where $d_{q+1}^c$ is the incompressibility corrector for $d_{q+1}^p$.

For the purpose of proving the weak solutions constructed in the present paper is smooth outside a fractal set of singular times with Hausdorff dimension strictly less than 1, we employ the method of gluing in \cite{bcv}. But compared to the Navier-Stokes equations in \cite{bcv}, the Hall-MHD system \eqref{Hall-MHD1} are more complex and require more sophisticated analysis and energy estimates in the gluing stage.

\vspace{-0.6cm}
\section{Outline of the proof}\label{sec2}
The construction of the weak solutions in Theorem \ref{main result} is based on the frame work in \cite{bcv}. Through the iterative scheme, to obtain a sequence of approximate solutions to \eqref{Hall-MHD1}, our main duty is justifying Proposition \ref{iteration prop}.
\subsection{Inductive estimates and main proposition}
For $q\in \mathbb{N}$, we will work with the approximating system:
\begin{equation}\label{Hall-MHD2}
\begin{cases}
\partial _tu_q+(-\Delta)^\alpha u_q+{\rm div}(u_q\otimes u_q-B_q\otimes B_q)+\nabla p_q= {\rm div}\mathring{R}_q^u,\\
\partial _tB_q+(-\Delta)^\alpha B_q+{\rm div}(u_q\otimes B_q-B_q\otimes u_q)+\nabla\times{\rm div}(B_q\otimes B_q)=\nabla\times{\rm div}\mathring{R}_q^B,\\
{\rm div}u_q = 0,\quad {\rm div}B_q=0,
\end{cases}
\end{equation}
where $\mathring{R}_q^u$ and $\mathring{R}_q^B$ are trace-less symmetric matrix.

In order to estimate the approximation solutions $(u_q,B_q,\mathring{R}_q^u,\mathring{R}_q^B)$, inspired by \cite{bcv}, we introduce the following parameters. Fix $b=b(\rho,\alpha)>0$ as a sufficiently large integer, and $\beta=\beta(\rho,\alpha,b)>0$ as a sufficiently small constant with $\beta b\ll 1$. Define
$$\delta_q=\lambda_1^{3\beta}\lambda_q^{-2\beta},$$
where $\lambda_q=a^{b^q}$, and $a\gg1$ is a large constant to be chosen later.

For $q\in \mathbb{N}$, suppose
\begin{subequations}\label{inductive est}
\begin{align}
&\|(\mathring{R}_q^u,\mathring{R}_q^B)\|_{L^1(\mathbb{T}^3)}\leq\lambda_q^{-\varepsilon_R}\delta_{q+1},\label{RestL1}\\
&\|(\mathring{R}_q^u,\mathring{R}_q^B)\|_{H^4(\mathbb{T}^3)}\leq\lambda_q^7,\label{RestH4}\\
&\|(u_q,B_q)\|_{L^2(\mathbb{T}^3)}\leq2\delta_0^{\frac{1}{2}}-\delta_{q}^{\frac{1}{2}},\label{uBestL2}\\
&\|(u_q,B_q)\|_{H^4(\mathbb{T}^3)}\leq\lambda_q^5,\label{uBestH4},
\end{align}
\end{subequations}
for some small constant $\varepsilon_R=\varepsilon_R(b,\beta,\alpha)>0$ which will be chosen later.

For the purpose of estimating the Hausdorff dimension to the singular set of the limiting solution $(u,B)$, for $q\geq1$, choose the parameters
\begin{align}\label{vq}
&\vartheta_q=\lambda_{q-1}^{-\frac{4\rho}{\rho-1}}\delta_q^{\frac{1}{2}},
\end{align}
\begin{align}\label{tq}
&\tau_q=\vartheta_q\lambda_{q-1}^{-\frac{\varepsilon_R}{4}}=\lambda_{q-1}^{-\frac{4\rho}{\rho-1}-\frac{\varepsilon_R}{4}}\delta_q^{\frac{1}{2}}.
\end{align}
Especially, let $\tau_0=T/15$, and $\vartheta_0$ don't need to assign a value. From the definition, we have
\begin{align}
\tau_{q+1}\ll\vartheta_{q+1}\ll\tau_{q}\ll 1,
\end{align}
and
\begin{align}\label{vqtqest}
\vartheta_{q+1}\leq\lambda_q^{-20}\delta_{q+1}^{\frac{1}{2}}, \quad\tau_{q+1}\leq\lambda_q^{-20-\frac{\varepsilon_R}{4}}\delta_{q+1}^{\frac{1}{2}},
\end{align}
since $\rho\in(1,\frac{5}{4})$.

For $q\in\mathbb{N}$, we split the interval $[0,T]$ into a closed good set $\mathscr{G}^{(q)}$ and an open bad set $\mathscr{B}^{(q)}=[0,T]\backslash\mathscr{G}^{(q)}$, which obeys the following properties:
\begin{enumerate}
 \item [(\romannumeral1)] $\mathscr{G}^{(0)}=[0,\frac{T}{3}]\cup[\frac{2T}{3},T]$.
 \item [(\romannumeral2)] $\mathscr{G}^{(q-1)}\subset\mathscr{G}^{(q)}$ for every $q\geq1$.
 \item [(\romannumeral3)] $\mathscr{B}^{(q)}$ is a finite union of disjoint open intervals of length $5\tau_q$.
 \item [(\romannumeral4)] For $q\geq1$, $|\mathscr{B}^{(q)}|\leq|\mathscr{B}^{(q-1)}|10\tau_q/\vartheta_q$.
 \item [(\romannumeral5)] If $t\in \mathscr{G}^{(q')}$ for some $q'<q$, then $(u_q,B_q)(t)=(u_{q'},B_{q'})(t)$.
 \item [(\romannumeral6)] $(\mathring{R}_q^u,\mathring{R}_q^B)(t)=0$ for all $t\in[0,T]$ such that ${\rm dist}(t,\mathscr{G}^{(q)})\leq\tau_q$.
\end{enumerate}

In view of (\romannumeral6) and the parabolic regularization of the Hall-MHD equations (see \eqref{decay1}, for example), $(u_q,B_q)$is a $C^\infty$ smooth exact solution of the Hall-MHD equations on $\mathscr{G}^{(q)}\backslash\{0\}$. Combining (\romannumeral5), we know $(u_q,B_q)=(u,B)$ on $\mathscr{G}^{(q)}\backslash\{0\}$, then the limiting solution $(u,B)$ is $C^\infty $ smooth on $(\mathscr{G}^{(q)}\backslash\{0\})\times\mathbb{T}^3$. Therefore, the singular set of times of $(u,B)$ satisfies $$\Sigma_T\subset\cap_{q\geq0}\mathscr{B}^{(q)}.$$
From (\romannumeral4), for $q\geq1$, we have
$$|\mathscr{B}^{(q)}|\leq|\mathscr{B}^{(0)}|\prod^q\limits_{q'=1}10\tau_q'/\vartheta_q'\leq T10^q\prod^{q-1}\limits_{q'=0}\lambda_{q'}^{-\frac{\varepsilon_R}{4}}=T10^qa^{-\frac{\varepsilon_R}{4}\frac{b^q-1}{b-1}}\leq T10^q\lambda^{-\frac{\varepsilon_R}{8(b-1)}}.$$
By (\romannumeral3), $\mathscr{B}^{(q)}$ is a finite union of disjoint open intervals of length $5\tau_q$, thus the number of such intervals is at most
$$T10^q\lambda^{-\frac{\varepsilon_R}{8(b-1)}}(5\tau_q)^{-1}.$$
At last, by the definition of $\tau_q$ and $\lambda_q$, we infer that
\begin{align*}
{\rm dim_{box}}(\Sigma_T)&\leq\lim_{q\rightarrow\infty}\frac{\log T+q\log 10-\frac{\varepsilon_R}{8(b-1)}\log \lambda_q-\log (5\tau_q)}{-\log (5\tau_q)}\\
&=1-\lim_{q\rightarrow\infty}\frac{\log T+q\log 10-\frac{\varepsilon_R}{8(b-1)}\log \lambda_q}{\log (5\tau_q)}\\
&=1-\frac{\varepsilon_R}{8(b-1)}\frac{b(\rho-1)}{4\rho+(\rho-1)(\beta b+\varepsilon_R/4)}\\
&<1-\frac{\rho-1}{48}\varepsilon_R<1.
\end{align*}

The following iteration proposition yields our main result (Theorem \ref{main result}) and the goal of the present paper is proving this proposition.
\begin{Proposition}[Main Iteration Proposition]\label{iteration prop}
There exists a sufficiently small parameter $\varepsilon_R=\varepsilon_R(\rho,\alpha,b,\beta)\in (0,1)$ and a sufficiently large parameter $a_0=a_0(\rho,\alpha,b,\beta,\varepsilon_R)\geq1$ such that for any $a\geq a_0$ satisfying the technical condition $a^{\frac{25-24\alpha}{24}}\in \mathbb{N}$, the following statement holds: Suppose $(u_q,B_q,\mathring{R}_q^u,\mathring{R}_q^B)$ is a solution to the system \eqref{Hall-MHD2} in $[0,T]\times\mathbb{T}^3$ satisfying the inductive estimates \eqref{inductive est}, and the corresponding good set $\mathscr{G}^{(q)}$ satisfying (\romannumeral1)--(\romannumeral5) listed above. Then there exists $(u_{q+1},B_{q+1},\mathring{R}_{q+1}^u,\mathring{R}_{q+1}^B)$ solving \eqref{Hall-MHD2} and a set $\mathscr{G}^{(q+1)}$ with estimates \eqref{inductive est} and properties (\romannumeral1)--(\romannumeral5) holds with $q$ replaced by $q+1$. Moreover,
\begin{align}\label{cha}
\|(u_{q+1}-u_q,B_{q+1}-B_{q})\|_{L^2}\leq \delta_{q+1}^{\frac{1}{2}}.
\end{align}
\end{Proposition}

\subsection{Gluing stage}
The first step of proving Proposition \ref{iteration prop} is to construct a new glued solution $(\bar{u}_q,\bar{B}_q,\mathring{\bar{R}}_q^u,\mathring{\bar{R}}_q^B)$ to the approximation system \eqref{Hall-MHD2}, which inherits the inductive estimates \eqref{inductive est} of $(u_q,B_q,\mathring{R}_q^u,\mathring{R}_q^B)$. We adapt the way of \cite{bcv}. Define
$$\bar{u}_q=\sum_i\eta_i(t)u_i(x,t),\quad \bar{B}_q=\sum_i\eta_i(t)B_i(x,t),$$
where $\eta_i$ are cutoff functions supported on $[t_i,t_{i+1}+\tau_{q+1}]$ for $t_i=i\vartheta_{q+1}$, and $(u_i,B_i)$ are exact solution of the Hall-MHD system with initial data given by $(u_i,B_i)(t_{i-1})=(u_q,B_q)(t_{i-1})$ (see  below). By the parabolic regularization of the Hall-MHD system, we know that $(u_i,B_i)$ are $C^\infty$ on the support of $\eta_i$, which implies $(\bar{u}_q,\bar{B}_q)$ are $C^\infty$.

Inspired by \cite{bcv}, define the index set
\begin{align}\label{index set}
\mathscr{C}=\left\{i\in{1,\ldots,n_{q+1}}: there\; exists\; t\in [t_{i-1},t_{i+1}+\tau_{q+1}]\; with\; (\mathring{R}_q^u,\mathring{R}_q^B)\neq0\right\},
\end{align}
and the bad set at level $q+1$
\begin{align}\label{bad q+1}
\mathscr{B}^{(q+1)}=\bigcup_{i\in \mathscr{C} or i-1\in \mathscr{C}}(t_i-2\tau_{q+1},t_i+3\tau_{q+1}).
\end{align}

In Section \ref{sec3}, we will prove the following proposition:
\begin{Proposition}\label{glue prop}
There exists a solution $(\bar{u}_q,\bar{B}_q,\mathring{\bar{R}}_q^u,\mathring{\bar{R}}_q^B)$ to system \eqref{Hall-MHD2} such that
\begin{align}
&(\bar{u}_q,\bar{B}_q)=(u_q,B_q)\; on\; \mathbb{T}^3\times\mathscr{G}^{(q)},\label{eq1}\\
&(\mathring{\bar{R}}_q^u,\mathring{\bar{R}}_q^B)=0\; for\;all\;t\in[0,T]\; with \;{\rm dist}(t,\mathscr{G}^{(q+1)})\leq2\tau_{q+1}, \label{bar supp}
\end{align}
and $(\bar{u}_q,\bar{B}_q,\mathring{\bar{R}}_q^u,\mathring{\bar{R}}_q^B)$ satisfies the following estimates:
\begin{align}
&\|(\bar{u}_q,\bar{B}_q)\|_{L^2}\leq2\delta_0^{\frac{1}{2}}-\delta_q^{\frac{1}{2}},\label{barL2}\\
&\|(\bar{u}_q,\bar{B}_q)\|_{H^4}\leq2\lambda_q^5,\label{barH4}\\
&\|(\bar{u}_q-u_q,\bar{B}_q-B_q)\|_{L^2}\leq\lambda_q^{-15}\delta_{q+1}^{\frac{1}{2}}, \label{barcha}\\
&\|\partial_t^MD^N(\bar{u}_q,\bar{B}_q)\|_{H^4}\lesssim\tau_{q+1}^{-M}\vartheta_{q+1}^{-\frac{N}{2\alpha}}\lambda_q^5
\lesssim\tau_{q+1}^{-M-N}\lambda_q^5,\label{bardecay}\\
&\|(\mathring{\bar{R}}_q^u,\mathring{\bar{R}}_q^B)\|_{L^1}\leq\lambda_q^{-\frac{\varepsilon_R}{4}}\delta_{q+1},\label{barL1}\\
&\|\partial_t^MD^N(\mathring{\bar{R}}_q^u,\mathring{\bar{R}}_q^B)\|_{H^4}\lesssim\tau_{q+1}^{-M-1}\vartheta_{q+1}^{-\frac{N}{2\alpha}}\lambda_q^5
\lesssim\tau_{q+1}^{-M-N-1}\lambda_q^5,\label{barRdecay}
\end{align}
for all $N\geq0$ and $M\in \{0,1\}$.
\end{Proposition}

\subsection{Convex integration stage}
In this step, we will construct a new approximating solution $(u_{q+1},B_{q+1},\mathring{R}_{q+1}^u,\mathring{R}_{q+1}^B)$ to \eqref{Hall-MHD2} based on $(\bar{u}_q,\bar{B}_q,\mathring{\bar{R}}_q^u,\mathring{\bar{R}}_q^B)$ designed in gluing step. More precisely, through $(\mathring{\bar{R}}_q^u,\mathring{\bar{R}}_q^B)$, we construct the perturbations
$$w_{q+1}=u_{q+1}-u_q,\quad d_{q+1}=B_{q+1}-B_q$$
by intermittent jets.

The correct for $\mathring{\bar{R}}_q^u$ is similar to that of the Navier--Stokes equation in \cite{bcv}, while the correct for $\mathring{\bar{R}}_q^B$ relies on the Hall term. In the definition of $d_{q+1}^p$, the intermittent jets $\bar{W}_\xi$ will be weighted by $a_\xi$:
$$d_{q+1}^p=\sum_{\xi}a_\xi\bar{W}_\xi,$$
where $a_\xi$ is constructed such that
\begin{align*}
\nabla\times{\rm div}\left(\mathring{\bar{R}}_q^B+d_{q+1}^p\otimes d_{q+1}^p\right)\thicksim\bar{\mu}^{-1}\partial_t\sum_{\xi}\nabla\times\left(a^2_\xi\bar{W}_\xi^2\xi\right)+(high\;frequency\;error),
\end{align*}
for some large parameter $\bar{\mu}$. The high frequency error can be ignored in convex integration schemes, while the first part would be cancelled out by temporal corrector $d_{q+1}^t$:
$$d_{q+1}^t=-\bar{\mu}^{-1}\sum_\xi\nabla\times\left(a^2_\xi|\bar{W}_\xi|^2\xi\right).$$
The divergence corrector $d_{q+1}^c$ is designed to make sure ${\rm div}(d_{q+1}^p+d_{q+1}^c)=0$, hence the perturbation
$$d_{q+1}:=d_{q+1}^p+d_{q+1}^c+d_{q+1}^t$$
is divergence free.

Inspired by \cite{lzz}, define $w_{q+1}^p$ as
$$w_{q+1}^p=\sum_{\xi}a_\xi W_\xi,$$
where $a_\xi$ is constructed such that
\begin{align*}
{\rm div}&\left(\mathring{\bar{R}}_q^u+w_{q+1}^p\otimes w_{q+1}^p-d_{q+1}^p\otimes d_{q+1}^p\right)\\
&\thicksim\partial_t\left(\mu^{-1}\sum_\xi\mathbb{P}_{H}\mathbb{P}_{\neq0}\left(a^2_\xi |W_\xi|^2\xi\right)-\bar{\mu}^{-1}\sum_{\xi}\mathbb{P}_{H}\mathbb{P}_{\neq0}\left(a^2_\xi|\bar{W}_\xi|^2\xi\right)\right)\\
&+(pressure\; gradient)+(high\;frequency\;error),
\end{align*}
for some large parameter $\mu$. Thus the temporal corrector $w_{q+1}^t$ is defined as:
$$w_{q+1}^t=-\mu^{-1}\sum_\xi\mathbb{P}_{H}\mathbb{P}_{\neq0}\left(a^2_\xi |W_\xi|^2\xi\right)+\bar{\mu}^{-1}\sum_{\xi}\mathbb{P}_{H}\mathbb{P}_{\neq0}\left(a^2_\xi|\bar{W}_\xi|^2\xi\right),$$
where $\mathbb{P}_{H}$ is the Helmholtz projection, and $\mathbb{P}_{\neq0}$ is the projection onto the functions with zero mean. In other words, $\mathbb{P}_{H}f=f-\nabla(\Delta{\rm div}f)$ and $\mathbb{P}_{\neq0}f=f-\fint_{\mathbb{T}^3}f$.
The divergence corrector $w_{q+1}^c$ is designed to make sure ${\rm div}(w_{q+1}^p+w_{q+1}^c)=0$, hence the perturbation
$$w_{q+1}:=w_{q+1}^p+w_{q+1}^c+w_{q+1}^t$$
is divergence free.

The intermittent jets will be designed to have disjoint supports through shifts. To see such shifts exist, note that the intermittent jets are confined to a cylinder of diameter $\thicksim \lambda_{q+1}^{-1}$ and length $\thicksim \ell_{\parallel}\ell_{\perp}^{-1}\lambda_{q+1}^{-1}$. Thus we have
$$w_{q+1}^p\otimes d_{q+1}^p-d_{q+1}^p\otimes w_{q+1}^p=\sum_{\xi\neq\xi'}a_\xi a_{\xi'}W_\xi\otimes\bar{W}_{\xi'}-\sum_{\xi\neq\xi'}a_\xi a_{\xi'}\bar{W}_{\xi'}\otimes W_\xi=0.$$

Similar to \cite{bcv}, we define parameters:
\begin{align}\label{para}
\ell_{\perp}=\lambda_{q+1}^{-\frac{20\alpha-1}{24}},\quad\ell_{\|}=\lambda_{q+1}^{-\frac{20\alpha-13}{12}},\quad\mu=\lambda_{q+1}^{2\alpha-1}\ell_{\|}\ell_{\perp}^{-1},\quad
\bar{\mu}=\lambda_{q+1}^{\frac{5}{12}(5\alpha-\frac{1}{4})}\ell_{\|}\ell_{\perp}^{-1}.
\end{align}
For the purpose of bounding $d_{q+1}^t$, $\bar{\mu}$ is chosen much larger than $\mu$. For technical reasons, we assume $a^{\frac{25-24\alpha}{24}}\in \mathbb{N}$, therefore we have $\lambda_{q+1}\ell_{\perp}\in\mathbb{N}$.

\subsection{Proof of Theorem \ref{main result}}
Recall Theorem \ref{main result}, let $(v_1,b_1), (v_2,b_2)\in C^0([0,T];H^4(\mathbb{T}^3))$ be two strong solutions of the Hall-MHD equations \eqref{Hall-MHD1} on $[0,T]$, with zero mean initial data. Let $\eta(t):[0,T]\rightarrow[0,1]$ be a smooth cut-off function such that $\eta=1$ on $[0,\frac{2T}{5}]$ and $\eta=0$ on $[\frac{3T}{5},T]$. Define
\begin{align}
&u_0(x,t)=\eta(t)v_1(x,t)+(1-\eta(t))v_2(x,t),\notag\\
&B_0(x,t)=\eta(t)b_1(x,t)+(1-\eta(t))b_2(x,t),\notag\\
&\mathring{R}_0^u=\partial_t\eta\mathcal{R}(v_1-v_2)+\eta(\eta-1)\left((v_1-v_2)\mathring{\otimes}(v_1-v_2)-(b_1-b_2)\mathring{\otimes}(b_1-b_2)\right),\label{Ru0}\\
&\mathring{R}_0^B=\partial_t\eta\mathcal{R}{\rm curl}^{-1}(b_1-b_2)\notag\\
&\quad\quad\quad+\eta(\eta-1)\mathcal{R}{\rm curl}^{-1}{\rm div}\left((v_1-v_2)\otimes(b_1-b_2)-(v_1-v_2)\otimes(b_1-b_2)\right)\notag\\
&\quad\quad\quad+\eta(\eta-1)(b_1-b_2)\mathring{\otimes}(b_1-b_2),\label{Rb0}
\end{align}
where $a\mathring{\otimes}b$ denotes the traceless part of $a\otimes b$, ${\rm curl}^{-1}$ is an operator on divergence free field defined by
$${\rm curl}^{-1}f=(-\Delta)^{-1}\nabla\times f,\quad {\rm for}\quad {\rm div}f=0,$$
and $\mathcal{R}$ is the inverse divergence operator, defined by
$$(\mathcal{R}u)^{k\ell}=(\partial_k\Delta^{-1}u^\ell+\partial_\ell\Delta^{-1}u^k)-\frac{1}{2}(\delta_{k\ell}+\partial_k\partial_\ell\Delta^{-1}){\rm div}\Delta^{-1}u,$$
for $\int_{\mathbb{T}^3}udx=0$. Note that ${\rm div}{\rm div}(S^{ij})=0$ whenever $S^{ij}$ is a shew-symmetric martix. The operator $\mathcal{R}$ returns symmetric and traceless matrices and one has ${\rm div}(\mathcal{R}u)=u$. We abuse the notation $\mathcal{R}u:=\mathcal{R}(u-\int_{\mathbb{T}^3}udx)$ for vector field $u$ with $\int_{\mathbb{T}^3}udx\neq0$. Moreover, $|\nabla|\mathcal{R}$ is Calderon--Zygmund operator, thus it is bounded in $L^p$ for $1<p<\infty$. See \cite{bcv} for more details.

Noted that $(u_0,B_0,\mathring{R}_0^u,\mathring{R}_0^B)$ solves the system \eqref{Hall-MHD2}. Choose $a\geq a_0$ large enough such that $(u_0,B_0,\mathring{R}_0^u,\mathring{R}_0^B)$ obeys \eqref{RestL1}--\eqref{uBestH4}. From the definition of $(\mathring{R}_0^u,\mathring{R}_0^B)$ in \eqref{Ru0} and \eqref{Rb0}, we know that $supp(\mathring{R}_0^u,\mathring{R}_0^B)\subset[\frac{2T}{5},\frac{3T}{5}]\times\mathbb{T}^3$. Thus (\romannumeral6) in Section \ref{sec2} holds for $\tau_0=\frac{T}{15}$ and $\mathscr{G}^{(0)}=[0,\frac{T}{3}]\cup[\frac{2T}{3},T]$. By using \eqref{cha} in Proposition \ref{iteration prop}, \eqref{uBestL2}, \eqref{uBestH4} and the interpolation, we have
\begin{align*}
\sum_{q=0}^{\infty}\|u_{q+1}-u_q\|_{\dot{H}^{\beta'}}&\lesssim\sum_{q=0}^{\infty}\|u_{q+1}-u_q\|_{L^2}^{1-\beta'/4}\left(\|u_{q+1}\|_{H^4}+\|u_q\|_{H^4}\right)^{\beta'/4}\\
&\lesssim \lambda_1^{\frac{3\beta}{2}(1-\frac{\beta'}{4})}\sum_{q=0}^{\infty}\lambda_{q+1}^{\frac{\beta'(\beta+5)}{4}-\beta}\lesssim1,
\end{align*}
for $0\leq\beta'<\frac{\beta}{\beta+5}$. Similarly, we have
$$\sum_{q=0}^{\infty}\|B_{q+1}-B_q\|_{\dot{H}^{\beta'}}\lesssim1.$$
Thus, define
$$(u,B)=\lim_{q\rightarrow\infty}(u_q,B_q)\;{\rm in}\; H^{\beta'}.$$
Since $\|(\mathring{R}_q^u,\mathring{R}_q^B)\|_{L^1}\rightarrow 0$ as $q\rightarrow\infty$ and $(u_q,B_q)\rightarrow(u,B)$ in $L^\infty_tL^{2+\beta''}$ for some $\beta''>0$, $(u,B)$ is a weak solution of Hall-MHD system. Furthermore, in view of (\romannumeral1), (\romannumeral5) in Section \ref{sec2} and the definition of $(u_0,B_0)$, we have
$$(u,B)=(v_1,b_1)\quad {\rm on} \quad [0,\frac{T}{3}] \quad and\quad (u,B)=(v_2,b_2)\quad {\rm on} \quad [\frac{2T}{3},T].$$

\section{Gluing step}\label{sec3}
\subsection{Local in time estimates}
\begin{Proposition}\label{local solu}
Let $(u,B)\mid_{t=t_0}=(u_0,B_0)\in H^4(\mathbb{T}^3)$ with zero mean be the initial data of system \eqref{Hall-MHD1}. Then there exists a small constant $c$ such that for any $t_1$ satisfying
\begin{align}\label{T1}
0<t_1-t_0\leq c\left(1+\|(u_0,B_0)\|_{L^2}\right)^{-\frac{3}{4}}\left\|(u_0,B_0)\right\|_{H^4}^{-\frac{5}{4}},
\end{align}
there exists a unique strong solution $(u,B)$ to \eqref{Hall-MHD1} on $[t_0,t_1)$ with
\begin{align}
\sup_{t\in[t_0,t_1]}\|(u,B)(t)\|_{L^2}\leq\|(u_0,B_0)\|_{L^2},\label{L2}\\
\sup_{t\in[t_0,t_1]}\|(u,B)(t)\|_{H^4}\leq2\|(u_0,B_0)\|_{H^4}.\label{H4}
\end{align}
Moreover, if
\begin{align}\label{T2}
0<t_1-t_0\leq c\left(\left(1+\|(u_0,B_0)\|_{L^2}\right)^{-\frac{5}{8}}\left\|(u_0,B_0)\right\|_{H^4}^{-\frac{3}{8}}\right)^{\frac{\alpha}{\alpha-1}},
\end{align}
we have
\begin{align}\label{decay1}
\sup_{t\in(t_0,t_1]}|t-t_0|^{M+\frac{N}{2\alpha}}\|\partial_t^MD^N(u,B)(t)\|_{H^4}\lesssim\|(u_0,B_0)\|_{H^4},
\end{align}
for any $N\geq0$ and $M\in\{0,1\}$.
\end{Proposition}
\begin{proof}
Without loss of generality, we prove this proposition with $t_0=0$. The $L^2$ estimate could be obtained by the conservation law:
\begin{align*}
\frac{1}{2}\frac{d}{dt}\|(u,B)\|^2_{L^2}+\|(u,B)\|^2_{\dot{H}^\alpha}=0.
\end{align*}
By using the commutator estimates and the Gagliardo--Nirenberg--Sobolev inequality, for $|\gamma|=4$, we have
\begin{align*}
\frac{1}{2}&\frac{d}{dt}\|\partial^\gamma(u,B)\|^2_{L^2}+\|\partial^\gamma(u,B)\|^2_{\dot{H}^{\alpha}}\\
&=\left([\partial^\gamma,B\cdot\nabla]B,\partial^\gamma u\right)-\left([\partial^\gamma,u\cdot\nabla]u,\partial^\gamma u\right)+\left([\partial^\gamma,B\cdot\nabla]u,\partial^\gamma B\right)-\left([\partial^\gamma,u\cdot\nabla]B,\partial^\gamma B\right)\\
&\quad+\left([\partial^\gamma,B\times](\nabla\times B),\partial^\gamma( \nabla\times B)\right)\\
&\lesssim \|u\|_{\dot{H}^4}\left(\|\nabla B\|_{L^\infty}\|B\|_{\dot{H}^4}+\|\nabla u\|_{L^\infty}\|u\|_{\dot{H}^4}\right)\\
&\quad+\|B\|_{\dot{H}^4}\left(\|\nabla B\|_{L^\infty}\|u\|_{\dot{H}^4}+\|\nabla u\|_{L^\infty}\|B\|_{\dot{H}^4}\right)
+\|B\|_{\dot{H}^5}\|\nabla B\|_{L^\infty}\|B\|_{\dot{H}^4}\\
&\leq\frac{1}{2}\|B\|_{\dot{H}^5}^2+C\|B\|_{L^2}^{\frac{3}{4}}\|B\|_{\dot{H}^4}^{\frac{13}{4}}+C\|(u,B)\|_{L^2}^{\frac{3}{8}}\|(u,B)\|_{\dot{H}^4}^{\frac{21}{8}}.
\end{align*}
Thus taking summation of $\gamma$, by \eqref{L2} we have
\begin{align}\label{est1}
&\|(u,B)(t)\|_{\dot{H}^4}^2+\int_0^t\|(u,B)(s)\|^2_{\dot{H}^{4+\alpha}}ds\notag\\
&\leq\|(u_0,B_0)(t)\|_{\dot{H}^4}^2+C\int_0^t\|(u_0,B_0)\|_{L^2}^{\frac{3}{4}}\|B(s)\|_{\dot{H}^4}^{\frac{13}{4}}+\|(u_0,B_0)\|_{L^2}^{\frac{3}{8}}\|(u,B)(s)\|_{\dot{H}^4}^{\frac{21}{8}}ds.
\end{align}
Suppose $\|(u,B)(t)\|_{\dot{H}^4}^2\leq2\|(u_0,B_0)(t)\|_{\dot{H}^4}^2$ for $t\in(0,t_1]$, then by \eqref{est1}, we have
\begin{align*}
&\|(u,B)(t)\|_{\dot{H}^4}^2+\int_0^t\|(u,B)(s)\|^2_{\dot{H}^{4+\alpha}}ds\notag\\
&\leq\|(u_0,B_0)(t)\|_{\dot{H}^4}^2\left(1+Ct_1\left(\|(u_0,B_0)\|_{L^2}^{\frac{3}{4}}\|(u_0,B_0)\|_{\dot{H}^4}^{\frac{5}{4}}+
\|(u_0,B_0)\|_{L^2}^{\frac{3}{8}}\|(u_0,B_0)\|_{\dot{H}^4}^{\frac{5}{8}}\right)\right).
\end{align*}
Bootstrap principle and \eqref{T1} yields that \eqref{H4}. The $L^\infty H^4$ prior estimate is sufficient to build uniqueness, see \cite{cdl} for example.

Now, we focus on the higher regularity estimates \eqref{decay1}. The mild solution of \eqref{Hall-MHD1} can be written as
\begin{align}
&u=e^{(-\Delta)^\alpha t} u_0+\int_0^te^{(-\Delta)^\alpha (t-s)}\mathbb{P}_{H}{\rm div}(B\otimes B-u\otimes u)(s)ds,\label{duhamelu}\\
&B=e^{(-\Delta)^\alpha t} B_0+\int_0^te^{(-\Delta)^\alpha (t-s)}\left({\rm div}(B\otimes u-u\otimes B)-\nabla\times {\rm div}(B\otimes B)\right)(s)ds.\label{duhamelb}
\end{align}
We first prove the case of $M=0$ and $N=1$. Employing the Gagliardo--Nirenberg--Sobolev inequality, \eqref{L2} and \eqref{H4}, we have {\small
\begin{align*}
t^{\frac{1}{2\alpha}}\|Du\|_{H^4}&\lesssim \|u_0\|_{H^4}+t^{\frac{1}{2\alpha}}\int_0^t(t-s)^{-\frac{1}{\alpha}}\|(u,B)\|_{L^\infty}\|(u,B)\|_{H^4}ds\\
&\lesssim\|(u_0,B_0)\|_{H^4}(1+t^{1-\frac{1}{2\alpha}}\|(u_0,B_0)\|_{L^2}^{\frac{5}{8}}\|(u_0,B_0)\|_{H^4}^{\frac{3}{8}}),\\
t^{\frac{1}{2\alpha}}\|DB\|_{H^4}
&\lesssim \|B_0\|_{H^4}+t^{\frac{1}{2\alpha}}\int_0^t(t-s)^{-\frac{1}{\alpha}}\left(\|(u,B)\|_{L^\infty}\|(u,B)\|_{H^4}
+\|\nabla B\|_{L^\infty}\|B\|_{H^4}\right)ds\\
&\quad+t^{\frac{1}{2\alpha}}\int_0^t(t-s)^{-\frac{1}{\alpha}}\|B\|_{L^\infty}\|\nabla B\|_{H^4}ds\\
&\lesssim\|(u_0,B_0)\|_{H^4}\left(1+t^{1-\frac{1}{2\alpha}}\left(\|(u_0,B_0)\|_{L^2}^{\frac{5}{8}}\|(u_0,B_0)\|_{H^4}^{\frac{3}{8}}
+\|(u_0,B_0)\|_{L^2}^{\frac{3}{8}}\|(u_0,B_0)\|_{H^4}^{\frac{5}{8}}\right)\right)\\
&\quad+t^{\frac{1}{2\alpha}}\int_0^t(t-s)^{-\frac{1}{\alpha}}\|\nabla B\|_{H^4}ds\|(u_0,B_0)\|_{L^2}^{\frac{5}{8}}\|(u_0,B_0)\|_{H^4}^{\frac{3}{8}}.
\end{align*}}
\!\!Thus combining the above argument with \eqref{T2}, we infer that
\begin{align*}
t^{\frac{1}{2\alpha}}\|D(u,B)(t)\|_{H^4}&\lesssim\|(u_0,B_0)\|_{H^4}\\
&\quad+t^{\frac{1}{2\alpha}}\int_0^t(t-s)^{-\frac{1}{\alpha}}\|\nabla B\|_{H^4}ds\|(u_0,B_0)\|_{L^2}^{\frac{5}{8}}\|(u_0,B_0)\|_{H^4}^{\frac{3}{8}}.
\end{align*}
If $t^{\frac{1}{2\alpha}}\|D(u,B)(t)\|_{H^4}\leq2C\|(u_0,B_0)\|_{H^4}$, we have
\begin{align*}
t^{\frac{1}{2\alpha}}\|D(u,B)(t)\|_{H^4}\leq C\|(u_0,B_0)\|_{H^4}\left(1+2Ct^{1-\frac{1}{\alpha}}\|(u_0,B_0)\|_{L^2}^{\frac{5}{8}}\|(u_0,B_0)\|_{H^4}^{\frac{3}{8}}\right).
\end{align*}
Combining bootstrap principle and \eqref{T2}, we get \eqref{decay1} holds for $M=0$ and $N=1$.

Now we prove the case of $M=0$ and $N>1$ by induction. Suppose $t^{\frac{n}{2\alpha}}\|D^n(u,B)(t)\|_{H^4}\lesssim\|(u_0,B_0)\|_{H^4}$ for $1\leq n\leq N-1$, then by Leibniz's rule, the Gagliardo--Nirenberg--Sobolev inequality, \eqref{L2}, \eqref{H4} and \eqref{T2}, we have for $t\in (0,t_1]:$ {\small
\begin{align*}
&\|D^{N-1}\left(B\otimes B-u\otimes u\right)\|_{H^4}+\|D^{N-1}\left(B\otimes u-u\otimes B\right)\|_{H^4}\\
&\lesssim \sum_{N'=0}^{N-1}\|D^{N-1-N'}(u,B)\|_{H^4}\|D^{N'}(u,B)\|_{L^\infty}\\
&\lesssim \sum_{N'=0}^{N-1}\|D^{N-1-N'}(u,B)\|_{H^4}\|D^{N'}(u,B)\|_{L^2}^{\frac{5}{8}}\|D^{N'}(u,B)\|_{H^4}^{\frac{3}{8}}\\
&\lesssim \|D^{N-1}(u,B)\|_{H^4}\|(u,B)\|_{L^2}^{\frac{5}{8}}\|(u,B)\|_{H^4}^{\frac{3}{8}}\\
&\quad+\|D^{N-2}(u,B)\|_{H^4}\|D(u,B)\|_{H^4}^{\frac{3}{8}}\|(u,B)\|_{L^2}^{\frac{15}{32}}\|(u,B)\|_{H^4}^{\frac{5}{32}}\\
&\quad+\|D^{N-3}(u,B)\|_{H^4}\|D^2(u,B)\|_{H^4}^{\frac{3}{8}}\|(u,B)\|_{L^2}^{\frac{5}{16}}\|(u,B)\|_{H^4}^{\frac{5}{16}}\\
&\quad+\|D^{N-4}(u,B)\|_{H^4}\|D^3(u,B)\|_{H^4}^{\frac{3}{8}}\|(u,B)\|_{L^2}^{\frac{5}{32}}\|(u,B)\|_{H^4}^{\frac{15}{32}}\\
&\quad+\sum_{N'=4}^{N-1}\|D^{N-1-N'}(u,B)\|_{H^4}\|D^{N'}(u,B)\|_{H^4}^{\frac{3}{8}}\|D^{N'-4}(u,B)\|_{H^4}^{\frac{5}{8}}\\
&\lesssim t^{\frac{1-N}{2\alpha}}\|(u_0,B_0)\|_{L^2}^{\frac{5}{8}}\|(u_0,B_0)\|_{H^4}^{\frac{11}{8}}
+t^{\frac{2-N}{2\alpha}-\frac{3}{16\alpha}}\|(u_0,B_0)\|_{L^2}^{\frac{15}{32}}\|(u_0,B_0)\|_{H^4}^{\frac{49}{32}}\\
&\quad+t^{\frac{3-N}{2\alpha}-\frac{3}{8\alpha}}\|(u_0,B_0)\|_{L^2}^{\frac{5}{16}}\|(u_0,B_0)\|_{H^4}^{\frac{27}{16}}
+t^{\frac{4-N}{2\alpha}-\frac{9}{16\alpha}}\|(u_0,B_0)\|_{L^2}^{\frac{5}{32}}\|(u_0,B_0)\|_{H^4}^{\frac{59}{32}}\\
&\quad+t^{\frac{1-N}{2\alpha}+\frac{5}{4\alpha}}\|(u_0,B_0)\|_{H^4}^{2}\\
&\lesssim t^{-\frac{N}{2\alpha}}\|(u_0,B_0)\|_{H^4},
\end{align*}}
and
{\small
\begin{align*}
&\sum_{N'=1}^{N}\|D^{N-N'}B\|_{H^4}\|D^{N'}B\|_{L^\infty}\\
&\lesssim \sum_{N'=1}^{N}\|D^{N-N'}B\|_{H^4}\|D^{N'}B\|_{L^2}^{\frac{5}{8}}\|D^{N'}B\|_{H^4}^{\frac{3}{8}}\\
&\lesssim \|D^{N-1}B\|_{H^4}\|DB\|_{H^4}^{\frac{3}{8}}\|B\|_{L^2}^{\frac{15}{32}}\|B\|_{H^4}^{\frac{5}{32}}
+\|D^{N-2}B\|_{H^4}\|D^2B\|_{H^4}^{\frac{3}{8}}\|B\|_{L^2}^{\frac{5}{16}}\|B\|_{H^4}^{\frac{5}{16}}\\
&\quad+\|D^{N-3}B\|_{H^4}\|D^3B\|_{H^4}^{\frac{3}{8}}\|B\|_{L^2}^{\frac{5}{32}}\|B\|_{H^4}^{\frac{15}{32}}
+\sum_{N'=4}^{N}\|D^{N-N'}B\|_{H^4}\|D^{N'}B\|_{H^4}^{\frac{3}{8}}\|D^{N'-4}B\|_{H^4}^{\frac{5}{8}}\\
&\lesssim t^{\frac{1-N}{2\alpha}-\frac{3}{16\alpha}}\|(u_0,B_0)\|_{L^2}^{\frac{15}{32}}\|(u,B)\|_{H^4}^{\frac{49}{32}}
+t^{\frac{2-N}{2\alpha}-\frac{3}{8\alpha}}\|(u_0,B_0)\|_{L^2}^{\frac{5}{16}}\|(u,B)\|_{H^4}^{\frac{27}{16}}\\
&\quad+t^{\frac{3-N}{2\alpha}-\frac{9}{16\alpha}}\|(u_0,B_0)\|_{L^2}^{\frac{5}{32}}\|(u,B)\|_{H^4}^{\frac{59}{32}}
+t^{-\frac{N}{2\alpha}+\frac{5}{4\alpha}}\|(u,B)\|_{H^4}^{2}\\
&\lesssim t^{-\frac{N}{2\alpha}}\|(u_0,B_0)\|_{H^4},
\end{align*}}
\!\!since $c$ is small enough in \eqref{T2}. Applying the above estimates, we deduce that {\small
\begin{align*}
t^{\frac{N}{2\alpha}}\|D^{N}u\|_{H^4}&\lesssim \|u_0\|_{H^4}+t^{\frac{N}{2\alpha}}\int_0^{t/2}(t-s)^{-\frac{N+1}{2\alpha}}\|(u,B)\|_{L^\infty}\|(u,B)\|_{H^4}ds\\
&\quad+t^{\frac{N}{2\alpha}}\int^t_{t/2}(t-s)^{-\frac{1}{\alpha}}\|D^{N-1}\left(B\otimes B-u\otimes u\right)\|_{H^4}ds\\
&\lesssim \|(u_0,B_0)\|_{H^4},
\end{align*}
\begin{align*}
t^{\frac{N}{2\alpha}}\|D^{N}B\|_{H^4}&\lesssim \|B_0\|_{H^4}+t^{\frac{N}{2\alpha}}\int_0^{t/2}(t-s)^{-\frac{N+1}{2\alpha}}\|(u,B)\|_{L^\infty}\|(u,B)\|_{H^4}ds\\
&\quad+t^{\frac{N}{2\alpha}}\int_0^{t/2}(t-s)^{-\frac{N+2}{2\alpha}}\|B\|_{L^\infty}\|B\|_{H^4}ds\\
&\quad+t^{\frac{N}{2\alpha}}\int_{t/2}^t(t-s)^{-\frac{1}{\alpha}}\|D^{N-1}\left(B\otimes u-u\otimes B\right)\|_{H^4}ds\\
&\quad+t^{\frac{N}{2\alpha}}\int_{t/2}^t(t-s)^{-\frac{1}{\alpha}}\sum_{N'=1}^{N}\|D^{N-N'}B\|_{H^4}\|D^{N'}B\|_{L^\infty}ds\\
&\quad+t^{\frac{N}{2\alpha}}\int_{t/2}^t(t-s)^{-\frac{1}{\alpha}}\|D^{N}B\|_{H^4}ds\|(u_0,B_0)\|_{L^2}^{\frac{5}{8}}\|(u,B)\|_{H^4}^{\frac{3}{8}}\\
&\lesssim \|(u_0,B_0)\|_{H^4}+t^{\frac{N}{2\alpha}}\int_{t/2}^t(t-s)^{-\frac{1}{\alpha}}\|D^{N}B\|_{H^4}ds\|(u_0,B_0)\|_{L^2}^{\frac{5}{8}}\|(u,B)\|_{H^4}^{\frac{3}{8}}.
\end{align*}}
\!\!Since $c$ is small enough in \eqref{T2}, bootstrap principle yields this \eqref{decay1} for $M=0$.

For the case of $M=1$ and $N=0$, applying the Gagliardo--Nirenberg--Sobolev inequality and \eqref{decay1} for $M=0$, we have{\small
\begin{align*}
\|\partial_tu\|_{H^4}&\lesssim\|(-\Delta)^\alpha u\|_{H^4}+\|u\cdot\nabla u\|_{H^4}+\|B\cdot\nabla B\|_{H^4}\\
&\lesssim t^{-1}\|(u_0,B_0)\|_{H^4}+\|(u,B)\|_{L^\infty}\|\nabla(u,B)\|_{H^4}+\|\nabla(u,B)\|_{L^\infty}\|(u,B)\|_{H^4}\\
&\lesssim t^{-1}\|(u_0,B_0)\|_{H^4}\left(1+t^{1-\frac{1}{2\alpha}}\|(u_0,B_0)\|_{L^2}^{\frac{5}{8}}\|(u,B)\|_{H^4}^{\frac{3}{8}}
+t\|(u_0,B_0)\|_{L^2}^{\frac{3}{8}}\|(u,B)\|_{H^4}^{\frac{5}{8}}\right)\\
&\lesssim t^{-1}\|(u_0,B_0)\|_{H^4}.
\end{align*}}
\!\!Similarly, one has{\small
\begin{align*}
\|\partial_tB\|_{H^4}&\lesssim\|(-\Delta)^\alpha B\|_{H^4}+\|u\cdot\nabla B-B\cdot\nabla u\|_{H^4}+\|\nabla\times {\rm div}(B\otimes B)\|_{H^4}\\
&\lesssim t^{-1}\|(u_0,B_0)\|_{H^4}+\|(u,B)\|_{L^\infty}\|\nabla(u,B)\|_{H^4}+\|\nabla(u,B)\|_{L^\infty}\|(u,B)\|_{H^4}\\
&\quad+\|B\|_{L^\infty}\|D^2B\|_{H^4}+\|D^2B\|_{L^\infty}\|B\|_{H^4}+\|DB\|_{L^\infty}\|DB\|_{H^4}\\
&\lesssim t^{-1}\|(u_0,B_0)\|_{H^4}+t^{-\frac{1}{\alpha}}\|(u_0,B_0)\|_{L^2}^{\frac{5}{8}}\|(u,B)\|_{H^4}^{\frac{11}{8}}\\
&\quad+\|(u_0,B_0)\|_{L^2}^{\frac{1}{8}}\|(u,B)\|_{H^4}^{\frac{15}{8}}+t^{-\frac{1}{2\alpha}}\|(u_0,B_0)\|_{L^2}^{\frac{3}{8}}\|(u,B)\|_{H^4}^{\frac{13}{8}}\\
&\lesssim t^{-1}\|(u_0,B_0)\|_{H^4}.
\end{align*}}
\!\!The case of $M=1$ and $N\geq1$ in \eqref{decay1} could be established in a similar way and we omit the details.
\end{proof}

\subsection{Stability estimates}
\begin{Proposition}\label{stability prop}
Fix $\alpha\in[\rho,\frac{5}{4})$, $\rho\in (1,\frac{5}{4})$ and $p_0\in(1,\frac{5}{4})$. For $q\in \mathbb{N}$, $(u_q,B_q,\mathring{R}_q^u,\mathring{R}_q^B)$ is a $C^0_tH^4$ smooth solution of \eqref{Hall-MHD2} with estimates \eqref{RestL1}--\eqref{uBestH4} holds. Let $t_0\in [0,T]$ and define initial data
$$(u_0,B_0)=(u_q,B_q)|_{t=t_0}.$$
Suppose $t_1>t_0$ with $[t_0,t_1]\subset[0,T]$ and
\begin{align}\label{T3}
0<t_1-t_0\leq\left(\delta_0^{-\frac{5}{8}}\lambda_q^{-\frac{15}{8}}\right)^{\frac{\rho}{\rho-1}}.
\end{align}
Then from \eqref{uBestL2}, \eqref{uBestH4} and Proposition \ref{local solu}, there exists a unique $C^0_tH^4$ smooth zero mean solution of the Hall-MHD equations \eqref{Hall-MHD1} on $[t_0,t_1]$, with initial data $(u_0,B_0)$. And there exists a constant $C=C(\rho,\alpha,p_0)$ such that for any $p\in[p_0,2]$ and $t\in (t_0,t_1]$,
\begin{align}
\|(u_q-u,B_q-B)(t)\|_{L^{2p}}&\leq C|t-t_0|\left\||\nabla|(\mathring{R}_q^u,{\rm div}\mathring{R}_q^B)\right\|_{L^\infty([t_0,t_1];L^{2p})},\label{chaL2p}\\
\|(u_q-u,B_q-B)(t)\|_{L^{p}}&\leq C|t-t_0|\left\||\nabla|(\mathring{R}_q^u,{\rm div}\mathring{R}_q^B)\right\|_{L^\infty([t_0,t_1];L^{p})},\label{chaLp}\\
\left\|\left(\mathcal{R}(u_q-u),\mathcal{R}{\rm curl}^{-1}(B_q-B)\right)(t)\right\|_{L^{p}}&\leq C|t-t_0|\left\|(\mathring{R}_q^u,\mathring{R}_q^B)\right\|_{L^\infty([t_0,t_1];L^{p})}.\label{chaLp1}
\end{align}
\end{Proposition}
Let $\frac{1}{p_0}=1-\frac{\varepsilon_R}{35}$ in the above proposition, we have the following corollary:
\begin{Corollary}\label{corollary}
Under the hypothesis of Proposition \ref{stability prop}, assuming $a\geq1$ large enough depending on $\varepsilon_R$, then we have:
\begin{align}
\|(u_q-u,B_q-B)\|_{L^\infty([t_0,t_1];L^{2p})}&\leq|t_1-t_0|\lambda_q^{\frac{7}{11}(10-\frac{3}{p})},\label{chaL2p1}\\
\|(u_q-u,B_q-B)\|_{L^\infty([t_0,t_1];L^{2})}&\leq|t_1-t_0|\lambda_q^5,\label{chaL2}\\
\left\|\left(\mathcal{R}(u_q-u),\mathcal{R}{\rm curl}^{-1}(B_q-B)\right)\right\|_{L^\infty([t_0,t_1];L^{1})}&\leq|t_1-t_0|\lambda_q^{-\frac{3}{4}\varepsilon_R}\delta_{q+1}.\label{chaL1}
\end{align}
\end{Corollary}
\begin{proof}[Proof of Corollary \ref{corollary}]
By the Gagliardo--Nirenberg--Sobolev inequality, \eqref{RestL1} and \eqref{RestH4}, we find that
\begin{align*}
\left\||\nabla|(\mathring{R}_q^u,{\rm div}\mathring{R}_q^B)\right\|_{L^{2p}}
&\lesssim\left\|\mathring{R}_q^u\right\|_{L^1}^{\frac{3}{11}(1+\frac{1}{p})}\left\|\mathring{R}_q^u\right\|_{H^4}^{\frac{1}{11}(8-\frac{3}{p})}
+\left\|\mathring{R}_q^B\right\|_{L^1}^{\frac{1}{11}(1+\frac{3}{p})}\left\|\mathring{R}_q^B\right\|_{H^4}^{\frac{1}{11}(10-\frac{3}{p})}\\
&\lesssim\left(\lambda_q^{-\varepsilon_R}\delta_{q+1}\right)^{\frac{1}{11}(1+\frac{3}{p})}\lambda_q^{\frac{7}{11}(10-\frac{3}{p})}\\
&\leq\lambda_q^{\frac{7}{11}(10-\frac{3}{p})}.
\end{align*}
Then \eqref{chaL2p} implies \eqref{chaL2p1}. The estimate \eqref{chaL2} could be established in a same way.

It follows from the Gagliardo--Nirenberg--Sobolev inequality, \eqref{RestL1} and \eqref{RestH4} that
\begin{align*}
\left\|\mathring{R}_q^u\right\|_{L^{p_0}}&\lesssim\left\|\mathring{R}_q^u\right\|_{L^1}^{\frac{1}{11}(5+\frac{6}{p_0})}
\left\|\mathring{R}_q^u\right\|_{L^1}^{\frac{1}{11}(1-\frac{1}{p_0})}
\lesssim\left(\lambda_q^{-\varepsilon_R}\delta_{q+1}\right)^{\frac{1}{11}(11-\frac{6}{35}\varepsilon_R)}\lambda_q^{\frac{6}{55}\varepsilon_R}\\
&\lesssim\lambda_q^{-\varepsilon_R}\delta_{q+1}\left(\lambda_q^{\frac{\varepsilon_R}{35}+\frac{1}{5}}\delta_{q+1}^{-\frac{1}{35}}\right)^{\frac{6}{11}\varepsilon_R}
\lesssim\lambda_q^{-\varepsilon_R}\delta_{q+1}\lambda_q^{\frac{6}{11}\varepsilon_R\left(\frac{\varepsilon_R}{35}+\frac{1}{5}+\frac{2}{35}\beta b\right)}\\
&\leq\lambda_q^{-\frac{4}{5}\varepsilon_R}\delta_{q+1},
\end{align*}
since $\beta b\ll1$ and $\varepsilon_R\ll1$. Similarly, we have $\left\|\mathring{R}_q^B\right\|_{L^{p_0}}\leq\lambda_q^{-\frac{4}{5}\varepsilon_R}\delta_{q+1}$. By using \eqref{chaLp1}, we conclude that
\begin{align*}
\left\|\left(\mathcal{R}(u_q-u),\mathcal{R}{\rm curl}^{-1}(B_q-B)\right)\right\|_{L^{1}}&\leq\left\|\left(\mathcal{R}(u_q-u),\mathcal{R}{\rm curl}^{-1}(B_q-B)\right)\right\|_{L^{p_0}}\\
&\leq|t-t_0|\lambda_q^{-\frac{3}{4}\varepsilon_R}\delta_{q+1}.
\end{align*}
\end{proof}

\begin{proof}[Proof of Proposition \ref{stability prop}]
Without loss of generality, we prove this proposition with $t_0=0$. Define
$$v=u_q-u,\quad b=B_q-B,$$
then one has
\begin{equation}\label{Hall-MHD3}
\begin{cases}
\partial_tv+(-\Delta)^{\alpha}v+{\rm div}\left(v\otimes u_q+u\otimes v-b\otimes B_q-B\otimes b\right)+\nabla(p_q-p)={\rm div}\mathring{R}_q^u,\\
\partial_tb+(-\Delta)^{\alpha}b+{\rm div}\left(v\otimes B_q+u\otimes b-b\otimes u_q-B\otimes v\right)\\
\quad\quad+\nabla\times{\rm div}(b\otimes B_q+B\otimes b)=\nabla\times{\rm div}\mathring{R}_q^B,\\
{\rm div}v={\rm div}b=0,\\
(v,b)|_{t=0}=(0,0).
\end{cases}
\end{equation}
From Proposition \ref{local solu}, we know that for $t\in(0,t_1]$, the unique solution $(v,b)$ can be written as
\begin{align}\label{vb}
&v(x,t)=\int_0^te^{-(-\Delta)^\alpha(t-s)}\mathbb{P}_H{\rm div}\left(\mathring{R}_q^u-v\otimes u_q-u\otimes v+b\otimes B_q+B\otimes b\right)(s)ds,\\
&b(x,t)=\int_0^te^{-(-\Delta)^\alpha(t-s)}\nabla\times{\rm div}\left(\mathring{R}_q^B-b\otimes B_q-B\otimes b\right)(s)\notag\\
&\quad\quad\quad\quad+ e^{-(-\Delta)^\alpha(t-s)}{\rm div}\left(b\otimes u_q+B\otimes v-v\otimes B_q-u\otimes b\right)(s)ds.
\end{align}
From Lemma 3.1 in \cite{myz}, we know that for $1\leq r\leq\infty$,
\begin{align}\label{fractional est}
\left\|(-\Delta)^{\nu/2} e^{-(-\Delta)^\alpha t}\phi\right\|_{L^r}\lesssim t^{-\frac{\nu}{2\alpha}}\left\|\phi\right\|_{L^r}.
\end{align}
Then we return to \eqref{vb}, to obtain
\begin{align*}
&\|v(t)\|_{L^{2p}}\lesssim t\left\|{\rm div}\mathring{R}_q^u\right\|_{L^{2p}}+\int_0^t(t-s)^{-\frac{1}{2\alpha}}\left(\|(u_q,u)\|_{L^\infty}\|v\|_{L^{2p}}
+\|(B_q,B)\|_{L^\infty}\|b\|_{L^{2p}}\right)(s)ds,\\
&\|b(t)\|_{L^{2p}}\lesssim t\left\|\nabla\times{\rm div}\mathring{R}_q^B\right\|_{L^{2p}}+\int_0^t(t-s)^{-\frac{1}{\alpha}}(\|(B_q,B)\|_{L^\infty}\|b\|_{L^{2p}})(s)\\
&\quad\quad\quad\quad+(t-s)^{-\frac{1}{2\alpha}}\left(\|(u_q,u)\|_{L^\infty}\|b\|_{L^{2p}}
+\|(B_q,B)\|_{L^\infty}\|v\|_{L^{2p}}\right)(s)ds.
\end{align*}
By the Gagliardo--Nirenberg--Sobolev inequality, Proposition \ref{local solu}, \eqref{uBestL2} and \eqref{uBestH4}, we can infer that
\begin{align}\label{es2}
\|(v,b)(t)\|_{L^{2p}}&\lesssim t\left\||\nabla|(\mathring{R}_q^B,{\rm div}\mathring{R}_q^u)\right\|_{L^{2p}}\notag\\
&+\int_0^t\left((t-s)^{-\frac{1}{\alpha}}+(t-s)^{-\frac{1}{2\alpha}}\right)\|(v,b)(s)\|_{L^{2p}}\|(u_q,B_q)\|_{L^2}^{\frac{5}{8}}\|(u_q,B_q)\|_{H^4}^{\frac{3}{8}}ds\notag\\
&\lesssim t\left\||\nabla|(\mathring{R}_q^B,{\rm div}\mathring{R}_q^u)\right\|_{L^{2p}}\notag\\
&+\delta_0^{\frac{5}{16}}\lambda_q^{\frac{15}{8}}\int_0^t\left((t-s)^{-\frac{1}{\alpha}}+(t-s)^{-\frac{1}{2\alpha}}\right)\|(v,b)(s)\|_{L^{2p}}ds.
\end{align}
Finally, we apply bootstrap principle to obtain \eqref{chaL2p}. Suppose that for $t\in(0,t_1]$,
$$\|(v,b)(t)\|_{L^{2p}}\leq2Ct\left\||\nabla|(\mathring{R}_q^B,{\rm div}\mathring{R}_q^u)\right\|_{L^{2p}},$$
then from \eqref{es2}, we have
\begin{align*}
\|(v,b)(t)\|_{L^{2p}}\leq Ct\left\||\nabla|(\mathring{R}_q^B,{\rm div}\mathring{R}_q^u)\right\|_{L^{2p}}
\left(1+2C\delta_0^{\frac{5}{16}}\lambda_q^{\frac{15}{8}}\left(t^{1-\frac{1}{\alpha}}+t^{1-\frac{1}{2\alpha}}\right)\right).
\end{align*}
In view of \eqref{T3}, \eqref{chaL2p} holds for $a$ large enough. The estimate \eqref{chaLp} can be obtained  in a same way, we omit the details.

In order to prove $\eqref{chaLp1}$, define
$$z=\Delta^{-1}\nabla\times v,\quad\varphi=\Delta^{-1} b.$$
Since ${\rm div}v={\rm div}b=0$, one has
$$v=-\nabla\times z, \quad b=\Delta \varphi=-\nabla\times\nabla\times \varphi, \quad {\rm curl}^{-1}b=-\Delta^{-1}\nabla\times b,$$
thus
$$\mathcal{R}v=-\mathcal{R}\nabla\times z,\quad \mathcal{R}{\rm curl}^{-1}b=-\mathcal{R}\Delta^{-1}\nabla\times b=-\mathcal{R}\nabla\times \varphi.$$
The Calderon-Zygmund inequality yields
$$\|\mathcal{R}v\|_{L^p}\lesssim\|z\|_{L^p},\quad\| \mathcal{R}{\rm curl}^{-1}b\|_{L^p}\lesssim\|\varphi\|_{L^p}.$$
Recall equation \eqref{Hall-MHD3}, then $(z,\varphi)$ satisfies
\begin{equation}\label{Hall-MHD4}
\begin{cases}
\partial_tz+(-\Delta)^\alpha z+\Delta^{-1}\nabla\times{\rm div}\left(-(\nabla\times z)\otimes u_q-u\otimes(\nabla\times z)-\Delta\varphi\otimes B_q-B\otimes\Delta\varphi\right)\\
\quad\quad+\nabla(p_q-p)=\Delta^{-1}\nabla\times{\rm div}\mathring{R}_q^u,\\
\partial_t\varphi+(-\Delta)^\alpha \varphi+\Delta^{-1}{\rm div}\left(-(\nabla\times z)\otimes B_q+B\otimes(\nabla\times z)-\Delta\varphi\otimes u_q+u\otimes\Delta\varphi\right)\\
\quad\quad+\Delta^{-1}\nabla\times{\rm div}\left(\Delta\varphi\otimes B_q+B\otimes\Delta\varphi\right)=\Delta^{-1}\nabla\times{\rm div}\mathring{R}_q^B.
\end{cases}
\end{equation}
We have the following identities:
\begin{align}
&(\nabla\times z)\cdot\nabla u_q={\rm div}(((z\times\nabla)u_q)^T),\;
(\nabla\times z)\cdot\nabla B_q={\rm div}(((z\times\nabla)B_q)^T),\label{eq2}\\
&\Delta^{-1}\nabla\times{\rm div}\left(u\otimes(\nabla\times z)\right)\notag\\
&\quad=-u\cdot\nabla z+\Delta^{-1}\nabla\times{\rm div}((z\times\nabla)u)+\Delta^{-1}\nabla{\rm div}((u\cdot\nabla)z),\label{eq3}\\
&B\cdot\nabla(\nabla\times z)=\nabla\times{\rm div}(z\otimes B)+{\rm div}((z\times\nabla)\otimes B),\label{eq4}\\
&\Delta\varphi\otimes B_q+B\otimes\Delta\varphi=\Delta\left(\varphi\otimes B_q+B\otimes\varphi\right)+\varphi\otimes \Delta B_q+\Delta B\otimes\varphi\notag\\
&\quad\quad\quad\quad\quad\quad\quad\quad\quad-2\partial_k\left(\varphi\otimes\partial_k B_q+\partial_k B\otimes\varphi\right),\label{eq5}\\
&\Delta\varphi\otimes u_q-u\otimes\Delta\varphi=\Delta\left(\varphi\otimes u_q-u\otimes\varphi\right)+\varphi\otimes \Delta u_q-\Delta u\otimes\varphi\notag\\
&\quad\quad\quad\quad\quad\quad\quad\quad\quad-2\partial_k\left(\varphi\otimes\partial_k u_q-\partial_k u\otimes\varphi\right),\label{eq6}
\end{align}
where $({\rm div}S)^i=\partial_jS^{ij}$, for $S=(S^{ij})_{i,j=1}^3$. Here we only show \eqref{eq4}--\eqref{eq6} and we can refer to \cite[Proposition 3.2]{bcv} for the details of \eqref{eq2} and \eqref{eq3}. The $j^{th}$ component of $B\cdot\nabla(\nabla\times z)$ can be written as:
\begin{align*}
\left(B\cdot\nabla(\nabla\times z)\right)^j&=B_k\partial_k(\varepsilon_{jln}\partial_lz_n)=\varepsilon_{jln}B_k\partial_k\partial_lz_n
=\varepsilon_{jln}\partial_k(B_k\partial_lz_n)\\
&=\varepsilon_{jln}\partial_l\partial_k(z_n B_k)+\partial_k\left(\varepsilon_{jnl}z_n\partial_lB_k\right)\\
&=\varepsilon_{jln}\partial_l({\rm div}(z\otimes B))^n+\partial_k\left((z\times\nabla)^jB_k\right)\\
&=\left(\nabla\times{\rm div}(z\otimes B)+{\rm div}((z\times\nabla)\otimes B)\right)^j.
\end{align*}
The $ij^{th}$ component of $B\otimes\Delta\varphi$ can be written as:
\begin{align*}
(B\otimes\Delta\varphi)^{ij}&=B_i\Delta\varphi_j=\Delta(B_i\varphi_j)-\varphi_j\Delta B_i-2\nabla B_i\cdot\nabla \varphi_j\\
&=\Delta(B_i\varphi_j)+\varphi_j\Delta B_i-2{\rm div}(\nabla B_i\varphi_j),
\end{align*}
and $\Delta\varphi\otimes B_q$ can be dealt with in a same way. Thus, we obtain \eqref{eq5} and \eqref{eq6}.

Plugging \eqref{eq2}--\eqref{eq6} into \eqref{Hall-MHD4}, we have
\begin{align*}
z=&\int_0^te^{-(-\Delta)^\alpha(t-s)}\mathbb{P}_H\left(\Delta^{-1}\nabla\times{\rm div}\mathring{R}_q^u+\Delta^{-1}\nabla\times{\rm div}(((z\times\nabla)u_q)^T)\right)\\
&+e^{-(-\Delta)^\alpha(t-s)}\mathbb{P}_H\left(-u\cdot\nabla z+\Delta^{-1}\nabla\times{\rm div}((z\times\nabla)u)+\Delta^{-1}\nabla{\rm div}((u\cdot\nabla)z)\right)\\
&+e^{-(-\Delta)^\alpha(t-s)}\mathbb{P}_H\left(\nabla\times{\rm div}\left(\varphi\otimes B_q+B\otimes\varphi\right)+\Delta^{-1}\nabla\times{\rm div}\left(\varphi\otimes \Delta B_q+\Delta B\otimes\varphi\right)\right)\\
&-2e^{-(-\Delta)^\alpha(t-s)}\mathbb{P}_H\left(\Delta^{-1}\nabla\times{\rm div}\partial_k\left(\varphi\otimes\partial_k B_q+\partial_k B\otimes\varphi\right)\right)ds,\\
\varphi=&\int_0^te^{-(-\Delta)^\alpha(t-s)}\left(\Delta^{-1}\nabla\times{\rm div}\mathring{R}_q^B+\Delta^{-1}{\rm div}(((z\times\nabla)B_q)^T)\right)\\
&+e^{-(-\Delta)^\alpha(t-s)}\left(-\Delta^{-1}\nabla\times{\rm div}(z\otimes B)-\Delta^{-1}{\rm div}((z\times\nabla)\otimes B)\right)\\
&+e^{-(-\Delta)^\alpha(t-s)}\left({\rm div}\left(\varphi\otimes u_q-u\otimes\varphi\right)+\Delta^{-1}{\rm div}\left(\varphi\otimes \Delta u_q-\Delta u\otimes\varphi\right)\right)\\
&-2e^{-(-\Delta)^\alpha(t-s)}\Delta^{-1}{\rm div}\partial_k\left(\varphi\otimes\partial_k u_q-\partial_k u\otimes\varphi\right)\\
&-e^{-(-\Delta)^\alpha(t-s)}\left(\nabla\times{\rm div}\left(\varphi\otimes B_q+B\otimes\varphi\right)+\Delta^{-1}\nabla\times{\rm div}\left(\varphi\otimes \Delta B_q+\Delta B\otimes\varphi\right)\right)\\
&+2e^{-(-\Delta)^\alpha(t-s)}\left(\Delta^{-1}\nabla\times{\rm div}\partial_k\left(\varphi\otimes\partial_k B_q+\partial_k B\otimes\varphi\right)\right)ds.
\end{align*}
Utilizing  \eqref{fractional est} and the Calderon-Zygmund inequality, we obtain
\begin{align*}
\|z\|_{L^p}\lesssim& t\left\|\mathring{R}_q^u\right\|_{L^p}+\int_0^t\|\nabla(u,u_q)\|_{L^\infty}\|z\|_{L^p}+(t-s)^{-\frac{1}{2\alpha}}\|u\|_{L^\infty}\|z\|_{L^p}\\
&+(t-s)^{-\frac{1}{\alpha}}\|(B,B_q)\|_{L^\infty}\|\varphi\|_{L^p}+(t-s)^{-\frac{1}{2\alpha}}\|\nabla(B,B_q)\|_{L^\infty}\|\varphi\|_{L^p}\\
&+\|\Delta(B,B_q)\|_{L^\infty}\|\varphi\|_{L^p}ds,\\
\|\varphi\|_{L^p}\lesssim& t\left\|\mathring{R}_q^B\right\|_{L^p}+\int_0^t\|B\|_{L^\infty}\|z\|_{L^p}+(t-s)^{-\frac{1}{2\alpha}}\|(u,u_q)\|_{L^\infty}\|\varphi\|_{L^p}\\
&+\|\nabla(u,u_q)\|_{L^\infty}\|\varphi\|_{L^p}+(t-s)^{-\frac{1}{\alpha}}\|(B,B_q)\|_{L^\infty}\|\varphi\|_{L^p}\\
&+(t-s)^{-\frac{1}{2\alpha}}\|\nabla(B,B_q)\|_{L^\infty}\|\varphi\|_{L^p}+\|\Delta(B,B_q)\|_{L^\infty}\|\varphi\|_{L^p}\\
&+\left\||\nabla|^{-1}\mathbb{P}_{\neq0}\left(((z\times\nabla)B_q)^T-(z\times\nabla)\otimes B+\varphi\otimes \Delta u_q-\Delta u\otimes\varphi\right)\right\|_{L^p}ds.
\end{align*}
Therefore, we conclude that
\begin{align}\label{est2}
\|(z,\varphi)\|_{L^p}\lesssim& t\left\|(\mathring{R}_q^u,\mathring{R}_q^B)\right\|_{L^p}+\int_0^t(t-s)^{-\frac{1}{\alpha}}\|(B,B_q)\|_{L^\infty}\|\varphi\|_{L^p}\notag\\
&+(t-s)^{-\frac{1}{2\alpha}}\left(\|\nabla(B,B_q)\|_{L^\infty}+\|(u,u_q)\|_{L^\infty}\right)\|(z,\varphi)\|_{L^p}\notag\\
&+\|(u,B,u_q,B_q)\|_{W^{2,\infty}}\|(z,\varphi)\|_{L^p}ds.
\end{align}
Now we suppose for $t\in (0,t_1]$,
$$\|(z,\varphi)\|_{L^p}\leq2Ct\left\|(\mathring{R}_q^u,\mathring{R}_q^B)\right\|_{L^p},$$
and then combining \eqref{est2}, the Gagliardo--Nirenberg--Sobolev inequality, Proposition \ref{local solu}, \eqref{uBestL2}, \eqref{uBestH4} and \eqref{T3}, we can infer that
\begin{align*}
\|(z,\varphi)\|_{L^p}&\left(Ct\left\|(\mathring{R}_q^u,\mathring{R}_q^B)\right\|_{L^p}\right)^{-1}\leq 1+2Ct\|(u_q,B_q)\|_{L^2}^{\frac{1}{8}}\|(u_q,B_q)\|_{H^4}^{\frac{7}{8}}\\
&+2C\left(t^{1-\frac{1}{2\alpha}}+t\right)\|(u_q,B_q)\|_{L^2}^{\frac{3}{8}}\|(u_q,B_q)\|_{H^4}^{\frac{5}{8}}\\
&+2C\left(t^{1-\frac{1}{\alpha}}+t^{1-\frac{1}{2\alpha}}+t\right)\|(u_q,B_q)\|_{L^2}^{\frac{5}{8}}\|(u_q,B_q)\|_{H^4}^{\frac{3}{8}}\\
&\leq1+2Ct\delta_0^{\frac{1}{16}}\lambda_q^{\frac{35}{8}}+2C\left(t^{1-\frac{1}{2\alpha}}+t\right)\delta_0^{\frac{3}{16}}\lambda_q^{\frac{25}{8}}
+2C\left(t^{1-\frac{1}{\alpha}}+t^{1-\frac{1}{2\alpha}}+t\right)\delta_0^{\frac{5}{16}}\lambda_q^{\frac{15}{8}}\\
&<2.
\end{align*}
This together with bootstrap principle yields that \eqref{chaLp1} holds.
\end{proof}

\subsection{Proof of Proposition \ref{glue prop}}
This section closely mirrors \cite[Section 3.3]{bcv}. Let $0\leq\eta_i(t)\leq1$
$$\sum_{i=0}^{n_{q+1}}\eta_i(t)=1\;on\;[\frac{T}{3},\frac{2T}{3}],$$
be a $C^\infty$ smooth partition of unity with the following properties:
\begin{align}\label{cutoffest}
{\rm supp}\eta_i\subset[t_i,t_{i+1}+\tau_{q+1}];\quad \eta_i=1\;on\;[t_i+\tau_{q+1},t_{i+1}];\quad \|\eta_i\|_{C^M_t}\lesssim \tau_{q+1}^{-M},
\end{align}
where $t_i=i\vartheta_{q+1}$.

Let $(u_i,B_i)$ be the solution of the Hall-MHD equations with initial data given by $(u_q,B_q)$:
\begin{equation}\label{Hall-MHD5}
\begin{cases}
\partial _tu_i+(-\Delta)^\alpha u_i+{\rm div}(u_i\otimes u_i-B_i\otimes B_i)+\nabla p_i= 0, \\
\partial _tB_i+(-\Delta)^\alpha B_i+{\rm div}(u_i\otimes B_i-B_i\otimes u_i)+\nabla\times{\rm div}(B_i\otimes B_i)=0,\\
{\rm div}u_i = {\rm div}B_i=0,\\
(u_i,B_i)|_{t=t_{i-1}}=(u_q,B_q)|_{t=t_{i-1}}.
\end{cases}
\end{equation}
In view of Proposition \ref{local solu}, for $0<t-t_{i-1}\leq\left(\delta_0^{-\frac{5}{8}}\lambda_q^{-\frac{15}{8}}\right)^{-\frac{\rho}{\rho-1}}$, the unique solution $(u_i,B_i)$ satisfies the following estimates:
\begin{align}
\|(u_i,B_i)(t)\|_{L^2}\leq\|(u_q,B_q)(t_{i-1})\|_{L^2}&\leq 2\delta_0^{\frac{1}{2}}-\delta_q^{\frac{1}{2}},\label{est3}\\
\|(u_i,B_i)(t)\|_{H^4}\leq2\|(u_q,B_q)(t_{i-1})\|_{H^4}&\leq 2\lambda_q^5,\label{est4}\\
|t-t_{i-1}|^{\frac{N}{2\alpha}+M}\left\|\partial_t^MD^N(u_i,B_i)(t)\right\|_{H^4}&\lesssim\lambda_q^5,\label{est5}
\end{align}
for $N\geq0$ and $M\in\{0,1\}$.
Since
$$\vartheta_{q+1}\leq\lambda_q^{-\frac{4\rho}{\rho-1}}\ll\left(\delta_0^{-\frac{5}{8}}\lambda_q^{-\frac{15}{8}}\right)^{-\frac{\rho}{\rho-1}},$$
from \eqref{est5}, we have
\begin{align}\label{decay2}
\sup_{t\in {\rm supp}\eta_i}\left\|\partial_t^MD^N(u_i,B_i)(t)\right\|_{H^4}\lesssim\vartheta_{q+1}^{-\frac{N}{2\alpha}-M}\lambda_q^5,
\end{align}
for $1\leq i\leq n_{q+1}$, $N\geq0$ and $M\in\{0,1\}$.

For $t\in[0,\frac{T}{3}]\cup[\frac{2T}{3},T]=\mathscr{G}^{(0)}$, define
\begin{align*}
&(\bar{u}_q,\bar{B}_q)(x,t)=(u_q,B_q)(x,t)=(u_0,B_0)(x,t),\\
&\bar{p}_q^{(1)}(x,t)=p_q(x,t)=p_0(x,t),
\end{align*}
and for $t\in [\frac{T}{3},\frac{2T}{3}]$, define
\begin{align}\label{bardef}
&\bar{u}_q(x,t)=\sum_{i=0}^{n_q+1}\eta_i(t)u_i(x,t),\quad\bar{B}_q(x,t)=\sum_{i=0}^{n_q+1}\eta_i(t)B_i(x,t),\quad\\
&\bar{p}_q^{(1)}(x,t)=\sum_{i=0}^{n_q+1}\eta_i(t)p_i(x,t).\notag
\end{align}
Immediately, for $t\in {\rm supp\eta_i}$, we have
\begin{align}\label{eq7}
\bar{u}_q=(1-\eta_i)u_{i-1}+\eta_iu_i,\quad\bar{B}_q=(1-\eta_i)B_{i-1}+\eta_iB_i,\quad\bar{p}_q^{(1)}=(1-\eta_i)p_{i-1}+\eta_ip_i.
\end{align}

Now we prove \eqref{eq1}. For any $t\in \mathscr{G}^{(q)}$, then $t\in [t_i,t_{i+1}]$ for some $i\in \mathbb{N}$. Thus, for any $s\in[t_{i-2},t_{i+1}+\tau_{q+1}]$, ${\rm dist}(s,\mathscr{G}^{(q)})\leq3\vartheta_{q+1}\ll\tau_q$,
which means that $(u_q,B_q)$ is an exact solution of the Hall-MHD equations on $[t_{i-2},t_{i+1}+\tau_{q+1}]$. By the uniqueness, one has
$$(u_{i-1},B_{i-1})=(u_i,B_i)=(u_q,B_q)$$
on $[t_{i-2},t_{i+1}+\tau_{q+1}]$, which implies \eqref{eq1}. In order to prove (\romannumeral2) in Section \ref{sec2} holds at the level of $q+1$, we show that $\mathscr{B}^{(q+1)}\subset\mathscr{B}^{(q)}$, for any $q\geq 0$. In view of (\romannumeral6), we infer
\begin{align}\label{baohan}
\left\{t:{\rm dist}\left(t, \left\{t:(\mathring{R}_q^u,\mathring{R}_q^B)\neq0\right\}\right)<\tau_q\right\}\subset\mathscr{B}^{(q)}.
\end{align}
By the definitions \eqref{index set} and \eqref{bad q+1}, for any $t\in \mathscr{B}^{(q+1)}$, $t\in (t_i-2\tau_{q+1},t_i+3\tau_{q+1})$ for some $i\in \mathscr{C}$ or $i-1\in \mathscr{C}$, which means that there exists $s\in[t_{i-2},t_{i+1}+\tau_{q+1}]$ such that $(\mathring{R}_q^u,\mathring{R}_q^B)(s)\neq0$, then
$${\rm dist}\left(t, \left\{t:(\mathring{R}_q^u,\mathring{R}_q^B)\neq0\right\}\right)<\tau_q.$$
Thus, combining this and \eqref{baohan}, we conclude $\mathscr{B}^{(q+1)}\subset\mathscr{B}^{(q)}$.

In order to prove (\romannumeral4) in Section \ref{sec2} holds at the level of $q+1$, we estimate the cardinality of $\mathscr{C}$. For any $i\in \mathscr{C}$, there exists $t\in [t_{i-1},t_{i+1}+\tau_{q+1}]$ with $(\mathring{R}_q^u,\mathring{R}_q^B)(t)\neq0$, which implies ${\rm dist}(t,\mathscr{G}^{(q)})>\tau_q$, thus we have $[t_i,t_{i+1}]\subset(t-\tau_q,t+\tau_q)\subset\mathscr{B}^{(q)}$. Therefore, one has
$$|\mathscr{C}|\leq\frac{|\mathscr{B}^{(q)}|}{\vartheta_{q+1}},$$
and in view of the definition \eqref{bad q+1}, we conclude that
$$|\mathscr{B}^{(q+1)}|\leq2|\mathscr{C}|5\tau_{q+1}=10|\mathscr{B}^{(q)}|\frac{\tau_{q+1}}{\vartheta_{q+1}}.$$

Now we prove \eqref{bar supp}. From \eqref{Hall-MHD5} and \eqref{eq7}, we can infer that
\begin{align}\label{eq8}
&\partial _t\bar{u}_q+(-\Delta)^\alpha \bar{u}_q+{\rm div}(\bar{u}_q\otimes \bar{u}_q-\bar{B}_q\otimes \bar{B}_q)+\nabla \bar{p}^{(1)}_q\notag\\
&=\partial_t\eta_i(u_i-u_{i-1})+(1-\eta_i)\partial_tu_{i-1}+(1-\eta_i)(-\Delta)^\alpha u_{i-1}\notag\\
&\quad+(1-\eta_i)^2{\rm div}(u_{i-1}\otimes u_{i-1}-B_{i-1}\otimes B_{i-1})+(1-\eta_i)\nabla p_{i-1}\notag\\
&\quad+\eta_i\partial_tu_i+\eta_i(-\Delta)^\alpha u_i+\eta_i^2{\rm div}(u_{i}\otimes u_{i}-B_{i}\otimes B_{i})+\eta_i\nabla p_{i}\notag\\
&\quad+\eta_i(1-\eta_i){\rm div}(u_{i-1}\otimes u_{i}+u_{i}\otimes u_{i-1}-B_{i-1}\otimes B_{i}-B_{i}\otimes B_{i-1})\notag\\
&=\partial_t\eta_i(u_i-u_{i-1})\notag\\
&\quad+\eta_i(\eta_i-1){\rm div}((u_{i}-u_{i-1})\otimes(u_{i}-u_{i-1})-(B_{i}-B_{i-1})\otimes(B_{i}-B_{i-1})),
\end{align}
and similarly, one has
\begin{align}\label{eq9}
&\partial _t\bar{B}_q+(-\Delta)^\alpha \bar{B}_q+{\rm div}(\bar{u}_q\otimes \bar{B}_q-\bar{B}_q\otimes \bar{u}_q)+\nabla\times{\rm div}(\bar{B}_q\otimes \bar{B}_q)\notag\\
&=\partial_t\eta_i(B_i-B_{i-1})+\eta_i(\eta_i-1)\nabla\times{\rm div}((B_{i}-B_{i-1})\otimes(B_{i}-B_{i-1}))\notag\\
&\quad+\eta_i(\eta_i-1){\rm div}((u_{i}-u_{i-1})\otimes(B_{i}-B_{i-1})-(B_{i}-B_{i-1})\otimes(u_{i}-u_{i-1})).
\end{align}
Thus we define
\begin{align}\label{def2}
\bar{p}_q&=\bar{p}_q^{(1)}-\frac{1}{3}\eta_i(\eta_i-1)\left(|u_{i}-u_{i-1}|^2-|B_{i}-B_{i-1}|^2-\int_{\mathbb{T}^3}|u_{i}-u_{i-1}|^2-|B_{i}-B_{i-1}|^2\right)\notag\\
\mathring{\bar{R}}_q^u&=\partial_t\eta_i\mathcal{R}(u_i-u_{i-1})
+\eta_i(\eta_i-1)\left((u_{i}-u_{i-1})\mathring{\otimes}(u_{i}-u_{i-1})-(B_{i}-B_{i-1})\mathring{\otimes}(B_{i}-B_{i-1})\right),\notag\\
\mathring{\bar{R}}_q^B&=\partial_t\eta_i\mathcal{R}{\rm curl}^{-1}(B_i-B_{i-1})+\eta_i(\eta_i-1)((B_{i}-B_{i-1})\mathring{\otimes}(B_{i}-B_{i-1}))\notag\\
&+\eta_i(\eta_i-1)\mathcal{R}{\rm curl}^{-1}{\rm div}((u_{i}-u_{i-1})\otimes(B_{i}-B_{i-1})-(B_{i}-B_{i-1})\otimes(u_{i}-u_{i-1})).
\end{align}
By \eqref{eq8} and \eqref{eq9}, we know that $(\bar{u}_q,\bar{B}_q,\mathring{\bar{R}}_q^u,\mathring{\bar{R}}_q^B)$ solves equation \eqref{Hall-MHD2} with pressure $\bar{p}_q$.

We are ready to prove \eqref{bar supp}. From the definition \eqref{def2}, one has $(\mathring{\bar{R}}_q^u,\mathring{\bar{R}}_q^B)(t)=0$ on the interval $[t_i+\tau_{q+1},t_{i+1}]$. Thus we only consider the case of $t\in[t_i,t_i+\tau_{q+1}]$. If $i$ or $i-1\in \mathscr{C}$, then $(t_i-2\tau_{q+1},t_i+3\tau_{q+1})\subset\mathscr{B}^{(q+1)}$, which implies ${\rm dist}(t,\mathscr{G}^{(q+1)})>2\tau_{q+1}$. Thus it's suffices to consider the case of $i\notin\mathscr{C}$ and $i-1\notin\mathscr{C}$, that is $(\mathring{R}_q^u,\mathring{R}_q^B)(t)=0$ on $[t_{i-2},t_{i+1}+\tau_{q+1}]$. This implies $(u_q,B_q)$ is an exact solution of the Hall-MHD equations on$[t_{i-2},t_{i+1}+\tau_{q+1}]$ and then by uniqueness, we have
$$(u_{i-1},B_{i-1})=(u_i,B_i)=(u_q,B_q),$$
which yields that $(\bar{u}_q,\bar{B}_q)=(u_q,B_q)$ is an exact solution of the Hall-MHD equations. Thus, we conclude that \eqref{bar supp}.

Next, we will prove estimates \eqref{barL2}--\eqref{barRdecay}. Using the definition \eqref{bardef}, we have
\begin{align*}
&\|\bar{u}_q\|_{L^2}\leq\sum_{i=0}^{n_{q+1}}\eta_i\|u_i\|_{L^2}\leq2\delta_0^{\frac{1}{2}}-\delta_q^{\frac{1}{2}},\\
&\|\bar{u}_q\|_{H^4}\leq\sum_{i=0}^{n_{q+1}}\eta_i\|u_i\|_{H^4}\leq2\lambda_q^5,\\
&\|\bar{u}_q-u_q\|_{L^2}=\left\|\sum_{i=0}^{n_{q+1}}\eta_i(u_i-u_q)\right\|_{L^2}\leq\sum_{i=0}^{n_{q+1}}\eta_i\|u_i-u_q\|_{L^2}
\leq\vartheta_{q+1}\lambda_q^5\leq\lambda_q^{-15}\delta_{q+1}^{\frac{1}{2}},
\end{align*}
and the same estimates could be obtained in a similar way for the magnetic field. Now we prove \eqref{bardecay}. Employing \eqref{eq7}, \eqref{cutoffest} and Proposition \ref{local solu}, we have
\begin{align*}
&\|\partial_t^MD^N\bar{u}_q\|_{L^\infty({\rm supp}\eta_i;{H^4})}\\
&\leq \|\partial_t^MD^N\left((1-\eta_i)u_{i-1}\right)\|_{L^\infty({\rm supp}\eta_i;{H^4})}
+\|\partial_t^MD^N\left(\eta_iu_i\right)|_{L^\infty({\rm supp}\eta_i;{H^4})}\\
&\lesssim\sum_{M'=0}^{M}\|\partial_t^{M-M'}\eta_i\|_{C^0}\left(\|\partial_t^{M'}D^Nu_i\|_{L^\infty({\rm supp}\eta_i;{H^4})}+\|\partial_t^{M'}D^Nu_{i-1}\|_{L^\infty({\rm supp}\eta_i;{H^4})}\right)\\
&\lesssim \sum_{M'=0}^{M}\tau_{q+1}^{-(M-M')}\vartheta_{q+1}^{-M'-\frac{N}{2\alpha}}\|u_q\|_{H^4}\\
&\lesssim \tau_{q+1}^{-M}\vartheta_{q+1}^{-\frac{N}{2\alpha}}\lambda_q^5
\lesssim \tau_{q+1}^{-M-N}\lambda_q^5,
\end{align*}
where we have used the definition of $\tau_{q+1}$ and $\vartheta_{q+1}$ in \eqref{vq} and \eqref{tq}. Notice that the magnetic field $\bar{B}_q$ has a same estimate.

Next, we prove estimate \eqref{barL1}. For the first term of $\mathring{\bar{R}}_q^u$, by estimates \eqref{chaL1}, \eqref{cutoffest}, we have
\begin{align*}
&\|\partial_t\eta_i\mathcal{R}(u_i-u_{i-1})\|_{L^1}\\
&\lesssim\|\partial_t\eta_i\mathcal{R}(u_i-u_{q})\|_{L^\infty({\rm supp}\eta_i;{L^1})}+\|\partial_t\eta_i\mathcal{R}(u_{i-1}-u_{q})\|_{L^{\infty}({\rm supp}\eta_i;{L^1})}\\
&\lesssim \tau_{q+1}^{-1}\vartheta_{q+1}\lambda_q^{-\frac{3}{4}\varepsilon_R}\delta_{q+1}\lesssim\lambda_q^{-\frac{1}{2}\varepsilon_R}\delta_{q+1}.
\end{align*}
For the second term of $\mathring{\bar{R}}_q^u$, by \eqref{chaL2} we get
\begin{align*}
&\|\eta_i(\eta_i-1)\left((u_{i}-u_{i-1})\mathring{\otimes}(u_{i}-u_{i-1})-(B_{i}-B_{i-1})\mathring{\otimes}(B_{i}-B_{i-1})\right)\|_{L^1}\\
&\lesssim\|(u_{i}-u_q)-(u_{i-1}-u_q)\|^2_{L^{\infty}({\rm supp}\eta_i;{L^2})}+\|(B_{i}-B_q)-(B_{i-1}-B_q)\|^2_{L^{\infty}({\rm supp}\eta_i;{L^2})}\\
&\lesssim(\vartheta_{q+1}\lambda_q^5)^{2}\lesssim(\lambda_q^{-15}\delta_{q+1}^{\frac{1}{2}})^{2}\leq\lambda_q^{-\frac{1}{2}\varepsilon_R}\delta_{q+1}.
\end{align*}
The first and second terms of $\mathring{\bar{R}}_q^B$ can be handled in a same way, and for the third term, by \eqref{chaL2p1},  we have
\begin{align*}
&\|\eta_i(\eta_i-1)\mathcal{R}{\rm curl}^{-1}{\rm div}((u_{i}-u_{i-1})\otimes(B_{i}-B_{i-1})-(B_{i}-B_{i-1})\otimes(u_{i}-u_{i-1}))\|_{L^p}\\
&\lesssim\left\||\nabla|^{-1}\mathbb{P}_{\neq0}((u_{i}-u_{i-1})\otimes(B_{i}-B_{i-1})-(B_{i}-B_{i-1})\otimes(u_{i}-u_{i-1}))\right\|_{L^{\infty}({\rm supp}\eta_i;{L^p})}\\
&\lesssim\|(u_{i}-u_q)-(u_{i-1}-u_q)\|_{L^{\infty}({\rm supp}\eta_i;{L^{2p}})}\|(B_{i}-B_q)-(B_{i-1}-B_q)\|_{L^{\infty}({\rm supp}\eta_i;{L^{2p}})}\\
&\lesssim\vartheta^2_{q+1}\lambda_q^{\frac{14}{11}(10-\frac{3}{p})}\leq \lambda_q^{-\frac{1}{2}\varepsilon_R}\delta_{q+1},
\end{align*}
for $1<p<2$. By the definition \eqref{def2}, we conclude that \eqref{barL1}.

Finally, we prove \eqref{barRdecay}. For the first term of $\mathring{\bar{R}}_q^u$, using Leibniz's rule, \eqref{cutoffest}, \eqref{decay2}, we have
\begin{align*}
&\left\|\partial_t^MD^N\left(\partial_t\eta_i\mathcal{R}(u_i-u_{i-1})\right)\right\|_{H^4}\\
&\leq\sum_{M'=0}^{M}\|\partial_t^{M-M'+1}\eta_i\|_{L^\infty}\|\partial_t^{M'}D^N(u_i-u_{i-1})\|_{L^{\infty}({\rm supp}\eta_i;H^4)}\\
&\lesssim\sum_{M'=0}^{M}\tau_{q+1}^{-(M-M'+1)}\vartheta_{q+1}^{-M'-\frac{N}{2\alpha}}\lambda_q^5\lesssim\tau_{q+1}^{-M-1}\vartheta_{q+1}^{-\frac{N}{2\alpha}}\lambda_q^5,
\end{align*}
and for the same reason, the second term of $\mathring{\bar{R}}_q^u$ can be estimated as
\begin{align*}
&\left\|\partial_t^MD^N\left(\eta_i(\eta_i-1)\left((u_{i}-u_{i-1})\mathring{\otimes}(u_{i}-u_{i-1})-(B_{i}-B_{i-1})\mathring{\otimes}(B_{i}-B_{i-1})\right)\right)\right\|_{H^4}\\
&\lesssim\sum_{M'=0}^{M}\tau_{q+1}^{-(M-M')}\\
&\;\left\|\partial_t^{M'}D^N\left((u_{i}-u_{i-1})\mathring{\otimes}(u_{i}-u_{i-1})-(B_{i}-B_{i-1})\mathring{\otimes}(B_{i}-B_{i-1})\right)\right\|_{L^{\infty}({\rm supp}\eta_i;H^4)}\\
&\lesssim\sum_{M'=0}^{M}\tau_{q+1}^{-(M-M')}\vartheta_{q+1}^{-M'-\frac{N}{2\alpha}}\lambda_q^{10}\lesssim\tau_{q+1}^{-M-1}\vartheta_{q+1}^{-\frac{N}{2\alpha}}\lambda_q^5.
\end{align*}
The same estimate for $\mathring{\bar{R}}_q^B$ can be deduced in a similar way, thus we omit the details. By the definition \eqref{def2}, we conclude that \eqref{barRdecay}.

\section{Convex integration step: the perturbation}\label{sec4}
\subsection{Intermittent jets}
\begin{Lemma}\label{lemma}
For $\alpha=0,1,2,3$, there exist disjoint subsets $\Lambda_\alpha\in \mathbb{S}^2\cap\mathbb{Q}^3$ and smooth functions $\gamma_\xi:\mathcal{N}\rightarrow\mathbb{R}$ such that
$$R=\sum_{\xi\in \Lambda_\alpha}\gamma^2_{\xi}(R)(\xi\otimes\xi)$$
for every symmetric matrix $R$ satisfying $|R-{\rm Id}|<\delta$, where $\delta>0$ is a small constant.
\end{Lemma}
We prove this lemma in Appendix \ref{sec6}. Since the index sets $\{\Lambda_\alpha\}_{\alpha=0,1,2,3}$ are finite, there exists $N_\Lambda\in\mathbb{N}$ such that
$$\{N_\Lambda\xi,N_\Lambda A_\xi,N_\Lambda\xi\times A_\xi\}\in N_\Lambda\mathbb{S}^2\cap\mathbb{N}^3,$$
for every $\xi\in \Lambda_\alpha$. The intermittent jets is defined in an analogous way of \cite{bcv}. The functions $\Phi,\phi,\psi$, and scaling versions $\Phi_{\ell_{\perp}},\phi_{\ell_{\perp}},\psi_{\ell_{\|}}$ are defined in \cite[Section 4.1]{bcv}. There exists a large real number $\lambda$ such that $\lambda\ell_{\perp}\in \mathbb{N}$. For $\xi\in\Lambda_0 $ or $\Lambda_1$, define
\begin{align*}
V_\xi&=\frac{1}{N_\Lambda^2\lambda^2}\psi_{\ell_{\|}}(N_\Lambda\ell_{\perp}\lambda(x\cdot\xi+\mu t))
\Phi_{\ell_{\perp}}(N_\Lambda\ell_{\perp}\lambda(x-\alpha_\xi)\cdot A_\xi,N_\Lambda\ell_{\perp}\lambda(x-\alpha_\xi)\cdot(\xi\times A_\xi))\xi,
\end{align*}
and for $\xi\in\Lambda_2 $ or $\Lambda_3$, define
\begin{align*}
\bar{V}_\xi=\frac{1}{N_\Lambda^2\lambda^2}\psi_{\ell_{\|}}(N_\Lambda\ell_{\perp}\lambda(x\cdot\xi+\bar{\mu} t))
\Phi_{\ell_{\perp}}(N_\Lambda\ell_{\perp}\lambda(x-\alpha_\xi)\cdot A_\xi,N_\Lambda\ell_{\perp}\lambda(x-\alpha_\xi)\cdot(\xi\times A_\xi))\xi,
\end{align*}
where the shifts $\alpha_\xi$ are chose to ensure $V_\xi$ and $\bar{V}_\xi$, $\xi\in \{\Lambda_\alpha\}_{\alpha=0,1,2,3}$, have disjoint supports. We let $\ell_{\perp}$ small enough for such shifts exist. For $\xi\in\Lambda_0 $ or $\Lambda_1$, the intermittent jets is defined as
\begin{align}\label{wxi}
W_\xi=\psi_{\ell_{\|}}(N_\Lambda\ell_{\perp}\lambda(x\cdot\xi+\mu t))
\phi_{\ell_{\perp}}(N_\Lambda\ell_{\perp}\lambda(x-\alpha_\xi)\cdot A_\xi,N_\Lambda\ell_{\perp}\lambda(x-\alpha_\xi)\cdot(\xi\times A_\xi))\xi,
\end{align}
and for $\xi\in\Lambda_2 $ or $\Lambda_3$, define
\begin{align}\label{wbarxi}
\bar{W}_\xi=\psi_{\ell_{\|}}(N_\Lambda\ell_{\perp}\lambda(x\cdot\xi+\bar{\mu} t))
\phi_{\ell_{\perp}}(N_\Lambda\ell_{\perp}\lambda(x-\alpha_\xi)\cdot A_\xi,N_\Lambda\ell_{\perp}\lambda(x-\alpha_\xi)\cdot(\xi\times A_\xi))\xi.
\end{align}
Immediately, one has that $W_\xi$ and $\bar{W}_\xi$ are $(\mathbb{T/\ell_{\perp}\lambda})^3$--periodic, with zero mean. Moreover, by the choice of $\alpha_\xi$, $W_\xi$ and $\bar{W}_\xi$, $\xi\in \{\Lambda_\alpha\}_{\alpha=0,1,2,3}$, have disjoint supports, that is,
\begin{align}\label{disjoint}
W_\xi\otimes W_{\xi'}=W_\xi\otimes \bar{W}_{\xi'}=\bar{W}_\xi\otimes \bar{W}_{\xi'}=0,\;\xi\neq\xi'\in\bigcup_{\alpha\in\{0,1,2,3\}}\Lambda_\alpha.
\end{align}
We introduce the following shorthand notation
\begin{align*}
&\psi_\xi=\psi_{\ell_{\|}}(N_\Lambda\ell_{\perp}\lambda(x\cdot\xi+\mu t)),\;\xi\in\Lambda_0\cup\Lambda_1,\\
&\bar{\psi}_\xi=\psi_{\ell_{\|}}(N_\Lambda\ell_{\perp}\lambda(x\cdot\xi+\bar{\mu} t)),\;\xi\in\Lambda_2\cup\Lambda_3,\\
&\Phi_\xi=\Phi_{\ell_{\perp}}(N_\Lambda\ell_{\perp}\lambda(x-\alpha_\xi)\cdot A_\xi,N_\Lambda\ell_{\perp}\lambda(x-\alpha_\xi)\cdot(\xi\times A_\xi)),\;\xi\in\Lambda_\alpha,\\
&\phi_\xi=\phi_{\ell_{\perp}}(N_\Lambda\ell_{\perp}\lambda(x-\alpha_\xi)\cdot A_\xi,N_\Lambda\ell_{\perp}\lambda(x-\alpha_\xi)\cdot(\xi\times A_\xi)),\;\xi\in\Lambda_\alpha,
\end{align*}
for $\alpha\in \{0,1,2,3\}$, and compute
\begin{align}
&\nabla\times\nabla\times V_\xi=W_\xi+W_\xi^c,\;where\;W_\xi^c=\frac{1}{N_\Lambda^2\lambda^2}\nabla\psi_\xi\times(\nabla\times(\Phi_\xi\xi)),\;\xi\in\Lambda_0\cup\Lambda_1,\label{divcorrector1}\\
&\nabla\times\nabla\times \bar{V}_\xi=\bar{W}_\xi+\bar{W}_\xi^c,\;where\;\bar{W}_\xi^c=
\frac{1}{N_\Lambda^2\lambda^2}\nabla\bar{\psi}_\xi\times(\nabla\times(\Phi_\xi\xi)),\;\xi\in\Lambda_2\cup\Lambda_3.\label{divcorrector2}
\end{align}
Thus, one has
$${\rm div}(W_\xi+W_\xi^c)={\rm div}(\bar{W}_\xi+\bar{W}_\xi^c)=0.$$
Similar to \cite[Section 4.1]{bcv}, using Lemma \ref{lemma}, one has
\begin{align}
\sum_{\xi\in\Lambda_\alpha}\gamma^2_{\xi}(R)\fint_{\mathbb{T}^3}W_\xi\otimes W_\xi dx=R,\quad\alpha=0,1,\label{eq10}\\
\sum_{\xi\in\Lambda_\alpha}\gamma^2_{\xi}(R)\fint_{\mathbb{T}^3}\bar{W}_\xi\otimes \bar{W}_\xi dx=R,\quad\alpha=2,3,\label{eq15}
\end{align}
for every symmetric matrix $R$ satisfying $|R-{\rm Id}|<\delta$.

By using scaling and Fubini's theorem, we have
\begin{align}
&\|\partial_t^M\nabla^N\psi_\xi\|_{L^p}\lesssim\ell_{\|}^{\frac{1}{p}-\frac{1}{2}}(\ell_{\perp}\ell_{\|}^{-1}\lambda)^N(\ell_{\perp}\ell_{\|}^{-1}\lambda\mu)^M,\;
\xi\in\Lambda_0\cup\Lambda_1,\label{est6}\\
&\|\partial_t^M\nabla^N\bar{\psi}_\xi\|_{L^p}\lesssim\ell_{\|}^{\frac{1}{p}-\frac{1}{2}}(\ell_{\perp}\ell_{\|}^{-1}\lambda)^N
(\ell_{\perp}\ell_{\|}^{-1}\lambda\bar{\mu})^M,\;\xi\in\Lambda_2\cup\Lambda_3,\label{est7}\\
&\|\nabla^N\phi_\xi\|_{L^p}+\|\nabla^N\Phi_\xi\|_{L^p}\lesssim\ell_{\perp}^{\frac{2}{p}-1}\lambda^N,\;\xi\in\Lambda_\alpha,\label{est8}\\
&\|\partial_t^M\nabla^NW_\xi\|_{L^p}+\lambda^2\|\partial_t^M\nabla^NV_\xi\|_{L^p}
\lesssim\ell_{\perp}^{\frac{2}{p}-1}\ell_{\|}^{\frac{1}{p}-\frac{1}{2}}\lambda^N(\ell_{\perp}\ell_{\|}^{-1}\lambda\mu)^M,\;
\xi\in\Lambda_0\cup\Lambda_1,\label{est9}\\
&\|\partial_t^M\nabla^N\bar{W}_\xi\|_{L^p}+\lambda^2\|\partial_t^M\nabla^N\bar{V}_\xi\|_{L^p}
\lesssim\ell_{\perp}^{\frac{2}{p}-1}\ell_{\|}^{\frac{1}{p}-\frac{1}{2}}\lambda^N(\ell_{\perp}\ell_{\|}^{-1}\lambda\bar{\mu})^M,\;
\xi\in\Lambda_2\cup\Lambda_3.\label{est10}
\end{align}
Since
$$(\xi\cdot\nabla)\psi_\xi=\mu^{-1}\partial_t\psi_\xi, \quad (\xi\cdot\nabla)\bar{\psi}_\xi=\bar{\mu}^{-1}\partial_t\bar{\psi}_\xi,$$
we have the following identities:
\begin{align}
&{\rm div}(W_\xi\otimes W_\xi)=2(W_\xi\cdot\nabla\psi_\xi)\phi_\xi\xi=\mu^{-1}\phi_\xi^2\partial_t\psi_\xi^2\xi,\quad\xi\in\Lambda_0\cup\Lambda_1,\label{eq11}\\
&{\rm div}(\bar{W}_\xi\otimes \bar{W}_\xi)=2(\bar{W}_\xi\cdot\nabla\bar{\psi}_\xi)\phi_\xi\xi=\bar{\mu}^{-1}\phi_\xi^2\partial_t\bar{\psi}_\xi^2\xi,\quad\xi\in\Lambda_2\cup\Lambda_3.\label{eq12}
\end{align}

\subsection{The perturbation}
Let $0\leq\tilde{\chi}_0,\tilde{\chi}\leq1$ are bump functions with support on $[0,4]$ and $[\frac{1}{4},4]$, respectively, and they form a partition of unity:
\begin{align}\label{eq13}
\tilde{\chi}_0^2(y)+\sum_{i\geq1}\tilde{\chi}_i^2(y)=1, \quad where \;\tilde{\chi}_i(y)=\tilde{\chi}(4^{-i}y),
\end{align}
for any $y>0$.
\subsubsection{The magnetic perturbations}
Define
\begin{align}\label{xib}
\chi_i^B(x,t)=\tilde{\chi}_i\left(\left\langle\frac{\mathring{\bar{R}}_q^B}{\lambda_q^{-\varepsilon_R/4}\delta_{q+1}}\right\rangle\right),
\end{align}
for all $i\geq0$. Here $\langle A\rangle=(1+|A|^2)^{\frac{1}{2}}$ and $|A|$ is the Euclidean norm of the matrix $A$. In view of \eqref{eq13}, one has $\sum_{i\geq0}|\chi_i^B|=1$. By the definition \eqref{xib} of $\chi_i^B$, we have
$${\rm supp}\chi_i^B(x,t)\subset \left\{(x,t):|\mathring{\bar{R}}_q^B|\leq4^{i+1}\lambda_q^{-\varepsilon_R/4}\delta_{q+1}\right\}.$$
Define
\begin{align}\label{defrhoi}
\rho_i=2\delta^{-1}4^{i+1}\lambda_q^{-\varepsilon_R/4}\delta_{q+1},
\end{align}
where $\delta$ is the constant introduced in Lemma \ref{lemma}, then we infer that $\rho_i^{-1}|\mathring{\bar{R}}_q^B|<\delta$ on the support of $\chi_i^B$ for all $i\geq0$.

For $i\geq0$ and $\xi\in\Lambda_{(i)+2}$, where $(i)=i\;{\rm mod}\;2$, define the coefficient function
\begin{align}\label{axiB}
a_\xi=\theta_B(t)\rho_i^{1/2}\chi_i^B(x,t)\gamma_\xi\left({\rm Id}-\frac{\mathring{\bar{R}}_q^B}{\rho_i}\right),
\end{align}
where $\theta_B(t):[0,T]\rightarrow[0,1]$ is a smooth temporal cut-off function with the following properties:
\begin{enumerate}
  \item $\theta_B(t)=1$ on $\left\{t:{\rm dist}(t,\mathscr{G}^{(q+1)})\geq2\tau_{q+1}\right\}$,
  \item $\theta_B(t)=0$ on $\left\{t:{\rm dist}(t,\mathscr{G}^{(q+1)})\leq\frac{3}{2}\tau_{q+1}\right\}$,
  \item $\|\theta_B\|_{C^M}\lesssim\tau_{q+1}^{-M}$, where the implicit constant depends on $M$.
\end{enumerate}
In view of the fact that $\mathscr{B}^{(q+1)}$ is the union of finite disjoint intervals of length $5\tau_{q+1}$, the choice of $\theta_B$ with property (3) holding is possible. From \eqref{bar supp}, we have that $\theta_B=1$ on ${\rm supp}(\mathring{\bar{R}}^u,\mathring{\bar{R}}^B)$. Thus, by using \eqref{eq15}, we obtain
\begin{align}\label{eq14}
\sum_{i\geq0}\sum_{\xi\in\Lambda_{(i)+2}}a^2_\xi\fint_{\mathbb{T}^3}\bar{W}_\xi\otimes \bar{W}_\xi dx
=\theta_B^2\sum_{i\geq0}\rho_i|\chi_i^B|^2{\rm Id}-\mathring{\bar{R}}^B.
\end{align}

Now we define the magnetic perturbations as follows. The principle part of $d_{q+1}$ is defined as
\begin{align}\label{dq+1p}
d_{q+1}^p=\sum_{i\geq0}\sum_{\xi\in\Lambda_{(i)+2}}a_\xi\bar{W}_\xi.
\end{align}
Noted that $\tilde{\chi}_i\tilde{\chi}_j=0$ if $|i-j|\geq2$, from \eqref{disjoint}, we know that the summands in $d_{q+1}^p$ have mutually disjoint supports. The incompressible corrector is defined as
\begin{align}\label{dq+1c}
d_{q+1}^c=\sum_{i\geq0}\sum_{\xi\in\Lambda_{(i)+2}}\nabla\times\left(\nabla a_\xi\times\bar{V}_\xi\right)
+\frac{1}{N_\Lambda^2\lambda_{q+1}^2}\nabla\left(a_\xi\bar{\psi}_\xi\right)\times\left(\nabla\times\left(\Phi_\xi\xi\right)\right).
\end{align}
Similar to \eqref{divcorrector2}, we have
\begin{align}\label{eq17}
d_{q+1}^p+d_{q+1}^c=\sum_{i\geq0}\sum_{\xi\in\Lambda_{(i)+2}}\nabla\times\nabla\times\left(a_\xi\bar{V}_\xi\right),
\end{align}
thus ${\rm div}(d_{q+1}^p+d_{q+1}^c)=0$. The temporal corrector $d_{q+1}^t$ is defined as
\begin{align}\label{dq+1t}
d_{q+1}^t=-\bar{\mu}^{-1}\sum_{i\geq0}\sum_{\xi\in\Lambda_{(i)+2}}\nabla\times\left( a_\xi^2\phi^2_\xi\bar{\psi}^2_\xi\xi\right).
\end{align}
Finally we define the magnetic increment $d_{q+1}$ and new magnetic field $B_{q+1}$ as
\begin{align}
d_{q+1}&=d_{q+1}^p+d_{q+1}^c+d_{q+1}^t,\label{dq+1}\\
B_{q+1}&=\bar{B}_{q}+d_{q+1}.\label{Bq+1}
\end{align}
By the construction, one has ${\rm div}d_{q+1}=0$ and $\int_{\mathbb{T}^3}d_{q+1}dx=0$. From the definition of $\theta_B$ and $d_{q+1}$, we know that if
$$t\in {\rm supp} d_{q+1}\subset {\rm supp}a_\xi\subset{\rm supp}\theta_B,$$
then ${\rm dist}(t,\mathscr{G}^{(q+1)})>\frac{3}{2}\tau_{q+1}$. Hence $B_{q+1}=\bar{B}_q$ on $\mathscr{G}^{(q+1)}$, combining \eqref{eq1} and (\romannumeral2) in Section \ref{sec2}, we know that $B_{q+1}=B_q$ on $\mathscr{G}^{(q)}$.

\subsubsection{The velocity perturbations}The construction of the velocity perturbations is similar to that of the MHD equations in \cite{lzz}. Firstly, define
\begin{align}\label{GB}
G^B=\sum_{i\geq0}\sum_{\xi\in\Lambda_{(i)+2}}a^2_\xi\fint_{\mathbb{T}^3}\bar{W}_\xi\otimes\bar{W}_\xi dx
\end{align}
and
\begin{align}\label{xiu}
\chi_i^u(x,t)=\tilde{\chi}_i\left(\left\langle\frac{\mathring{\bar{R}}_q^u-G^B}{\lambda_q^{-\varepsilon_R/4}\delta_{q+1}}\right\rangle\right),
\end{align}
for all $i\geq0$, and then by the definition of $\tilde{\chi}_i$ in \eqref{eq13}, we can infer that $\sum_{i\geq0}|\chi_i^u|^2=1$. By using the definition \eqref{xiu}, one has
$${\rm supp}\chi_i^u(x,t)\subset \left\{(x,t):|\mathring{\bar{R}}_q^u-G^B|\leq4^{i+1}\lambda_q^{-\varepsilon_R/4}\delta_{q+1}\right\}.$$
Then, we deduce that $\rho_i^{-1}|\mathring{\bar{R}}_q^u-G^B|<\delta$ on the support of $\chi_i^u$ for all $i\geq0$.

For $i\geq0$ and $\xi\in\Lambda_{(i)}$, define the coefficient function
\begin{align}\label{axiu}
a_\xi=\theta_u(t)\rho_i^{1/2}\chi_i^u(x,t)\gamma_\xi\left({\rm Id}-\frac{\mathring{\bar{R}}_q^u-G^B}{\rho_i}\right),
\end{align}
where $\theta_u(t):[0,T]\rightarrow[0,1]$ is a smooth temporal cut-off function with the following properties:
\begin{enumerate}
  \item $\theta_u(t)=1$ on $\left\{t:{\rm dist}(t,\mathscr{G}^{(q+1)})\geq\frac{3}{2}\tau_{q+1}\right\}$,
  \item $\theta_u(t)=0$ on $\left\{t:{\rm dist}(t,\mathscr{G}^{(q+1)})\leq\tau_{q+1}\right\}$,
  \item $\|\theta_u\|_{C^M}\lesssim\tau_{q+1}^{-M}$, where the implicit constant depends on $M$.
\end{enumerate}
In view of the fact that $\mathscr{B}^{(q+1)}$ is the union of finite disjoint intervals of length $5\tau_{q+1}$, the choice of $\theta_u$ with property (3) holding is possible. By the definition \eqref{GB}, one has
$${\rm supp}G^B\subset\left\{t:{\rm dist}(t,\mathscr{G}^{(q+1)})\geq\frac{3}{2}\tau_{q+1}\right\},$$
which together with \eqref{bar supp} yields that
\begin{align*}
{\rm supp}(\mathring{\bar{R}}_q^u-G^B)\subset {\rm supp}\mathring{\bar{R}}_q^u\cup{\rm supp}G^B\subset\left\{t:{\rm dist}(t,\mathscr{G}^{(q+1)})\geq\frac{3}{2}\tau_{q+1}\right\}.
\end{align*}
Thus by the definition of $\theta_u$, we know $\theta_u=1$ on ${\rm supp}(\mathring{\bar{R}}_q^u-G^B)$.
From \eqref{eq10} and definition \eqref{axiu}, we compute
\begin{align}\label{eq16}
\sum_{i\geq0}\sum_{\xi\in\Lambda_{(i)}}a^2_\xi\fint_{\mathbb{T}^3}W_\xi\otimes W_\xi dx
=\theta_u^2\sum_{i\geq0}\rho_i|\chi_i^u|^2{\rm Id}-(\mathring{\bar{R}}^u_q-G^B).
\end{align}

Now we define the velocity perturbations as follows. The principal part of $w_{q+1}$ is defined as
\begin{align}\label{wq+1p}
w_{q+1}^p=\sum_{i\geq0}\sum_{\xi\in\Lambda_{(i)}}a_\xi W_\xi.
\end{align}
Noted that $\tilde{\chi}_i\tilde{\chi}_j=0$ if $|i-j|\geq2$, from \eqref{disjoint}, we know that the summands in $w_{q+1}^p$ have mutually disjoint supports. The incompressible corrector is defined as
\begin{align}\label{wq+1c}
w_{q+1}^c=\sum_{i\geq0}\sum_{\xi\in\Lambda_{(i)}}\nabla\times\left(\nabla a_\xi\times V_\xi\right)
+\frac{1}{N_\Lambda^2\lambda_{q+1}^2}\nabla\left(a_\xi\psi_\xi\right)\times\left(\nabla\times\left(\Phi_\xi\xi\right)\right).
\end{align}
Similar to \eqref{divcorrector1}, we have
\begin{align}\label{eq18}
w_{q+1}^p+w_{q+1}^c=\sum_{i\geq0}\sum_{\xi\in\Lambda_{(i)}}\nabla\times\nabla\times\left(a_\xi V_\xi\right).
\end{align}
Thus, ${\rm div}(w_{q+1}^p+w_{q+1}^c)=0$. The temporal corrector $w_{q+1}^t$ is defined by
\begin{align}\label{wq+1t}
w_{q+1}^t=-\mu^{-1}\sum_{i\geq0}\sum_{\xi\in\Lambda_{(i)}}\mathbb{P}_{H}\mathbb{P}_{\neq0}\left(a_\xi^2\phi^2_\xi\psi^2_\xi\xi\right)
+\bar{\mu}^{-1}\sum_{i\geq0}\sum_{\xi\in\Lambda_{(i)+2}}\mathbb{P}_{H}\mathbb{P}_{\neq0}\left( a_\xi^2\phi^2_\xi\bar{\psi}^2_\xi\xi\right).
\end{align}
Finally, we define the magnetic increment $w_{q+1}$ and new velocity field $u_{q+1}$ as
\begin{align}
w_{q+1}&=w_{q+1}^p+w_{q+1}^c+w_{q+1}^t,\label{wq+1}\\
u_{q+1}&=\bar{u}_{q}+w_{q+1}.\label{uq+1}
\end{align}
By the construction, one has ${\rm div}w_{q+1}=0$ and $\int_{\mathbb{T}^3}w_{q+1}dx=0$. From the definition of $\theta_u$, $\theta_B$ and $w_{q+1}$, we know that
$${\rm supp} w_{q+1}\subset{\rm supp}\theta_u\cup {\rm supp}\theta_B\subset
\left\{t:{\rm dist}(t,\mathscr{G}^{(q+1)})\geq\tau_{q+1}\right\}.$$
Hence $u_{q+1}=\bar{u}_q$ on $\mathscr{G}^{(q+1)}$, combining \eqref{eq1} and (\romannumeral2) in Section \ref{sec2}, we know that $u_{q+1}=u_q$ on $\mathscr{G}^{(q)}$.

\subsubsection{Estimates of the perturbation} This part can be referred to \cite[Section 4.2.3]{bcv}, and here we omit some details.

\begin{Lemma}\label{lemma1}
For $q\geq0$, there exists $\bar{i}_{max}(q)\geq0$ and $i_{max}(q)\geq0$ such that
\begin{align}\label{=0}
\chi_i^B=0\;for\;all\;i>\bar{i}_{max}\quad and\quad \chi_i^u=0\;for\;all\;i>i_{max},
\end{align}
moreover, we have
\begin{align}
&4^{\bar{i}_{max}}\leq\left(\lambda_1^{-\frac{9}{2}\beta}\lambda_{q}^{6+\frac{4\rho}{\rho-1}}\right)^{\frac{6}{11}},\label{est13}\\
&4^{i_{max}}\leq\lambda_q^{\frac{\varepsilon_R}{32}+\frac{6}{11}(6+\frac{4\rho}{\rho-1})},\label{est11}\\
&\bar{i}_{max},i_{max}\leq\lambda_q^{\frac{\varepsilon_R}{32}}.\label{est12}
\end{align}
For $0\leq i\leq\bar{i}_{max}$ and $0\leq j\leq i_{max}$
\begin{align}
\|\chi_i^B\|_{L^2}&\lesssim2^{-i}, \label{est14}\\
\|\chi_i^B\|_{C^N_{t,x}}&\lesssim\tau_{q+1}^{-3N},\label{est15}\\
\|\chi_j^u\|_{L^2}&\lesssim2^{-j}\lambda_q^{\frac{\varepsilon_R}{64}}, \label{est16}\\
\|\chi_j^u\|_{C^N_{t,x}}&\lesssim\tau_{q+1}^{-3N},\label{est17}
\end{align}
for all $N\geq1$.
\end{Lemma}
\begin{proof}
It follows from the Gagliardo--Nirenberg--Sobolev inequalities, Proposition \ref{glue prop}, \eqref{RestL1} and \eqref{RestH4} that
\begin{align}\label{est22}
\left\langle\lambda_q^{\frac{\varepsilon_R}{4}}\delta_{q+1}^{-1}\mathring{\bar{R}}^B_q\right\rangle&\leq1
+\lambda_q^{\frac{\varepsilon_R}{4}}\delta_{q+1}^{-1}\|\mathring{\bar{R}}^B_q\|_{L^\infty}\lesssim1
+\lambda_q^{\frac{\varepsilon_R}{4}}\delta_{q+1}^{-1}\|\mathring{\bar{R}}^B_q\|_{L^1}^{\frac{5}{11}}\|\mathring{\bar{R}}^B_q\|_{H^4}^{\frac{6}{11}}\notag\\
&\leq C\left(\lambda_q^{\frac{\varepsilon_R}{4}}\delta_{q+1}^{-1}\right)^{\frac{6}{11}}(\tau_{q+1}^{-1}\lambda_q^5)^{\frac{6}{11}}.
\end{align}
Thus, we define
\begin{align}\label{ibar}
\bar{i}_{max}(q)={\rm max}\left\{i\in\mathbb{N}:4^{i-1}\leq C\left(\lambda_q^{\frac{\varepsilon_R}{4}}\delta_{q+1}^{-1}\right)^{\frac{6}{11}}(\tau_{q+1}^{-1}\lambda_q^5)^{\frac{6}{11}}\right\}.
\end{align}
Therefore, for any $i>\bar{i}_{max}$,
$$4^{i-1}>C\left(\lambda_q^{\frac{\varepsilon_R}{4}}\delta_{q+1}^{-1}\right)^{\frac{6}{11}}(\tau_{q+1}^{-1}\lambda_q^5)^{\frac{6}{11}}
\geq\left\langle\lambda_q^{\frac{\varepsilon_R}{4}}\delta_{q+1}^{-1}\mathring{\bar{R}}^B_q\right\rangle,$$
which means that $\chi_i^B(x,t)=0$. Moreover, by \eqref{ibar} and the definition of $\tau_{q+1}$ in \eqref{tq}, one has
\begin{align*}
4^{\bar{i}_{max}-1}&\leq C\left(\lambda_q^{\frac{\varepsilon_R}{4}}\delta_{q+1}^{-1}\right)^{\frac{6}{11}}(\tau_{q+1}^{-1}\lambda_q^5)^{\frac{6}{11}}
\lesssim\left(\lambda_1^{-\frac{9}{2}\beta}\lambda_q^{\frac{4\rho}{\rho-1}+5+3\beta b+\frac{\varepsilon_R}{2}}\right)^{\frac{6}{11}}\\
&\leq\frac{1}{4}\left(\lambda_1^{-\frac{9}{2}\beta}\lambda_q^{\frac{4\rho}{\rho-1}+6}\right)^{\frac{6}{11}},
\end{align*}
which yields that \eqref{est13}. Therefore, $\bar{i}_{max}\leq\lambda_q^{\frac{\varepsilon_R}{32}}$ as $a$ large enough depends on $\varepsilon_R$. By the definition of $\chi_i^B$, Chebyshev's inequality and estimate \eqref{barL1}, we have
\begin{align*}
\|\chi_i^B\|_{L^1_x}&\leq\sup_{t}\left|\left\{x:4^{i-1}\leq
\left\langle\lambda_q^{\frac{\varepsilon_R}{4}}\delta_{q+1}^{-1}\mathring{\bar{R}}^B_q\right\rangle\leq4^{i+1}\right\}\right|\\
&\leq\sup_{t}\left|\left\{x:4^{i-2}\leq\lambda_q^{\frac{\varepsilon_R}{4}}\delta_{q+1}^{-1}|\mathring{\bar{R}}^B_q|\right\}\right|\\
&\lesssim4^{-i}\lambda_q^{\frac{\varepsilon_R}{4}}\delta_{q+1}^{-1}\left\|\mathring{\bar{R}}^B_q\right\|_{L^1_x}\lesssim4^{-i}.
\end{align*}
Then, the interpolation inequality yields that \eqref{est14}. Next, we prove \eqref{est15}. From \eqref{barRdecay} and \cite[Proposition C.1]{bdis}, we have
\begin{align}\label{est19}
\|\chi_i^B\|_{C^N_{t,x}}&\lesssim\left\|\left\langle\lambda_q^{\frac{\varepsilon_R}{4}}\delta_{q+1}^{-1}\mathring{\bar{R}}^B_q\right\rangle\right\|_{C^N_{t,x}}
+\left\|\left\langle\lambda_q^{\frac{\varepsilon_R}{4}}\delta_{q+1}^{-1}\mathring{\bar{R}}^B_q\right\rangle\right\|_{C^1_{t,x}}^N\notag\\
&\lesssim1+\lambda_q^{\frac{\varepsilon_R}{4}}\delta_{q+1}^{-1}\left\|\mathring{\bar{R}}^B_q\right\|_{C^N_{t,x}}
+\left(\lambda_q^{\frac{\varepsilon_R}{4}}\delta_{q+1}^{-1}\left\|\mathring{\bar{R}}^B_q\right\|_{C^1_{t,x}}\right)^N\notag\\
&\lesssim1+\lambda_q^{\frac{\varepsilon_R}{4}}\delta_{q+1}^{-1}\tau_{q+1}^{-N-1}\lambda_q^5
+\left(\lambda_q^{\frac{\varepsilon_R}{4}}\delta_{q+1}^{-1}\tau_{q+1}^{-2}\lambda_q^5\right)^N\notag\\
&\lesssim\left(\lambda_q^{\frac{\varepsilon_R}{4}}\delta_{q+1}^{-1}\tau_{q+1}^{-2}\lambda_q^5\right)^N,
\end{align}
then we get \eqref{est15} immediately by \eqref{vqtqest}. Therefore, by \eqref{est14}, \eqref{ibar}, and the definition of $\rho_i$ in \eqref{defrhoi}, for $\xi\in\Lambda_{(i)+2}$, we deduce that
\begin{align}
\|a_\xi\|_{L^2_x}&\lesssim\rho_i^{\frac{1}{2}}\|\chi_i^B\|_{L^2_x}\lesssim\left(\lambda_q^{-\frac{\varepsilon_R}{4}}\delta_{q+1}4^i\right)^{\frac{1}{2}}2^{-i}
\lesssim\lambda_q^{-\frac{\varepsilon_R}{8}}\delta_{q+1}^{\frac{1}{2}},\label{est18}\\
\|a_\xi\|_{L^\infty_x}&\lesssim\rho_i^{\frac{1}{2}}\lesssim\left(\lambda_q^{-\frac{\varepsilon_R}{4}}\delta_{q+1}4^{\bar{i}_{max}}\right)^{\frac{1}{2}}
\lesssim\left(\lambda_q^{-\frac{\varepsilon_R}{4}}\delta_{q+1}\right)^{\frac{5}{22}}\left(\tau_{q+1}^{-1}\lambda_q^5\right)^{\frac{3}{11}},\label{est20}
\end{align}
which implies
\begin{align}\label{est21}
\|G^B\|_{L^\infty}\lesssim\bar{i}_{max}\|a_\xi\|_{L^\infty}^2\lesssim\lambda_q^{\frac{\varepsilon_R}{32}}
\left(\lambda_q^{-\frac{\varepsilon_R}{4}}\delta_{q+1}\right)^{\frac{5}{11}}\left(\tau_{q+1}^{-1}\lambda_q^5\right)^{\frac{6}{11}}.
\end{align}
By \eqref{est21}, similar to estimate \eqref{est22}, one has
\begin{align*}
\left\langle\frac{\mathring{\bar{R}}_q^u-G^B}{\lambda_q^{-\varepsilon_R/4}\delta_{q+1}}\right\rangle&\lesssim1+\lambda_q^{\varepsilon_R/4}\delta_{q+1}^{-1}
\left(\|\mathring{\bar{R}}_q^u\|_{L^\infty}+\|G^B\|_{L^\infty}\right)\\
&\lesssim\lambda_q^{\frac{\varepsilon_R}{32}}\left(\lambda_q^{\frac{\varepsilon_R}{4}}\delta_{q+1}^{-1}\right)^{\frac{6}{11}}
\left(\tau_{q+1}^{-1}\lambda_q^5\right)^{\frac{6}{11}}\\
&\lesssim\lambda_q^{\frac{\varepsilon_R}{32}}\left(\lambda_1^{-\frac{9}{2}\beta}\lambda_q^{\frac{4\rho}{\rho-1}+5+\frac{\varepsilon_R}{2}+3\beta b}\right)^{\frac{6}{11}}\\
&\leq\frac{1}{4}\lambda_q^{\frac{\varepsilon_R}{32}+\frac{6}{11}(\frac{4\rho}{\rho-1}+6)},
\end{align*}
where we have used the definition of $\tau_{q+1}$ and $\delta_{q+1}$. Thus, we define
\begin{align}\label{i}
i_{max}={\rm max}\left\{i\in\mathbb{N}:4^{i-1}\leq\frac{1}{4}\lambda_q^{\frac{\varepsilon_R}{32}+\frac{6}{11}(\frac{4\rho}{\rho-1}+6)}\right\},
\end{align}
then for any $i>i_{max}$,
$$4^{i-1}>\frac{1}{4}\lambda_q^{\frac{\varepsilon_R}{32}+\frac{6}{11}(\frac{4\rho}{\rho-1}+6)}
\geq\left\langle\frac{\mathring{\bar{R}}_q^u-G^B}{\lambda_q^{-\varepsilon_R/4}\delta_{q+1}}\right\rangle,$$
which implies $\chi_i^u=0$. By the definition \eqref{i}, we have
$$4^{i_{max}}\leq\lambda_q^{\frac{\varepsilon_R}{32}+\frac{6}{11}(\frac{4\rho}{\rho-1}+6)},$$
and $i_{max}\leq\lambda_q^{\frac{\varepsilon_R}{32}}$ as $a$ large enough depending on $\varepsilon_R$. We complete the proof of \eqref{est11} and \eqref{est12}.

Now we prove \eqref{est16}. Applying the definition of $G^B$ in \eqref{GB} and estimate \eqref{est18}, one has
\begin{align*}
\|G^B\|_{L^1}\lesssim\sum_{i\geq0}\sum_{\xi\in\Lambda_{(i)+2}}\|a_\xi\|_{L^2_x}^2\lesssim i_{max}\lambda_q^{-\frac{\varepsilon_R}{4}}\delta_{q+1}
\lesssim\lambda_q^{-\frac{7\varepsilon_R}{32}}\delta_{q+1}.
\end{align*}
Similar to the proof of \eqref{est14}, for $0\leq j\leq i_{max}$, we have
\begin{align*}
\|\chi_j^u\|_{L^1_x}&\leq\sup_{t}\left|\left\{x:4^{j-2}\leq\lambda_q^{\frac{\varepsilon_R}{4}}\delta_{q+1}^{-1}|\mathring{\bar{R}}^u_q-G^B|\right\}\right|\\
&\lesssim4^{-j}\lambda_q^{\frac{\varepsilon_R}{4}}\delta_{q+1}^{-1}\left(\lambda_q^{-\frac{\varepsilon_R}{4}}\delta_{q+1}
+\lambda_q^{-\frac{7\varepsilon_R}{32}}\delta_{q+1}\right)\lesssim4^{-j}\lambda_q^{\frac{\varepsilon_R}{32}}.
\end{align*}
Then the interpolation inequality yields that \eqref{est16}. By using \eqref{GB}, \eqref{ibar}, \eqref{est12} and \eqref{est19}, we have
\begin{align}\label{est24}
\|G^B\|_{C^N_{t,x}}&\lesssim\bar{i}_{max}\|a^2_\xi\|_{C^N_{t,x}}\lesssim\lambda_q^{\frac{\varepsilon_R}{32}}\rho_i\||\chi_i^B|^2\|_{C^N_{t,x}}\notag\\
&\lesssim\lambda_q^{\frac{\varepsilon_R}{32}}\lambda_q^{-\frac{\varepsilon_R}{4}}\delta_{q+1}4^{\bar{i}_{max}}
\left(\lambda_q^{\frac{\varepsilon_R}{4}}\delta_{q+1}^{-1}\tau_{q+1}^{-2}\lambda_q^5\right)^N\notag\\
&\lesssim\lambda_q^{\frac{\varepsilon_R}{32}}\left(\lambda_q^{-\frac{\varepsilon_R}{4}}\delta_{q+1}\right)^{\frac{5}{11}}(\tau_{q+1}^{-1}\lambda_q^5)^{\frac{6}{11}}
\left(\lambda_q^{\frac{\varepsilon_R}{4}}\delta_{q+1}^{-1}\tau_{q+1}^{-2}\lambda_q^5\right)^N.
\end{align}
Similar to \eqref{est19}, using \eqref{est24} and the fact
$$\|\mathring{\bar{R}}^u_q\|_{C^N_{t,x}}\lesssim\tau_{q+1}^{-N-1}\lambda_q^5,$$
which deduced by \eqref{barRdecay}, we have
\begin{align}\label{est23}
\|\chi_j^u\|_{C^N_{t,x}}&\lesssim1+\lambda_q^{\frac{\varepsilon_R}{4}}\delta_{q+1}^{-1}\left\|\mathring{\bar{R}}^u_q-G^B\right\|_{C^N_{t,x}}
+\left(\lambda_q^{\frac{\varepsilon_R}{4}}\delta_{q+1}^{-1}\left\|\mathring{\bar{R}}^u_q-G^B\right\|_{C^1_{t,x}}\right)^N\notag\\
&\lesssim\lambda_q^{\frac{\varepsilon_R}{32}}\left(\lambda_q^{\frac{\varepsilon_R}{4}}\delta_{q+1}^{-1}\right)^{\frac{6}{11}}(\tau_{q+1}^{-1}\lambda_q^5)^{\frac{6}{11}}
\left(\lambda_q^{\frac{\varepsilon_R}{4}}\delta_{q+1}^{-1}\tau_{q+1}^{-2}\lambda_q^5\right)^N\notag\\
&\;\;+\left(\lambda_q^{\frac{\varepsilon_R}{32}}\left(\lambda_q^{\frac{\varepsilon_R}{4}}\delta_{q+1}^{-1}\right)^{\frac{17}{11}}(\tau_{q+1}^{-1}\lambda_q^5)^{\frac{6}{11}}
\tau_{q+1}^{-2}\lambda_q^5\right)^N\notag\\
&\leq\left(\lambda_q^8\tau_{q+1}^{-2-\frac{6}{11}}\right)^N\leq\tau_{q+1}^{-3},
\end{align}
for $0\leq j\leq i_{max}$, where we have used the fact that $\varepsilon_R\ll1$, $\beta b\ll1$ and \eqref{vqtqest} in the last two inequalities.
\end{proof}
\begin{Lemma}\label{lemma2}
Let $N\geq1$. For $\xi\in \Lambda_{(i)+2}$, we have
\begin{align}
\|a_\xi\|_{L^2_x}&\lesssim\lambda_q^{-\varepsilon_R/8}\delta_{q+1}^{1/2},\label{est25}\\
\|a_\xi\|_{L^\infty_x}&\lesssim2^{\bar{i}_{max}}\lambda_q^{-\varepsilon_R/8}\delta_{q+1}^{1/2}\leq\left(\tau_{q+1}^{-1}\lambda_q^5\right)^{\frac{3}{11}},\label{est26}\\
\|a_\xi\|_{C^N_{t,x}}&\lesssim\tau_{q+1}^{-3N}.\label{est27}
\end{align}
For $\xi\in \Lambda_{(i)}$, we have
\begin{align}
\|a_\xi\|_{L^2_x}&\lesssim\lambda_q^{-\varepsilon_R/16}\delta_{q+1}^{1/2},\label{est28}\\
\|a_\xi\|_{L^\infty_x}&\lesssim2^{i_{max}}\lambda_q^{-\varepsilon_R/8}\delta_{q+1}^{1/2}\leq\left(\tau_{q+1}^{-1}\lambda_q^6\right)^{\frac{3}{11}},\label{est29}\\
\|a_\xi\|_{C^N_{t,x}}&\lesssim\tau_{q+1}^{-3N-1}.\label{est30}
\end{align}
\end{Lemma}
\begin{proof}
The estimates \eqref{est25} and \eqref{est26} can be deduced by \eqref{est18} and \eqref{est20}. Now we prove \eqref{est27}. By the definition of $\bar{i}_{max}$ in \eqref{ibar} and estimate \eqref{est19}, for $\xi\in \Lambda_{(i)+2}$, we have
\begin{align*}
\|a_\xi\|_{C^N_{t,x}}&\lesssim\rho_i^{\frac{1}{2}}\|\chi_i^B\|_{C^N_{t,x}}\lesssim\left(\lambda_q^{-\varepsilon_R/4}\delta_{q+1}4^{\bar{i}_{max}}\right)^{\frac{1}{2}}
\left(\lambda_q^{\frac{\varepsilon_R}{4}}\delta_{q+1}^{-1}\tau_{q+1}^{-2}\lambda_q^5\right)^N\\
&\lesssim\left(\lambda_q^{-\frac{\varepsilon_R}{4}}\delta_{q+1}\right)^{\frac{5}{22}}\left(\tau_{q+1}^{-1}\lambda_q^5\right)^{\frac{3}{11}}
\left(\lambda_q^{\frac{\varepsilon_R}{4}}\delta_{q+1}^{-1}\tau_{q+1}^{-2}\lambda_q^5\right)^N\\
&\lesssim\left(\lambda_q^{\frac{\varepsilon_R}{4}}\delta_{q+1}^{-1}\lambda_q^{5+\frac{15}{11}}\tau_{q+1}^{-2-\frac{3}{11}}\right)^N\leq\tau_{q+1}^{-3N}.
\end{align*}

For $\xi\in \Lambda_{(i)}$, by the definition of $a_\xi$ in \eqref{axiu} and estimate \eqref{est16}, we infer that
\begin{align*}
\|a_\xi\|_{L^2}\lesssim\rho_i^{\frac{1}{2}}\|\chi_i^u\|_{L^2}\lesssim\left(\lambda_q^{-\varepsilon_R/4}\delta_{q+1}4^{i}\right)^{\frac{1}{2}}2^{-i}\lambda_q^{\varepsilon_R/64}
\leq\lambda_q^{-\varepsilon_R/16}\delta_{q+1}^{1/2},
\end{align*}
then \eqref{est28} holds. For getting \eqref{est29}, by using \eqref{est11} and the definition of $\tau_{q+1}$ in \eqref{tq}, we have
\begin{align}\label{est31}
\|a_\xi\|_{L^\infty}\lesssim\rho_i^{\frac{1}{2}}\lesssim\left(\lambda_q^{-\varepsilon_R/4}\delta_{q+1}4^{i_{max}}\right)^{\frac{1}{2}}
\leq\lambda_q^{\frac{\varepsilon_R}{64}+\frac{3}{11}(6+\frac{4\rho}{\rho-1})}\lambda_q^{-\frac{\varepsilon_R}{8}}\delta_{q+1}^{1/2}
\leq\left(\lambda_q^6\tau_{q+1}^{-1}\right)^{\frac{3}{11}}.
\end{align}
Finally, we prove \eqref{est30}. By the estimate of $\rho_i$ in \eqref{est31} and estimate \eqref{est17}, one has
\begin{align*}
\|a_\xi\|_{C^N_{t,x}}\lesssim\rho_i^{\frac{1}{2}}\|\chi_i^u\|_{C^N_{t,x}}\lesssim\left(\lambda_q^6\tau_{q+1}^{-1}\right)^{\frac{3}{11}}\tau_{q+1}^{-3N}
\leq\tau_{q+1}^{-3N-1}.
\end{align*}
\end{proof}

\begin{Proposition}\label{prop}
The perturbation has the following bounds:
\begin{align}
\|(d_{q+1}^p,w_{q+1}^p)\|_{L^2}&\leq\frac{1}{2}\delta_{q+1}^{\frac{1}{2}},\label{est32}\\
\|(d_{q+1}^p,w_{q+1}^p)\|_{W^{N,p}}&\lesssim\lambda_q^{\frac{\varepsilon_R}{32}}
\tau_{q+1}^{-1}\ell_{\perp}^{\frac{2}{p}-1}\ell_{\|}^{\frac{1}{p}-\frac{1}{2}}\lambda_{q+1}^N,\label{est33}\\
\|(d_{q+1}^c,w_{q+1}^c)\|_{W^{N,p}}&\lesssim\lambda_q^{\frac{\varepsilon_R}{32}}
\tau_{q+1}^{-1}\ell_{\perp}^{\frac{2}{p}-1}\ell_{\|}^{\frac{1}{p}-\frac{1}{2}}\lambda_{q+1}^N\lambda_{q+1}^{-\frac{5}{6}(\frac{5}{4}-\alpha)},\label{est34}\\
\|d_{q+1}^t\|_{W^{N,p}}&\lesssim\lambda_q^{\frac{\varepsilon_R}{32}}
\ell_{\perp}^{\frac{2}{p}-1}\ell_{\|}^{\frac{1}{p}-\frac{1}{2}}\lambda_{q+1}^N\lambda_{q+1}^{-\frac{5}{12}(\frac{5}{4}-\alpha)},\label{est35}\\
\|w_{q+1}^t\|_{W^{N,p}}&\lesssim\lambda_q^{\frac{\varepsilon_R}{32}}
\tau_{q+1}^{-2}\ell_{\perp}^{\frac{2}{p}-1}\ell_{\|}^{\frac{1}{p}-\frac{1}{2}}\lambda_{q+1}^N\lambda_{q+1}^{-\frac{1}{2}(\frac{5}{4}-\alpha)},\label{est36}
\end{align}
for $N\in\mathbb{N}$ and $p>1$.
\end{Proposition}
\begin{proof}
We only give the details of estimates \eqref{est32}--\eqref{est34} for $d_{q+1}$, since those estimates of $w_{q+1}$ can be obtained in a same way. For any $\xi\in \Lambda_{(i)+2}\cup\Lambda_{(i)}$, utilizing \eqref{est25}, \eqref{est27}, \eqref{est28} and \eqref{est30}, we have
\begin{align}\label{est37}
\|D^Na_\xi\|_{L^2}\lesssim\lambda_q^{-\frac{\varepsilon_R}{16}}\delta_{q+1}^{\frac{1}{2}}\tau_{q+1}^{-5N},
\end{align}
for $N\geq0$. Recall that $W_\xi$ and $\bar{W}_{\xi}$ are $(\mathbb{T}/(\lambda_{q+1}\ell_\perp))^3$ periodic. In order to apply \cite[Lemma 4.5]{bcv} with $\lambda=\tau_{q+1}^{-5}$ and $\kappa=\lambda_{q+1}\ell_\perp$, by the definitions \eqref{tq} and \eqref{para}, we compute
\begin{align*}
\lambda\kappa^{-1}&=\tau_{q+1}^{-5}\lambda_{q+1}^{-\frac{25-20\alpha}{24}}=\delta_{q+1}^{-\frac{5}{2}}\lambda_q^{5(\frac{4\rho}{\rho-1}
+\frac{\varepsilon_R}{4})-\frac{5}{24}b(5-4\alpha)}\\
&=\lambda_1^{-\frac{15}{2}\beta}\lambda_q^{5\beta b+5(\frac{4\rho}{\rho-1}+\frac{\varepsilon_R}{4})-\frac{5}{24}b(5-4\alpha)}
\leq\lambda_q^{5\left(\beta b+\frac{4\rho}{\rho-1}+\frac{1}{4}-\frac{1}{24}b(5-4\alpha)\right)},
\end{align*}
where we have used the assumption $\varepsilon_R\ll1$. Here we require that the parameters satisfy
\begin{align}\label{paracon}
\beta<\frac{5-4\alpha}{100},\quad \frac{1}{b}<\frac{5-4\alpha}{100},\quad \frac{4\rho}{\rho-1}<\frac{b}{50}(5-4\alpha),
\end{align}
then one has $\lambda\kappa^{-1}\ll1$. By using \cite[Lemma 4.5]{bcv} with $C_f=C\lambda_q^{-\frac{\varepsilon_R}{16}}\delta_{q+1}^{\frac{1}{2}}$ and \eqref{est10}, we conclude that
\begin{align*}
\|d_{q+1}^p\|_{L^2}\lesssim\lambda_q^{\frac{\varepsilon_R}{32}}\lambda_q^{-\frac{\varepsilon_R}{16}}\delta_{q+1}^{\frac{1}{2}}\|\bar{W}_{\xi}\|_{L^2}
\leq\frac{1}{2}\delta_{q+1}^{\frac{1}{2}}.
\end{align*}
Now we prove \eqref{est33}. Employing Leibniz's rule, \eqref{est10} and \eqref{est27}, we have
\begin{align*}
\|d_{q+1}^p\|_{W^{N,p}}&\lesssim\sum_{i}\sum_{\xi\in\Lambda_{(i)+2}}\sum_{N'=0}^N\|a_\xi\|_{C^{N-N'}}\|\bar{W}_\xi\|_{W^{N',p}}\\
&\lesssim\lambda_q^{\frac{\varepsilon_R}{32}}\sum_{N'=0}^N\tau_{q+1}^{-3(N-N')-1}\ell_\perp^{\frac{2}{p}-1}\ell_{\|}^{\frac{1}{p}-\frac{1}{2}}\lambda_{q+1}^{N'}\\
&\lesssim\lambda_q^{\frac{\varepsilon_R}{32}}\tau_{q+1}^{-1}\ell_\perp^{\frac{2}{p}-1}\ell_{\|}^{\frac{1}{p}-\frac{1}{2}}\lambda_{q+1}^{N}.
\end{align*}
Next, we prove \eqref{est34}. By Leibniz's rule, Fubini's theorem, \eqref{est7}, \eqref{est8}, \eqref{est10} and \eqref{est27}, one has{\small
\begin{align*}
\|d_{q+1}^c\|_{W^{N,p}}&\lesssim\sum_{i}\sum_{\xi\in\Lambda_{(i)+2}}\sum_{N'=0}^{N+1}\|a_\xi\|_{C^{N-N'+2}}\|\bar{V}_\xi\|_{W^{N',p}}\\
&+\sum_{i}\sum_{\xi\in\Lambda_{(i)+2}}\sum_{N'=0}^{N}\sum_{N''=0}^{N-N'+1}\lambda_{q+1}^{-2}\|a_\xi\|_{C^{N-N'-N''+1}}\|\bar{\psi}_\xi\|_{W^{N'',p}}\|\Phi_\xi\|_{W^{N'+1,p}}\\
&\lesssim\lambda_q^{\frac{\varepsilon_R}{32}}\sum_{N'=0}^{N+1}\tau_{q+1}^{-3(N-N'+2)-1}\ell_\perp^{\frac{2}{p}-1}\ell_{\|}^{\frac{1}{p}-\frac{1}{2}}\lambda_{q+1}^{N'-2}\\
&+\lambda_q^{\frac{\varepsilon_R}{32}}\sum_{N'=0}^{N}\sum_{N''=0}^{N-N'+1}\lambda_{q+1}^{-2}\tau_{q+1}^{-3(N-N'-N''+1)-1}\ell_{\|}^{\frac{1}{p}-\frac{1}{2}}
(\ell_\perp\ell_{\|}^{-1}\lambda_{q+1})^{N''}\ell_\perp^{\frac{2}{p}-1}\lambda_{q+1}^{N'+1}\\
&\lesssim\lambda_q^{\frac{\varepsilon_R}{32}}\ell_\perp^{\frac{2}{p}-1}\ell_{\|}^{\frac{1}{p}-\frac{1}{2}}
\left(\tau_{q+1}^{-4}\lambda_{q+1}^{N-1}+\tau_{q+1}^{-1}\lambda_{q+1}^{N}\ell_\perp\ell_{\|}^{-1}\right)\\
&\lesssim\lambda_q^{\frac{\varepsilon_R}{32}}
\tau_{q+1}^{-1}\ell_{\perp}^{\frac{2}{p}-1}\ell_{\|}^{\frac{1}{p}-\frac{1}{2}}\lambda_{q+1}^N\lambda_{q+1}^{-\frac{5}{6}(\frac{5}{4}-\alpha)}.
\end{align*}}
\!\!By the definition of $d_{q+1}^t$ in \eqref{dq+1t} and $\bar{\mu}$ in \eqref{para}, Leibniz's rule, Fubini's theorem, \eqref{est7}, \eqref{est8}, and \eqref{est27}, we infer that
\begin{align*}
\|d_{q+1}^t\|_{W^{N,p}}&\lesssim\lambda_q^{\frac{\varepsilon_R}{32}}\bar{\mu}^{-1}\sum_{N'=0}^{N+1}\sum_{N''=0}^{N'}\|a^2_\xi\|_{C^{N-N'+1}}
\|\phi_\xi^2\|_{W^{N'-N'',p}}\|\bar{\psi}_\xi^2\|_{W^{N'',p}}\\
&\lesssim\lambda_q^{\frac{\varepsilon_R}{32}}\bar{\mu}^{-1}\sum_{N'=0}^{N+1}\sum_{N''=0}^{N'}\tau_{q+1}^{-3(N-N'+1)}\ell_\perp^{\frac{2}{p}-2}\lambda_{q+1}^{N'-N''}
\ell_{\|}^{\frac{1}{p}-1}(\ell_\perp\ell_{\|}^{-1}\lambda_{q+1})^{N''}\\
&\lesssim\lambda_q^{\frac{\varepsilon_R}{32}}\ell_{\perp}^{\frac{2}{p}-1}\ell_{\|}^{\frac{1}{p}-\frac{1}{2}}\ell_{\perp}^{-1}\ell_{\|}^{-\frac{1}{2}}
\bar{\mu}^{-1}\lambda_{q+1}^{N+1}\\
&\lesssim\lambda_q^{\frac{\varepsilon_R}{32}}\ell_{\perp}^{\frac{2}{p}-1}\ell_{\|}^{\frac{1}{p}-\frac{1}{2}}\lambda_{q+1}^N
\lambda_{q+1}^{-\frac{5}{12}(\frac{5}{4}-\alpha)}.
\end{align*}
Similarly, we can prove \eqref{est36} as
\begin{align*}
\|w_{q+1}^t\|_{W^{N,p}}&\lesssim\lambda_q^{\frac{\varepsilon_R}{32}}(\mu^{-1}+\bar{\mu}^{-1})\sum_{N'=0}^{N}\sum_{N''=0}^{N'}\|a^2_\xi\|_{C^{N-N'}}
\|\phi_\xi^2\|_{W^{N'-N'',p}}\left(\|\psi_\xi^2\|_{W^{N'',p}}+\|\bar{\psi}_\xi^2\|_{W^{N'',p}}\right)\\
&\lesssim\lambda_q^{\frac{\varepsilon_R}{32}}\mu^{-1}\sum_{N'=0}^{N}\sum_{N''=0}^{N'}\tau_{q+1}^{-3(N-N')-2}\ell_\perp^{\frac{2}{p}-2}\lambda_{q+1}^{N'-N''}
\ell_{\|}^{\frac{1}{p}-1}(\ell_\perp\ell_{\|}^{-1}\lambda_{q+1})^{N''}\\
&\lesssim\lambda_q^{\frac{\varepsilon_R}{32}}\tau_{q+1}^{-2}\ell_{\perp}^{\frac{2}{p}-1}\ell_{\|}^{\frac{1}{p}-\frac{1}{2}}\ell_{\perp}^{-1}\ell_{\|}^{-\frac{1}{2}}
\mu^{-1}\lambda_{q+1}^{N}\\
&\lesssim\lambda_q^{\frac{\varepsilon_R}{32}}\tau_{q+1}^{-2}\ell_{\perp}^{\frac{2}{p}-1}\ell_{\|}^{\frac{1}{p}-\frac{1}{2}}\lambda_{q+1}^{N}
\lambda_{q+1}^{-\frac{1}{2}(\frac{5}{4}-\alpha)}.
\end{align*}
\end{proof}

\begin{Proposition}\label{prop1}
For $s\geq0$ and $1<p\leq+\infty$, there holds
\begin{align}
\|(d_{q+1},w_{q+1})\|_{L^2}&\leq\frac{3}{4}\delta_{q+1}^{\frac{1}{2}},\label{est38}\\
\|(B_{q+1}-B_q,u_{q+1}-u_q)\|_{L^2}&\leq\delta_{q+1}^{\frac{1}{2}},\label{est39}\\
\|(d_{q+1},w_{q+1})\|_{W^{s,p}}&\leq\lambda_q^{\frac{\varepsilon_R}{32}}\tau_{q+1}^{-1}\ell_{\perp}^{\frac{2}{p}-1}\ell_{\|}^{\frac{1}{p}-\frac{1}{2}}\lambda_{q+1}^{s},\label{est40}\\
\|\partial_td_{q+1}\|_{H^2}&\leq\lambda_{q+1}^6,\label{est41}\\
\|\partial_tw_{q+1}\|_{H^3}&\leq\lambda_{q+1}^6.\label{est42}
\end{align}
\end{Proposition}
\begin{proof}
The estimate \eqref{est38} follows from Proposition \ref{prop} directly. By the definition of $B_{q+1}$ in \eqref{Bq+1}, \eqref{barcha} and \eqref{est38}, we have
\begin{align*}
\|B_{q+1}-B_q\|_{L^2}\leq\|\bar{B}_{q}+d_{q+1}-\bar{B}_q\|_{L^2}+\|\bar{B}_{q}-B_q\|_{L^2}\leq\frac{3}{4}\delta_{q+1}^{\frac{1}{2}}
+\frac{1}{4}\delta_{q+1}^{\frac{1}{2}}\leq\delta_{q+1}^{\frac{1}{2}},
\end{align*}
and $\|u_{q+1}-u_q\|_{L^2}$ have the same estimate. The estimate \eqref{est40} can be deduced by \eqref{est33}--\eqref{est36} and interpolation inequality. Now we prove \eqref{est41}. By \eqref{eq17}, Leibniz's rule and \eqref{est10}, we find that
\begin{align}\label{est43}
\|\partial_t(d_{q+1}^p+d_{q+1}^c)\|_{H^2}&=\|\sum_{i\geq0}\sum_{\xi\in\Lambda_{(i)+2}}\partial_t\nabla\times\nabla\times\left(a_\xi\bar{V}_\xi\right)\|_{H^2}\notag\\
&\lesssim\lambda_q^{\frac{\varepsilon_R}{32}}\lambda_{q+1}^2\left(\ell_\perp\ell_\|\bar{\mu}\lambda_{q+1}\right)\leq\lambda_{q+1}^6.
\end{align}
For the term $\partial_td_{q+1}^t$, by the definition \eqref{dq+1t}, \eqref{est7}, \eqref{est8} and \eqref{est27}, we have
\begin{align}\label{est44}
\|\partial_td_{q+1}^t\|_{H^2}&\lesssim\lambda_q^{\frac{\varepsilon_R}{32}}\bar{\mu}^{-1}\|\partial_t(a_\xi^2\phi^2_\xi\bar{\psi}^2_\xi\xi)\|_{H^3}\notag\\
&\lesssim\lambda_q^{\frac{\varepsilon_R}{32}}\bar{\mu}^{-1}\sum_{N'=0}^3\sum_{N''=0}^{N'}\|\partial_ta_\xi^2\|_{C^{3-N'}}\|\phi^2_\xi\|_{H^{N'-N''}}
\|\bar{\psi}^2_\xi\|_{H^{N''}}\notag\\
&+\|a_\xi^2\|_{C^{3-N'}}\|\phi^2_\xi\|_{H^{N'-N''}}\|\partial_t\bar{\psi}^2_\xi\|_{H^{N''}}\notag\\
&\lesssim\lambda_q^{\frac{\varepsilon_R}{32}}\bar{\mu}^{-1}\sum_{N'=0}^3\sum_{N''=0}^{N'}\tau_{q+1}^{-3(4-N')}\ell_\perp^{-1}\lambda_{q+1}^{N'-N''}\ell_\|^{-\frac{1}{2}}
(\ell_\perp\ell_\|^{-1}\lambda_{q+1})^{N''}\notag\\
&+\tau_{q+1}^{-3(3-N')}\ell_\perp^{-1}\lambda_{q+1}^{N'-N''}\ell_\|^{-\frac{1}{2}}
(\ell_\perp\ell_\|^{-1}\lambda_{q+1})^{N''}\ell_\perp\ell_\|^{-1}\lambda_{q+1}\bar{\mu}\notag\\
&\lesssim\lambda_q^{\frac{\varepsilon_R}{32}}\bar{\mu}^{-1}\ell_\perp^{-1}\ell_\|^{-\frac{1}{2}}(\tau_{q+1}^{-3}\lambda_{q+1}^{3}+\lambda_{q+1}^4\bar{\mu}\ell_\perp\ell_\|^{-1})
\leq\lambda_{q+1}^6.
\end{align}
Then combining \eqref{est43} and \eqref{est44}, we conclude \eqref{est41}. Next, we prove \eqref{est42}.
Similar to \eqref{est43}, by using \eqref{eq18}, we have
\begin{align}\label{est45}
\|\partial_t(w_{q+1}^p+w_{q+1}^c)\|_{H^3}&\lesssim\lambda_q^{\frac{\varepsilon_R}{32}}\|\partial_t\left(a_\xi V_\xi\right)\|_{H^5}\lesssim\lambda_q^{\frac{\varepsilon_R}{32}}\lambda_{q+1}^3\left(\ell_\perp\ell_\|^{-1}\mu\lambda_{q+1}\right)\notag\\
&\lesssim\lambda_q^{\frac{\varepsilon_R}{32}}\lambda_{q+1}^{3+2\alpha}\leq\lambda_{q+1}^6.
\end{align}
By the definition of $w_{q+1}^t$ in \eqref{wq+1t}, one has
\begin{align}\label{est46}
\|\partial_tw_{q+1}^t\|_{H^3}\lesssim\lambda_q^{\frac{\varepsilon_R}{32}}\mu^{-1}\|\partial_t(a^2_\xi\phi^2_\xi\psi^2_\xi)\|_{H^3}
+\lambda_q^{\frac{\varepsilon_R}{32}}\bar{\mu}^{-1}\|\partial_t(a^2_\xi\phi^2_\xi\bar{\psi}^2_\xi)\|_{H^3}.
\end{align}
Since the estimate of the second term in \eqref{est46} is same to \eqref{est44}, we only give the estimate of the first term. Similar to \eqref{est44}, by \eqref{est6}, \eqref{est8} and  \eqref{est30}, we have
\begin{align}\label{est47}
&\lambda_q^{\frac{\varepsilon_R}{32}}\mu^{-1}\|\partial_t(a^2_\xi\phi^2_\xi\psi^2_\xi)\|_{H^3}\notag\\
&\lesssim\lambda_q^{\frac{\varepsilon_R}{32}}\mu^{-1}\ell_\perp^{-1}\ell_\|^{-\frac{1}{2}}(\tau_{q+1}^{-5}\lambda_{q+1}^{3}+\tau_{q+1}^{-2}\lambda_{q+1}^{3+2\alpha})
\leq\lambda_{q+1}^6.
\end{align}
Thus, estimates \eqref{est45}--\eqref{est47} yields \eqref{est42}.
\end{proof}
As a direct consequence of estimates \eqref{est38} and \eqref{barL2}, one has
\begin{align*}
\|u_{q+1}\|_{L^2}\leq\|\bar{u}_{q}\|_{L^2}+\|w_{q+1}\|_{L^2}\leq2\delta_0^{\frac{1}{2}}-\delta_q^{\frac{1}{2}}+\frac{3}{4}\delta_{q+1}^{\frac{1}{2}}
\leq2\delta_0^{\frac{1}{2}}-\delta_{q+1}^{\frac{1}{2}},
\end{align*}
where we have used the fact that $2\lambda_q^\beta\leq\lambda_{q+1}^\beta$. Similarly, we have $\|B_{q+1}\|_{L^2}\leq2\delta_0^{\frac{1}{2}}-\delta_{q+1}^{\frac{1}{2}}$, then we conclude the inductive estimate \eqref{uBestL2} at the level $q+1$. Estimates \eqref{est40} and \eqref{barH4} yields that
$$\|u_{q+1}\|_{H^4}\leq\|\bar{u}_{q}\|_{H^4}+\|w_{q+1}\|_{H^4}\lesssim\lambda_q^5+\lambda_q^{\frac{\varepsilon_R}{32}}\tau_{q+1}^{-1}\lambda_{q+1}^4\leq\lambda_{q+1}^5,$$
where we have used the constraints \eqref{paracon} on parameters. Similarly, one has $\|B_{q+1}\|_{H^4}\leq\lambda_{q+1}^5$, thus we conclude the inductive estimate \eqref{uBestH4} at the level of $q+1$.

Next, using Proposition \ref{prop1}, we shall establish serval bounds which will be used in Section \ref{sec5}. Applying \eqref{bardecay} and \eqref{est40}, we have
\begin{align}
\|u_{q+1}\|_{L^\infty(T/3,2T/3;H^{11/2})}&\leq\|\bar{u}_{q}\|_{L^\infty(T/3,2T/3;H^{11/2})}+\|w_{q+1}\|_{H^{11/2}}\notag\\
&\lesssim\tau_{q+1}^{-\frac{3}{2}}\lambda_q^5+\lambda_q^{\frac{\varepsilon_R}{32}}\tau_{q+1}^{-1}\lambda_{q+1}^{\frac{11}{2}}\leq\lambda_{q+1}^6,\label{est48}\\
\|B_{q+1}\|_{L^\infty(T/3,2T/3;H^{9/2})}&\leq\|\bar{B}_{q}\|_{L^\infty(T/3,2T/3;H^{9/2})}+\|d_{q+1}\|_{H^{9/2}}\notag\\
&\lesssim\tau_{q+1}^{-\frac{1}{2}}\lambda_q^5+\lambda_q^{\frac{\varepsilon_R}{32}}\tau_{q+1}^{-1}\lambda_{q+1}^{\frac{9}{2}}\leq\lambda_{q+1}^5.\label{est49}
\end{align}
Similarly, by \eqref{bardecay}, \eqref{est41} and \eqref{est42}, we get
\begin{align}
\|\partial_tu_{q+1}\|_{H^3}\leq\|\partial_t\bar{u}_{q}\|_{H^3}+\|\partial_tw_{q+1}\|_{H^3}\leq\tau_{q+1}^{-1}\lambda_{q}^5+\lambda_{q+1}^6\leq2\lambda_{q+1}^6,\label{est50}\\
\|\partial_tB_{q+1}\|_{H^2}\leq\|\partial_t\bar{B}_{q}\|_{H^2}+\|\partial_td_{q+1}\|_{H^2}\leq\tau_{q+1}^{-1}\lambda_{q}^5+\lambda_{q+1}^6\leq2\lambda_{q+1}^6.\label{est51}
\end{align}

\section{Convex integration step: the Reynolds stress}\label{sec5}
In this section, we will prove the following proposition:
\begin{Proposition}\label{prop2}
There exists an $\varepsilon_R>0$ sufficiently small, and a parameter $p>0$ sufficiently close to $1$, depending only on $\beta,\alpha,b$ and $\rho$, such that the following statement holds: There exists traceless symmetric 2-tensor $\tilde{R}^u$ and $\mathring{R}_{q+1}^B$, and a scalar pressure field $\tilde{p}$, defined implicitly in \eqref{eq19} below, satisfying:
\begin{equation}
\begin{cases}
\partial_tu_{q+1}+(-\Delta)^\alpha u_{q+1}+{\rm div}(u_{q+1}\otimes u_{q+1}-B_{q+1}\otimes B_{q+1})+\nabla\tilde{p}={\rm div}\tilde{R}^u,\\
\partial_tB_{q+1}+(-\Delta)^\alpha B_{q+1}+{\rm div}(u_{q+1}\otimes B_{q+1}-B_{q+1}\otimes u_{q+1})\\
\quad\quad+\nabla\times{\rm div}(B_{q+1}\otimes B_{q+1})=\nabla\times{\rm div}\mathring{R}_{q+1}^B,\\
{\rm div}u_{q+1}={\rm div}B_{q+1}=0.
\end{cases}
\end{equation}
Furthermore, $(\tilde{R}^u,\mathring{R}_{q+1}^B)$ has bounds
\begin{align}\label{est52}
\|(\tilde{R}^u,\mathring{R}_{q+1}^B)\|_{L^p}\lesssim\lambda_{q+1}^{-2\varepsilon_R}\delta_{q+2},
\end{align}
where the constant depends on the choice of $p$ and $\varepsilon_R$, but independent of $q$. And $(\tilde{R}^u,\mathring{R}_{q+1}^B)$ has the support property:
\begin{align}\label{supp}
{\rm supp}(\tilde{R}^u,\mathring{R}_{q+1}^B)\subset\mathbb{T}^3\times\left\{t\in[0,T]: {\rm dist}(t,\mathscr{G}^{(q+1)})>\tau_{q+1}\right\}.
\end{align}
\end{Proposition}
\begin{proof}
By \eqref{dq+1}, \eqref{Bq+1}, \eqref{wq+1}, \eqref{uq+1}, and the fact that $(\bar{u}_q,\bar{B}_q,\mathring{\bar{R}}_{q}^u,\mathring{\bar{R}}_{q}^B)$ satisfies equation \eqref{Hall-MHD2}, we obtain\small{
\begin{align}\label{eq19}
&{\rm div}\tilde{R}^u-\nabla\tilde{p}\notag\\
&=(-\Delta)^\alpha w_{q+1}+\partial_t(w_{q+1}^p+w_{q+1}^c)+{\rm div}(\bar{u}_{q}\otimes w_{q+1}+w_{q+1}\otimes \bar{u}_{q}-\bar{B}_{q}\otimes d_{q+1}-d_{q+1}\otimes \bar{B}_{q})\notag\\
&\quad+{\rm div}((w_{q+1}^c+w_{q+1}^t)\otimes w_{q+1}+w_{q+1}^p\otimes(w_{q+1}^c+w_{q+1}^t))\notag\\
&\quad\quad\quad-{\rm div}((d_{q+1}^c+d_{q+1}^t)\otimes d_{q+1}+d_{q+1}^p\otimes(d_{q+1}^c+d_{q+1}^t))\notag\\
&\quad+{\rm div}(\mathring{\bar{R}}_{q}^u+w_{q+1}^p\otimes w_{q+1}^p-d_{q+1}^p\otimes d_{q+1}^p)+\partial_tw_{q+1}^t\notag\\
&=:{\rm div}(\tilde{R}^u_{linear}+\tilde{R}^u_{corrector}+\tilde{R}^u_{oscillation})+\nabla p,
\end{align}}
\!\!where the linear error is defined by applying $\mathcal{R}$ to the first line of \eqref{eq19}, while the corrector error is defined by applying $\mathcal{R}$ to the second and third lines of \eqref{eq19}. The oscillation error will be defined below. Similar to \eqref{eq19}, we calculate\small{
\begin{align}\label{eq20}
&\nabla\times{\rm div}\mathring{R}_{q+1}^B\notag\\
&=(-\Delta)^\alpha d_{q+1}+\partial_t(d_{q+1}^p+d_{q+1}^c)+{\rm div}(\bar{u}_{q}\otimes d_{q+1}+w_{q+1}\otimes \bar{B}_{q}-\bar{B}_{q}\otimes w_{q+1}-d_{q+1}\otimes \bar{u}_{q})\notag\\
&\quad\quad\quad+\nabla\times{\rm div}(\bar{B}_{q}\otimes d_{q+1}+d_{q+1}\otimes \bar{B}_{q})\notag\\
&\quad+\nabla\times{\rm div}((d_{q+1}^c+d_{q+1}^t)\otimes d_{q+1}+d_{q+1}^p\otimes(d_{q+1}^c+d_{q+1}^t))+{\rm div}(w_{q+1}\otimes d_{q+1}-d_{q+1}\otimes w_{q+1})\notag\\
&\quad+\nabla\times{\rm div}(\mathring{\bar{R}}_{q}^B+d_{q+1}^p\otimes d_{q+1}^p)+\partial_td_{q+1}^t\notag\\
&=:\nabla\times{\rm div}(\tilde{R}^B_{linear}+\tilde{R}^B_{corrector}+\tilde{R}^B_{oscillation}),
\end{align}}
\!\!where the linear error is defined by applying $\mathcal{R}{\rm curl}^{-1}$ to the first and second lines of \eqref{eq20}, while the corrector error and oscillation error are defined by applying $\mathcal{R}{\rm curl}^{-1}$ to the second and third line of \eqref{eq20}, respectively. Next, we will bound these errors step by step.

By using \eqref{eq17}, \eqref{barH4}, \eqref{est10} and \eqref{est40}, one has
\begin{align}\label{est53}
&\|\tilde{R}^B_{linear}\|_{L^p}\lesssim\|d_{q+1}\|_{W^{2\alpha-2,p}}+\|\mathcal{R}{\rm curl}^{-1}\partial_t\sum_{i\geq0}\sum_{\xi\in\Lambda_{(i)+2}}
\nabla\times\nabla\times(a_\xi\bar{V}_\xi)\|_{L^p}\notag\\
&+\|\mathcal{R}{\rm curl}^{-1}{\rm div}\mathbb{P}_{\neq0}(\bar{u}_{q}\otimes d_{q+1}+w_{q+1}\otimes \bar{B}_{q}-\bar{B}_{q}\otimes w_{q+1}-d_{q+1}\otimes \bar{u}_{q})\|_{L^p}\notag\\
&+\|\bar{B}_{q}\otimes d_{q+1}+d_{q+1}\otimes \bar{B}_{q}\|_{L^p}\notag\\
&\lesssim\lambda_q^{\frac{\varepsilon_R}{32}}\tau_{q+1}^{-1}\ell_\perp^{\frac{2}{p}-1}\ell_\|^{\frac{1}{p}-\frac{1}{2}}\lambda_{q+1}^{2\alpha-2}
+\lambda_q^{\frac{\varepsilon_R}{32}}\ell_\perp^{\frac{2}{p}-1}\ell_\|^{\frac{1}{p}-\frac{1}{2}}\lambda_{q+1}^{-2}(\ell_\perp\ell_\|^{-1}\lambda_{q+1}\bar{\mu})
+\lambda_q^5\lambda_q^{\frac{\varepsilon_R}{32}}\tau_{q+1}^{-1}\ell_\perp^{\frac{2}{p}-1}\ell_\|^{\frac{1}{p}-\frac{1}{2}}\notag\\
&\lesssim\ell_\perp^{\frac{2}{p}-2}\ell_\|^{\frac{1}{p}-1}\lambda_q^{\frac{\varepsilon_R}{32}}
\left(\tau_{q+1}^{-1}\lambda_{q+1}^{-\frac{1}{3}(\frac{17}{4}-\alpha)}+\lambda_{q+1}^{-\frac{5}{12}(\frac{5}{4}-\alpha)}+\lambda_q^5\tau_{q+1}^{-1}
\lambda_{q+1}^{-\frac{5}{3}(\alpha-\frac{7}{20})}\right)\notag\\
&\lesssim\ell_\perp^{\frac{2}{p}-2}\ell_\|^{\frac{1}{p}-1}\lambda_{q+1}^{-\frac{5}{12}(\frac{5}{4}-\alpha)}\lambda_q^{\frac{\varepsilon_R}{32}},
\end{align}
where we have used the Calderon--Zygmund inequality and the following estimate
\begin{align}\label{schuarder est}
\||\nabla|^{-1}\mathbb{P}_{\neq0}\|_{L^p\rightarrow L^p}\lesssim1,
\end{align}
for $p\in(1,2]$, see \cite[Lemma B.1]{bv},  for example. Similarly, by \eqref{eq18}, \eqref{barH4}, \eqref{est9} and \eqref{est40}, we can  infer that
\begin{align}\label{est54}
&\|\tilde{R}^u_{linear}\|_{L^p}\lesssim\|w_{q+1}\|_{W^{2\alpha-1,p}}+\|\mathcal{R}\partial_t\sum_{i\geq0}\sum_{\xi\in\Lambda_{(i)}}
\nabla\times\nabla\times(a_\xi V_\xi)\|_{L^p}\notag\\
&+\|(\bar{u}_q,\bar{B}_q)\|_{L^\infty}\|(w_{q+1},d_{q+1})\|_{L^p}\notag\\
&\lesssim\lambda_q^{\frac{\varepsilon_R}{32}}\tau_{q+1}^{-1}\ell_\perp^{\frac{2}{p}-1}\ell_\|^{\frac{1}{p}-\frac{1}{2}}\lambda_{q+1}^{2\alpha-1}
+\lambda_q^{\frac{\varepsilon_R}{32}}\ell_\perp^{\frac{2}{p}-1}\ell_\|^{\frac{1}{p}-\frac{1}{2}}\lambda_{q+1}^{-1}(\ell_\perp\ell_\|^{-1}\lambda_{q+1}\mu)
+\lambda_q^5\lambda_q^{\frac{\varepsilon_R}{32}}\tau_{q+1}^{-1}\ell_\perp^{\frac{2}{p}-1}\ell_\|^{\frac{1}{p}-\frac{1}{2}}\notag\\
&\lesssim\ell_\perp^{\frac{2}{p}-2}\ell_\|^{\frac{1}{p}-1}\lambda_q^{\frac{\varepsilon_R}{32}}\tau_{q+1}^{-1}
\left(\lambda_{q+1}^{-\frac{1}{3}(\frac{5}{4}-\alpha)}+\lambda^5_q\lambda_{q+1}^{-\frac{5}{3}(\alpha-\frac{7}{20})}\right)\notag\\
&\lesssim\ell_\perp^{\frac{2}{p}-2}\ell_\|^{\frac{1}{p}-1}\lambda_{q+1}^{-\frac{1}{3}(\frac{5}{4}-\alpha)}\lambda_q^{\frac{\varepsilon_R}{32}}\tau_{q+1}^{-1},
\end{align}

Now we bound the corrector errors. Estimates \eqref{est33}--\eqref{est36} and \eqref{est40} yields that
\begin{align}\label{est55}
\|\tilde{R}^u_{corrector}\|_{L^p}&\lesssim\|w_{q+1}^c+w_{q+1}^t\|_{L^{2p}}(\|w_{q+1}\|_{L^{2p}}+\|w_{q+1}^p\|_{L^{2p}})\notag\\
&+\|d_{q+1}^c+d_{q+1}^t\|_{L^{2p}}(\|d_{q+1}\|_{L^{2p}}+\|d_{q+1}^p\|_{L^{2p}})\notag\\
&\lesssim\lambda_q^{\frac{\varepsilon_R}{16}}\tau_{q+1}^{-3}\ell_\perp^{\frac{2}{p}-2}\ell_\|^{\frac{1}{p}-1}\lambda_{q+1}^{-\frac{1}{2}(\frac{5}{4}-\alpha)}
+\lambda_q^{\frac{\varepsilon_R}{16}}\tau_{q+1}^{-1}\ell_\perp^{\frac{2}{p}-2}\ell_\|^{\frac{1}{p}-1}\lambda_{q+1}^{-\frac{5}{12}(\frac{5}{4}-\alpha)}\notag\\
&\lesssim\lambda_q^{\frac{\varepsilon_R}{16}}\tau_{q+1}^{-3}\ell_\perp^{\frac{2}{p}-2}\ell_\|^{\frac{1}{p}-1}\lambda_{q+1}^{-\frac{5}{12}(\frac{5}{4}-\alpha)}.
\end{align}
In terms of the corrector error $\tilde{R}^B_{corrector}$, note the fact that $w_{q+1}^p\otimes d_{q+1}^p=0$. Indeed, one has
\begin{align}\label{eq21}
w_{q+1}^p\otimes d_{q+1}^p=\sum_{i,j}\sum_{\xi\in\Lambda_{(i)},\xi'\in\Lambda_{(j)+2}}a_\xi a_{\xi'}W_\xi\otimes\bar{W}_{\xi'},
\end{align}
then using the disjoint support property of intermittent jets and the fact that $\Lambda_{(i)}\cap\Lambda_{(j)+2}=\emptyset$, we conclude $w_{q+1}^p\otimes d_{q+1}^p=0$. In view of Calderon--Zygmund inequality and the bound \eqref{schuarder est}, mirroring \eqref{est55}, we have
\begin{align}\label{est56}
\|\tilde{R}^B_{corrector}\|_{L^p}\lesssim
\lambda_q^{\frac{\varepsilon_R}{16}}\tau_{q+1}^{-3}\ell_\perp^{\frac{2}{p}-2}\ell_\|^{\frac{1}{p}-1}\lambda_{q+1}^{-\frac{5}{12}(\frac{5}{4}-\alpha)}.
\end{align}

Now, we focus on the estimate of $\tilde{R}^B_{oscillation}$. Recalling \eqref{eq14}, we compute\small{
\begin{align}\label{eq22}
&\nabla\times{\rm div}(\mathring{\bar{R}}_q^B+d_{q+1}\otimes d_{q+1})\notag\\
&=\nabla\times{\rm div}\left(\theta_B^2\sum_{i\geq0}\rho_i|\chi_i^B|^2{\rm Id}-\sum_{i\geq0}\sum_{\xi\in\Lambda_{(i)+2}}a^2_\xi\fint_{\mathbb{T}^3}\bar{W}_\xi\otimes \bar{W}_\xi dx\right)\notag\\
&\quad+\nabla\times{\rm div}\sum_{i,j}\sum_{\xi\in\Lambda_{(i)+2},\xi'\in\Lambda_{(j)+2}}a_\xi a_{\xi'}\bar{W}_\xi\otimes\bar{W}_{\xi'}\notag\\
&=\nabla\times{\rm div}\sum_{i\geq0}\sum_{\xi\in\Lambda_{(i)+2}}a_\xi^2\mathbb{P}_{\neq0}\left(\bar{W}_\xi\otimes\bar{W}_{\xi}\right)\notag\\
&=\nabla\times\sum_{i\geq0}\sum_{\xi\in\Lambda_{(i)+2}}\nabla(a^2_\xi)\mathbb{P}_{\neq0}\left(\bar{W}_\xi\otimes\bar{W}_{\xi}\right)
+a^2_\xi{\rm div}\left(\bar{W}_\xi\otimes\bar{W}_{\xi}\right)\notag\\
&=:E_{(\xi,1)}^B+E_{(\xi,2)}^B,
\end{align}}
\!\!where the second equality follows from the definition of $a_\xi$ in \eqref{axiB}, the disjoint support property of $\bar{W}_{\xi}$ and the fact that $\Lambda_2\cap\Lambda_3=\emptyset$. The term $E_{(\xi,1)}^B$ can be bounded similar to \cite[Section 5.1.3]{bcv}. More precisely, combining the fact that the frequency of $a_\xi$ is localized inside of the ball of radius $\tau_{q+1}^{-3}\ll\lambda_{q+1}\ell_\perp=\lambda_{q+1}^{-\frac{5}{6}(\alpha-\frac{5}{4})}$, and Lemma B.1 of \cite{bv} with parameter choices $\lambda=\tau_{q+1}^{-3}$, $C_a=\tau_{q+1}^{-2}$, $\kappa=\lambda_{q+1}\ell_\perp/2$ and $L$ sufficiently large, one has
\begin{align}\label{est57}
\|\mathcal{R}{\rm curl}^{-1}E_{(\xi,1)}^B\|_{L^p}&\lesssim\lambda_q^{\frac{\varepsilon_R}{32}}\||\nabla|^{-1}\mathbb{P}_{\neq0}\left(\nabla(a^2_\xi)
\mathbb{P}_{\geq\lambda_{q+1}\ell_\perp/2}\left(\bar{W}_\xi\otimes\bar{W}_{\xi}\right)\right)\|_{L^p}\notag\\
&\lesssim\lambda_q^{\frac{\varepsilon_R}{32}}\tau_{q+1}^{-2}\left(1+\frac{\tau_{q+1}^{-3L}}{(\lambda_{q+1}\ell_\perp)^{L-2}}\right)(\lambda_{q+1}\ell_\perp)^{-1}
\left\|\bar{W}_\xi\otimes\bar{W}_{\xi}\right\|_{L^p}\notag\\
&\lesssim\lambda_q^{\frac{\varepsilon_R}{32}}\tau_{q+1}^{-2}\ell_\perp^{\frac{2}{p}-2}\ell_\|^{\frac{1}{p}-1}\lambda_{q+1}^{-\frac{5}{6}(\frac{5}{4}-\alpha)}.
\end{align}
In view of \eqref{eq12}, $E_{(\xi,2)}^B$ can be written as
\begin{align*}
E_{(\xi,2)}^B&=\bar{\mu}^{-1}\nabla\times\sum_{i\geq0}\sum_{\xi\in\Lambda_{(i)+2}}a_\xi^2\phi^2_\xi\partial_t\bar{\psi}^2\xi\\
&=\bar{\mu}^{-1}\nabla\times\sum_{i\geq0}\sum_{\xi\in\Lambda_{(i)+2}}\partial_t\left(a_\xi^2\phi^2_\xi\bar{\psi}^2\xi\right)-\partial_t\left(a_\xi^2\right)\phi^2_\xi\bar{\psi}^2\xi.
\end{align*}
Recalling the definition of $d_{q+1}^t$ in \eqref{dq+1t}, we have
$$E_{(\xi,2)}^B+\partial_td_{q+1}^t=-\bar{\mu}^{-1}\nabla\times\sum_{i\geq0}\sum_{\xi\in\Lambda_{(i)+2}}\partial_t\left(a_\xi^2\right)\phi^2_\xi\bar{\psi}^2\xi.$$
Then, by estimates \eqref{est7}, \eqref{est8} and \eqref{est27}, one has
\begin{align}\label{est58}
\|\mathcal{R}{\rm curl}^{-1}(E_{(\xi,2)}^B+\partial_td_{q+1}^t)\|_{L^p}&\lesssim\bar{\mu}^{-1}\lambda_q^{\frac{\varepsilon_R}{32}}
\left\||\nabla|^{-1}\mathbb{P}_{\neq0}\left(\partial_t\left(a_\xi^2\right)\phi^2_\xi\bar{\psi}^2\xi\right)\right\|_{L^p}\notag\\
&\lesssim\bar{\mu}^{-1}\lambda_q^{\frac{\varepsilon_R}{32}}\tau_{q+1}^{-5}\ell_\perp^{\frac{2}{p}-2}\ell_\|^{\frac{1}{p}-1}.
\end{align}
Combining \eqref{eq22}--\eqref{est58} with the definition of $\tilde{R}^B_{oscillation}$ in \eqref{eq20}, we conclude that
\begin{align}\label{est59}
\|\tilde{R}^B_{oscillation}\|_{L^p}\lesssim\lambda_q^{\frac{\varepsilon_R}{32}}\tau_{q+1}^{-2}\ell_\perp^{\frac{2}{p}-2}\ell_\|^{\frac{1}{p}-1}
\lambda_{q+1}^{-\frac{5}{6}(\frac{5}{4}-\alpha)}.
\end{align}
Therefore, \eqref{est53}, \eqref{est56} and \eqref{est59} yields that
\begin{align}\label{est60}
\|\mathring{R}_{q+1}^B\|_{L^p}\lesssim\lambda_q^{\frac{\varepsilon_R}{16}}\tau_{q+1}^{-3}\ell_\perp^{\frac{2}{p}-2}\ell_\|^{\frac{1}{p}-1}
\lambda_{q+1}^{-\frac{5}{12}(\frac{5}{4}-\alpha)}.
\end{align}

Next, we deal with the term $\tilde{R}^u_{oscillation}$. Applying the equality \eqref{eq16}, the definitions of $G^B$ in \eqref{GB}, $w_{q+1}^p$ and $d_{q+1}^p$ in \eqref{wq+1p} and \eqref{dq+1p} respectively, and using the same argument of \eqref{eq22}, then one has\small{
\begin{align}\label{eq23}
&{\rm div}\left(\mathring{\bar{R}}_q^u+w_{q+1}\otimes w_{q+1}-d_{q+1}\otimes d_{q+1}\right)\notag\\
&={\rm div}\left(\theta_u^2\sum_{i\geq0}\rho_i|\chi_i^u|^2{\rm Id}-\sum_{i\geq0}\sum_{\xi\in\Lambda_{(i)}}a^2_\xi\fint_{\mathbb{T}^3}W_\xi\otimes W_\xi dx
+\sum_{i\geq0}\sum_{\xi\in\Lambda_{(i)+2}}a^2_\xi\fint_{\mathbb{T}^3}\bar{W}_\xi\otimes\bar{W}_\xi dx\right)\notag\\
&\quad+{\rm div}\left(\sum_{i\geq0}\sum_{\xi\in\Lambda_{(i)}}a^2_\xi W_\xi\otimes W_\xi
-\sum_{i\geq0}\sum_{\xi\in\Lambda_{(i)+2}}a^2_\xi\bar{W}_\xi\otimes\bar{W}_\xi \right)\notag\\
&=\nabla\left(\theta_u^2\sum_{i\geq0}\rho_i|\chi_i^u|^2\right)
+{\rm div}\left(\sum_{i\geq0}\sum_{\xi\in\Lambda_{(i)}}a^2_\xi\mathbb{P}_{\neq0}(W_\xi\otimes W_\xi)
-\sum_{i\geq0}\sum_{\xi\in\Lambda_{(i)+2}}a^2_\xi\mathbb{P}_{\neq0}(\bar{W}_\xi\otimes\bar{W}_\xi)\right)\notag\\
&=\nabla\left(\theta_u^2\sum_{i\geq0}\rho_i|\chi_i^u|^2\right)\notag\\
&\quad+\sum_{i\geq0}\sum_{\xi\in\Lambda_{(i)}}\mathbb{P}_{\neq0}\left(\nabla(a_\xi^2)\mathbb{P}_{\neq0}(W_\xi\otimes W_\xi)\right)
-\sum_{i\geq0}\sum_{\xi\in\Lambda_{(i)+2}}\mathbb{P}_{\neq0}\left(\nabla(a_\xi^2)\mathbb{P}_{\neq0}(\bar{W}_\xi\otimes \bar{W}_\xi)\right)\notag\\
&\quad+\sum_{i\geq0}\sum_{\xi\in\Lambda_{(i)}}\mathbb{P}_{\neq0}\left(a_\xi^2{\rm div}(W_\xi\otimes W_\xi)\right)
-\sum_{i\geq0}\sum_{\xi\in\Lambda_{(i)+2}}\mathbb{P}_{\neq0}\left(a_\xi^2{\rm div}(\bar{W}_\xi\otimes \bar{W}_\xi)\right)\notag\\
&=:\nabla\left(\theta_u^2\sum_{i\geq0}\rho_i|\chi_i^u|^2\right)+E_{\xi,1}^u+E_{\xi,2}^u,
\end{align}}
\!\!where $E_{\xi,1}^u$ and $E_{\xi,2}^u$ are the second and  third line of \eqref{eq23}, respectively. Similar to \eqref{est57}, we have
\begin{align}\label{est61}
\|\mathcal{R}E_{\xi,1}^u\|_{L^p}&\lesssim\lambda_q^{\frac{\varepsilon_R}{32}}\||\nabla|^{-1}\mathbb{P}_{\neq0}\left(\nabla(a_\xi^2)
\mathbb{P}_{\geq\lambda_{q+1}\ell_\perp/2}(W_\xi\otimes W_\xi)\right)\|_{L^p}\notag\\
&\quad+\lambda_q^{\frac{\varepsilon_R}{32}}\||\nabla|^{-1}\mathbb{P}_{\neq0}\left(\nabla(a_\xi^2)
\mathbb{P}_{\geq\lambda_{q+1}\ell_\perp/2}(\bar{W}_\xi\otimes \bar{W}_\xi)\right)\|_{L^p}\notag\\
&\lesssim\lambda_q^{\frac{\varepsilon_R}{32}}\tau_{q+1}^{-2}\ell_\perp^{\frac{2}{p}-2}\ell_\|^{\frac{1}{p}-1}
\lambda_{q+1}^{-\frac{5}{6}(\frac{5}{4}-\alpha)}.
\end{align}
In view of \eqref{eq11} and \eqref{eq12}, $E_{\xi,2}^u$ can be written as
\begin{align*}
E_{\xi,2}^u&=\mu^{-1}\sum_{i\geq0}\sum_{\xi\in\Lambda_{(i)}}\mathbb{P}_{\neq0}\left(a_\xi^2\phi^2_\xi\partial_t\psi^2_\xi\xi\right)
-\bar{\mu}^{-1}\sum_{i\geq0}\sum_{\xi\in\Lambda_{(i)+2}}\mathbb{P}_{\neq0}\left(a_\xi^2\phi^2_\xi\partial_t\bar{\psi}^2_\xi\xi\right)\\
&=\mu^{-1}\sum_{i\geq0}\sum_{\xi\in\Lambda_{(i)}}\partial_t\mathbb{P}_{\neq0}\left(a_\xi^2\phi^2_\xi\psi^2_\xi\xi\right)
-\bar{\mu}^{-1}\sum_{i\geq0}\sum_{\xi\in\Lambda_{(i)+2}}\partial_t\mathbb{P}_{\neq0}\left(a_\xi^2\phi^2_\xi\bar{\psi}^2_\xi\xi\right)\\
&-\mu^{-1}\sum_{i\geq0}\sum_{\xi\in\Lambda_{(i)}}\mathbb{P}_{\neq0}\left(\partial_t(a_\xi^2)\phi^2_\xi\psi^2_\xi\xi\right)
+\bar{\mu}^{-1}\sum_{i\geq0}\sum_{\xi\in\Lambda_{(i)+2}}\mathbb{P}_{\neq0}\left(\partial_t(a_\xi^2)\phi^2_\xi\bar{\psi}^2_\xi\xi\right),
\end{align*}
then by the definition of $\partial_tw_{q+1}^t$ in \eqref{wq+1t} and the fact that ${\rm Id}-\mathbb{P}_H=\nabla(\Delta^{-1}{\rm div})$, we have
\begin{align}\label{eq24}
&E_{\xi,2}^u+\partial_tw_{q+1}^t\notag\\
&=\mu^{-1}\sum_{i\geq0}\sum_{\xi\in\Lambda_{(i)}}({\rm Id}-\mathbb{P}_H)\partial_t\mathbb{P}_{\neq0}\left(a_\xi^2\phi^2_\xi\psi^2_\xi\xi\right)
-\bar{\mu}^{-1}\sum_{i\geq0}\sum_{\xi\in\Lambda_{(i)+2}}({\rm Id}-\mathbb{P}_H)\partial_t\mathbb{P}_{\neq0}\left(a_\xi^2\phi^2_\xi\bar{\psi}^2_\xi\xi\right)\notag\\
&\quad-\mu^{-1}\sum_{i\geq0}\sum_{\xi\in\Lambda_{(i)}}\mathbb{P}_{\neq0}\left(\partial_t(a_\xi^2)\phi^2_\xi\psi^2_\xi\xi\right)
+\bar{\mu}^{-1}\sum_{i\geq0}\sum_{\xi\in\Lambda_{(i)+2}}\mathbb{P}_{\neq0}\left(\partial_t(a_\xi^2)\phi^2_\xi\bar{\psi}^2_\xi\xi\right)\notag\\
&=\nabla q-\mu^{-1}\sum_{i\geq0}\sum_{\xi\in\Lambda_{(i)}}\mathbb{P}_{\neq0}\left(\partial_t(a_\xi^2)\phi^2_\xi\psi^2_\xi\xi\right)
+\bar{\mu}^{-1}\sum_{i\geq0}\sum_{\xi\in\Lambda_{(i)+2}}\mathbb{P}_{\neq0}\left(\partial_t(a_\xi^2)\phi^2_\xi\bar{\psi}^2_\xi\xi\right),
\end{align}
where
\begin{align*}
q=\Delta^{-1}{\rm div}\left(\mu^{-1}\sum_{i\geq0}\sum_{\xi\in\Lambda_{(i)}}\partial_t\mathbb{P}_{\neq0}\left(a_\xi^2\phi^2_\xi\psi^2_\xi\xi\right)
-\bar{\mu}^{-1}\sum_{i\geq0}\sum_{\xi\in\Lambda_{(i)+2}}\partial_t\mathbb{P}_{\neq0}\left(a_\xi^2\phi^2_\xi\bar{\psi}^2_\xi\xi\right)\right)
\end{align*}
is the pressure term. Similar to \eqref{est58}, using \eqref{eq24}, \eqref{est6}--\eqref{est8}, \eqref{est27} and \eqref{est30}, we can estimate the second contribution to $\tilde{R}^u_{oscillation}$ by
\begin{align}\label{est62}
&\left\|\mu^{-1}\sum_{i\geq0}\sum_{\xi\in\Lambda_{(i)}}\mathcal{R}\mathbb{P}_{\neq0}\left(\partial_t(a_\xi^2)\phi^2_\xi\psi^2_\xi\xi\right)
-\bar{\mu}^{-1}\sum_{i\geq0}\sum_{\xi\in\Lambda_{(i)+2}}\mathcal{R}\mathbb{P}_{\neq0}\left(\partial_t(a_\xi^2)\phi^2_\xi\bar{\psi}^2_\xi\xi\right)\right\|_{L^p}\notag\\
&\lesssim\lambda_q^{\frac{\varepsilon_R}{32}}\left(\mu^{-1}\tau_{q+1}^{-5}\ell_\perp^{\frac{2}{p}-2}\ell_\|^{\frac{1}{p}-1}
+\bar{\mu}^{-1}\tau_{q+1}^{-5}\ell_\perp^{\frac{2}{p}-2}\ell_\|^{\frac{1}{p}-1}\right)\notag\\
&\lesssim\lambda_q^{\frac{\varepsilon_R}{32}}\tau_{q+1}^{-5}\ell_\perp^{\frac{2}{p}-2}\ell_\|^{\frac{1}{p}-1}\mu^{-1}.
\end{align}
Combining \eqref{eq23}--\eqref{est62}, one infers that
\begin{align}\label{est63}
\left\|\tilde{R}^u_{oscillation}\right\|_{L^p}\lesssim\lambda_q^{\frac{\varepsilon_R}{32}}\tau_{q+1}^{-2}\ell_\perp^{\frac{2}{p}-2}\ell_\|^{\frac{1}{p}-1}
\lambda_{q+1}^{-\frac{5}{6}(\frac{5}{4}-\alpha)}.
\end{align}
Thus, recalling estimates \eqref{est54}, \eqref{est55} and \eqref{est60}, we have
\begin{align*}
\left\|(\tilde{R}^u,\mathring{R}_{q+1}^B)\right\|_{L^p}\lesssim\ell_\perp^{\frac{2}{p}-2}\ell_\|^{\frac{1}{p}-1}
\left(\lambda_q^{\frac{\varepsilon_R}{16}}\tau_{q+1}^{-3}\lambda_{q+1}^{-\frac{5}{12}(\frac{5}{4}-\alpha)}+
\lambda_q^{\frac{\varepsilon_R}{32}}\tau_{q+1}^{-1}\lambda_{q+1}^{-\frac{1}{3}(\frac{5}{4}-\alpha)}\right)\leq\lambda_{q+1}^{-2\varepsilon_R}\delta_{q+2},
\end{align*}
for $p$ sufficiently close to $1$, and $2(\varepsilon_R+\beta b)\leq\frac{1}{12}(\frac{5}{4}-\alpha)$. Hence, we conclude \eqref{est52}.

Finally, we prove the temporal support property of $(\tilde{R}^u,\mathring{R}_{q+1}^B)$. By the definitions \eqref{eq19} and \eqref{eq20}, we have
\begin{align}\label{supp1}
{\rm supp}(\tilde{R}^u,\mathring{R}_{q+1}^B)\subset&{\rm supp}w_{q+1}^p\cup{\rm supp}w_{q+1}^c\cup{\rm supp}w_{q+1}^t\cup{\rm supp}\mathring{\bar{R}}_{q}^u\notag\\
&\cup{\rm supp}d_{q+1}^p\cup{\rm supp}d_{q+1}^c\cup{\rm supp}d_{q+1}^t\cup{\rm supp}\cup\mathring{\bar{R}}_{q}^B.
\end{align}
By using the definition of $a_\xi$ in \eqref{axiu} and \eqref{axiB}, $\xi \in\Lambda_{(i)}$ or $\Lambda_{(i)+2}$, we obtain $a_\xi=0$ whenever ${\rm dist}(t,\mathscr{G}^{(q+1)})\leq\tau_{q+1}$. In view of \eqref{bar supp}, \eqref{supp1} yields $(\tilde{R}^u,\mathring{R}_{q+1}^B)=0$ whenever ${\rm dist}(t,\mathscr{G}^{(q+1)})\leq\tau_{q+1}$, which implies \eqref{supp}.
\end{proof}

As a direct consequence of Proposition \ref{prop2}, we have the following Corollary, which will complete the inductive estimates \eqref{RestL1} and \eqref{RestH4} at the level $q+1$.
\begin{Corollary}
There exists traceless 2--tensor $(\mathring{R}_{q+1}^u,\mathring{R}_{q+1}^B)$ and a scalar pressure field $p_{q+1}$ such that $(u_{q+1},B_{q+1},\mathring{R}_{q+1}^u,\mathring{R}_{q+1}^B)$ solves the system \eqref{Hall-MHD2} at the level $q+1$. Moreover, the following estimates holds:
\begin{align}
\|(\mathring{R}_{q+1}^u,\mathring{R}_{q+1}^B)\|_{L^1}&\leq\lambda_{q+1}^{-\varepsilon_R\delta_{q+2}},\label{est64}\\
\|(\mathring{R}_{q+1}^u,\mathring{R}_{q+1}^B)\|_{H^4}&\leq\lambda_{q+1}^7,\label{est65}
\end{align}
and $(\mathring{R}_{q+1}^u,\mathring{R}_{q+1}^B)=0$ whenever ${\rm dist}(t,\mathscr{G}^{(q+1)})\leq\tau_{q+1}$.
\end{Corollary}
\begin{proof}
Define
\begin{align*}
\mathring{R}_{q+1}^u=\mathcal{R}\mathbb{P}_{H}{\rm div}\tilde{R}^u,\quad p_{q+1}=\tilde{p}-\Delta^{-1}{\rm div}{\rm div}\tilde{R}^u,
\end{align*}
then $(u_{q+1},B_{q+1},\mathring{R}_{q+1}^u,\mathring{R}_{q+1}^B)$ satisfies system \eqref{Hall-MHD2} with pressure $p_{q+1}$. Since the operator $\mathcal{R}\mathbb{P}_{H}{\rm div}$ is time independent, by Proposition \ref{prop2}, we have
$$(\mathring{R}_{q+1}^u,\mathring{R}_{q+1}^B)=0\;{\rm whenever} \;{\rm dist}(t,\mathscr{G}^{(q+1)})\leq\tau_{q+1},$$
which implies (\romannumeral6) in Section \ref{sec2} holds at level $q+1$. Recall $p>1$ in Proposition \ref{prop2} and the bound $\|\mathcal{R}\mathbb{P}_{H}{\rm div}\|_{L^p\rightarrow L^p}\lesssim1$, by \eqref{est52}, one has
$$\|(\mathring{R}_{q+1}^u,\mathring{R}_{q+1}^B)\|_{L^1}\leq\|(\mathring{R}_{q+1}^u,\mathring{R}_{q+1}^B)\|_{L^p}
\lesssim\|(\tilde{R}^u,\mathring{R}_{q+1}^B)\|_{L^p}\lesssim\lambda_{q+1}^{-2\varepsilon_R}\delta_{q+2}\leq\lambda_{q+1}^{-\varepsilon_R}\delta_{q+2},$$
which implies that \eqref{est64}.

Next, we prove \eqref{est65}. Since ${\rm supp}(\mathring{R}_{q+1}^u,\mathring{R}_{q+1}^B)\subset\mathbb{T}^3\times[T/3,2T/3]$, then by the system \eqref{Hall-MHD2}, estimates \eqref{est48}--\eqref{est51} and the Gagliardo--Nirenberg--Sobolev inequality, we have\small{
\begin{align*}
&\|\mathring{R}_{q+1}^u\|_{H^4}=\|\mathcal{R}\mathbb{P}_{H}{\rm div}\tilde{R}^u\|_{H^4}\\
&\lesssim\|\partial_tu_{q+1}+(-\Delta)^\alpha u_{q+1}+{\rm div}(u_{q+1}\otimes u_{q+1}-B_{q+1}\otimes B_{q+1})\|_{L^\infty(\frac{T}{3},\frac{2T}{3};H^3)}\\
&\lesssim\|\partial_tu_{q+1}\|_{L^\infty(\frac{T}{3},\frac{2T}{3};H^3)}+\|u_{q+1}\|_{L^\infty(\frac{T}{3},\frac{2T}{3};H^{\frac{11}{2}})}
+\|(u_{q+1},B_{q+1})\|_{L^2}^{\frac{5}{8}}\|(u_{q+1},B_{q+1})\|_{H^4}^{\frac{11}{8}}\\
&\lesssim\lambda_{q+1}^6+(2\delta_0^{\frac{1}{2}}-\delta_{q+1}^{\frac{1}{2}})^{\frac{5}{8}}\lambda_{q+1}^{\frac{55}{8}}\leq\lambda_{q+1}^7,\\
&\|\mathring{R}_{q+1}^B\|_{H^4}
\lesssim\|\partial_tB_{q+1}+(-\Delta)^\alpha B_{q+1}+{\rm div}(u_{q+1}\otimes B_{q+1}-B_{q+1}\otimes u_{q+1})\\
&\quad\quad\quad\quad\quad\quad\quad+\nabla\times{\rm div}(B_{q+1}\otimes B_{q+1})\|_{L^\infty(\frac{T}{3},\frac{2T}{3};H^2)}\\
&\lesssim\|\partial_tB_{q+1}\|_{L^\infty(\frac{T}{3},\frac{2T}{3};H^2)}+\|B_{q+1}\|_{L^\infty(\frac{T}{3},\frac{2T}{3};H^{\frac{9}{2}})}
+\|(u_{q+1},B_{q+1})\|_{L^2}^{\frac{5}{8}}\|(u_{q+1},B_{q+1})\|_{H^4}^{\frac{11}{8}}\\
&\lesssim\lambda_{q+1}^6+(2\delta_0^{\frac{1}{2}}-\delta_{q+1}^{\frac{1}{2}})^{\frac{5}{8}}\lambda_{q+1}^{\frac{55}{8}}\leq\lambda_{q+1}^7.
\end{align*}}
\end{proof}
\appendix
\section{}\label{sec6}
\begin{proof}[Proof of Lemma \ref{lemma}]
Let $e_1=(1,0,0)$, $e_2=(0,1,0)$, $e_3=(0,0,1)$, and $\Lambda_0=\{\frac{3}{5}e_1\pm \frac{4}{5}e_2,\frac{4}{5}e_1\pm \frac{3}{5}e_3,\frac{3}{5}e_2\pm \frac{4}{5}e_3\}$. The orthonormal bases are given by
\begin{table}[H]
\begin{tabular}{c|c|c}
  $\xi$ & $A_\xi$  & $\xi\times A_\xi$  \\\hline
  $\frac{3}{5}e_1\pm \frac{4}{5}e_2$& $\frac{4}{5}e_1\mp \frac{3}{5}e_2$ & $e_3$\\
  $\frac{4}{5}e_1\pm \frac{3}{5}e_3$& $\frac{3}{5}e_1\mp \frac{4}{5}e_3$ & $e_2$\\
  $\frac{3}{5}e_2\pm \frac{4}{5}e_3$& $\frac{4}{5}e_2\mp \frac{3}{5}e_3$ & $e_1$\\
\end{tabular}
\end{table}
Then we have $\sum_{\xi \in\Lambda_0}\frac{1}{2}\xi\otimes\xi={\rm Id}$. The implicit function theorem implies that there exists a small constant $\delta_0$ such that for all symmetric matrices $R$ with $|R-{\rm Id}|<\delta_0$, it holds
$$R=\sum_{\xi\in \Lambda_0}\gamma^2_{\xi}(R)(\xi\otimes\xi),$$
for some smooth functions $\gamma_\xi$ with $\gamma^2_{\xi}({\rm Id})=\frac{1}{2}$. For more details, see \cite{bdis}.

Similarly, taking
\begin{align*}
\Lambda_1=\{\frac{5}{13}e_1\pm \frac{12}{13}e_2,\frac{12}{13}e_1\pm \frac{5}{13}e_3,\frac{5}{13}e_2\pm \frac{12}{13}e_3\},\\
\Lambda_2=\{\frac{7}{25}e_1\pm \frac{24}{25}e_2,\frac{24}{25}e_1\pm \frac{7}{25}e_3,\frac{7}{25}e_2\pm \frac{24}{25}e_3\},\\
\Lambda_3=\{\frac{8}{17}e_1\pm \frac{15}{17}e_2,\frac{15}{17}e_1\pm \frac{8}{17}e_3,\frac{8}{17}e_2\pm \frac{15}{17}e_3\},
\end{align*}
and repeating the above process, we can conclude that the same conclusion holds for $\Lambda_\alpha$, $\alpha=1,2,3$. Moreover, one has $\Lambda_{\alpha_1}\cap\Lambda_{\alpha_2}=\emptyset$ for $\alpha_1\neq\alpha_2$.
\end{proof}

\end{document}